\documentclass[10pt]{amsart}
\usepackage[utf8]{inputenc}
\usepackage[russian,english]{babel}
\usepackage{amsmath}
\usepackage{amssymb}
\usepackage{amsfonts}
\usepackage{tikz}
\usepackage{appendix}
\usepackage{hyperref}

\def\udcs{519.148} 
\def\mscs{05B99} 
\newtheorem{lemma}{Lemma}
\newtheorem{theorem}{Theorem}

\def\logo{{\bf\huge S\raisebox{0.2ex}{\hspace{0.55ex}\raisebox{0.05ex}e\hspace{-1.65ex}$\bigcirc$}MR}}

\def\semrtop{
  \vbox{
     \noindent\logo\hfill\raisebox{1ex}{ISSN 1813-3304}\par
     \vspace{5mm}
     \begin{otherlanguage}{russian}
     \begin{center}
     {\huge SIBIRSKIE \ \`ELEKTRONNYE} \\[2mm]  
     {\huge MATEMATICHESKIE IZVESTIYA} \\[2mm]
     {\large Siberian Electronic Mathematical Reports} \\[1mm]
     {\LARGE\tt{http://semr.math.nsc.ru}}\\[0.5mm]
     \end{center}
     \vspace{-3mm}
     \noindent
     \begin{tabular}{c}
     \hphantom{aaaaaaaaaaaaaaaaaaaaaaaaaaaaaaaaaaaaaaaaaaaaaaaaaaaaaaaaaaaaaaaaaaaaaa} \\
     \hline\hline
     \end{tabular}
     \vspace{1mm}
     {\flushleft\it Vol. 13, pp. 987--1016 (2016) \hfill{\rm\small UDC \udcs}} 
     \newline {\rm\small DOI~10.17377/semi.2016.13.079}\hfill{\rm\small MSC\ \ \mscs }\par
     \end{otherlanguage}
  }
}

\begin{document}
\renewcommand{\refname}{References}
\renewcommand{\proofname}{Proof.}
\thispagestyle{empty}

\title[Completely regular codes in the infinite hexagonal grid]{Completely regular codes in the infinite hexagonal grid}
\author{{S.~V. Avgustinovich, D.~S. Krotov, A.~Yu. Vasil'eva}}%


\address{Sergey Vladimirovich Avgustinovich
\newline\hphantom{iii} Sobolev Institute of Mathematics,
\newline\hphantom{iii} pr. Koptyuga, 4,
\newline\hphantom{iii} 630090, Novosibirsk, Russia}%
\email{\href{mailto:avgust@math.nsc.ru}{avgust@math.nsc.ru}}%

\address{Denis Stanislavovich Krotov
\newline\hphantom{iii} Sobolev Institute of Mathematics,
\newline\hphantom{iii} pr. Koptyuga, 4,
\newline\hphantom{iii} 630090, Novosibirsk, Russia}%
\email{\href{mailto:krotov@math.nsc.ru}{krotov@math.nsc.ru}}%

\address{Anastasia Yurievna Vasil'eva
\newline\hphantom{iii} Sobolev Institute of Mathematics,
\newline\hphantom{iii} pr. Koptyuga, 4,
\newline\hphantom{iii} 630090, Novosibirsk, Russia}%
\email{\href{mailto:vasilan@math.nsc.ru}{vasilan@math.nsc.ru}}%

\thanks{\sc Avgustinovich, S.V., Krotov, D.S., Vasil'eva, A.Yu.,
Completely regular codes in the infinite hexagonal grid}
\thanks{\copyright \ 2016 Avgustinovich S.V., Krotov D.S., Vasil'eva A.Yu.}
\thanks{\rm The results of the paper were partially presented at the International Conference ``Mal'tsev Meeting'' in November 2013, Novosibirsk, Russia.}
\thanks{\it Received  April, 15, 2016, published  November, 15,  2016.}%

\semrtop \vspace{1cm}
\maketitle {\small
\begin{quote}
\noindent{\sc Abstract. } 
A set $C$ of vertices of a simple graph is called a completely regular code if 
for each $i=0$, $1$, $2$, \ldots and $j = i-1$, $i$, $i+1$, all vertices at distance $i$ 
from $C$ have the same number $s_{ij}$ of neighbors at distance $j$ from $C$.
We characterize the completely regular codes in the infinite hexagonal grid graph.
\medskip

 \end{quote}
}

\newlength{\dk} \dk=1.7em 
\newlength{\dktmp} \dktmp=0.7\dk 
\tikzset{x=1.0\dk,y=1.0\dk,
 cell/.style={circle,draw=black, thin, minimum size=0.77\dk},
 }
\def\ucolor#1#2{ ++(1,0.29) node [cell,fill=#1] {\raisebox{-0.25em}[0.2em][0em]{\makebox[0mm][c]{#2}}} ++(0,-0.29) }
\def\dcolor#1#2{ ++(1,-0.29) node [cell,fill=#1] {\raisebox{-0.25em}[0.2em][0em]{\makebox[0mm][c]{#2}}} ++(0,0.29) }
\def\Ucolor#1#2{ ++(1,0.29) node [cell,very thick,fill=#1] {\raisebox{-0.25em}[0.2em][0em]{\makebox[0mm][c]{#2}}} ++(0,-0.29) }
\def\Dcolor#1#2{ ++(1,-0.29) node [cell,very thick,fill=#1] {\raisebox{-0.25em}[0.2em][0em]{\makebox[0mm][c]{#2}}} ++(0,0.29) }
\def\uq{\ucolor{white}{$\boldsymbol ?$}}
\def\dq{\dcolor{white}{$\boldsymbol ?$}}
\def\uo{\ucolor{white}{$~$}}%
\def\do{\dcolor{white}{$~$}}%

\section{Introduction}\label{sec:intro}

Completely regular codes 
and corresponding vertex colorings 
(we call them dis\-tance-re\-gu\-lar colorings, 
a partial case of perfect colorings or equitable partitions) 
play important role 
in the theory of distance-regular graphs \cite{Brouwer}.
Recall that a distance-regular graph is a regular graph 
such that every singleton is a completely regular code.
The existence of completely regular codes 
in non-distance-regular graphs 
shows that some phenomena taking place in distance-regular graphs
can also happen in other graphs.
For example, 
the weight distribution of 
one completely regular code with respect to another one
satisfies some equations \cite{Kro:struct},
which helps to study such codes 
(as well as other objects with similar algebraic properties)
and to prove the nonexistence for some given parameters.
So, it is natural 
to study completely regular codes 
in different kind of graphs.
However, 
the most interesting cases are transitive graphs, 
including infinite grids.
Some grids were considered in previous papers.
In \cite{Axenovich:2003}, 
\cite{Puz2005.en}, and \cite{Kro:9colors}, 
perfect colorings (equitable partitions) 
of the infinite rectangular grid into two colors,
three colors, and up to nine colors, respectively, were described.
In \cite{Plotnikov:2005}, the parameter matrices of perfect colorings 
of the infinite hexagonal grid into three colors were listed,
but one was missed (it corresponds to a distance-regular coloring and occurs in the current classification).
As was proven in \cite{Puz2004.en} and \cite{Puz:2011en}, for every perfect coloring of the infinite rectangular, triangular, or hexagonal grid,
there exists a periodic perfect coloring with the same parameter matrix. 
In \cite{AVS:2012en} and \cite{Vas:Maltsev14}, 
the intersection arrays of completely regular codes in the infinite rectangular and, 
respectively, triangular grid were characterized.
In \cite{AvgVas:ACCT2012}, an $n$-dimensional rectangular grid were considered 
and it was proven that the maximum number of colors of an irreducible distance-regular coloring is $2n+1$.

In the current paper, we describe the completely regular codes in the hexagonal grid, 
which is an infinite edge-transitive graph of degree $3$.
The set of intersection arrays of the completely regular codes in this graph consists of several infinite series and several ``sporadic'' arrays,
see Theorem~1.
For every feasible intersection array, there are one, two, or infinitely many  nonequivalent completely regular codes.
In the appendices, all possible cases are illustrated by pictures.

 \emph{Remark}. Many codes in our characterization are very symmetrical. 
It is not a surprise that there are connections with some other classes of objects that are related with the symmetries of the plane.
For example, every edge-to-edge tiling (mosaic) of the plane by regular triangles and hexagons can be considered as a partition of the vertices of the infinite hexagonal grid
into two parts (a triangle corresponds to a node, while a hexagon covers six nodes). 
It is notable that if we consider so-called $1$- and $2$-uniform tilings by triangles and hexagons, see e.g. \cite[Sect.~2.2]{GruShe}
(there are $9$ such tilings, named 
$(3^6)$, 
$(6^3)$, 
$(3^4{\cdot} 6)$, 
$(3{\cdot} 6{\cdot} 3{\cdot} 6)$, 
$(3^6;3^4{\cdot} 6)_1$, 
$(3^6;3^4{\cdot} 6)_2$,   
$(3^6;3^2{\cdot} 6^2)$,  
$(3^4{\cdot} 6;3^2{\cdot} 6^2)$,  
$(3^2{\cdot} 6^2;3{\cdot} 6{\cdot} 3{\cdot} 6)$ in \cite[pp. 63--67]{GruShe}),
then, with only one exception $(3^6;\linebreak[2]3^4{\cdot} 6)_1$,
one of these two sets is always a completely regular code (the corresponding inter\-sec\-ti\-on arrays are 
$[3]$, 
$[3]$, 
$[21{-}111{-}30]$, 
$[03{-}12]$, 
$[21{-}102{-}21]$, 
$[21{-}102{-}21]$, 
$[21{-}111{-}21]$, 
$[12{-}111{-}21]$, using the notation defined below, 
see also the corresponding figures in the appendices).
Also, the codes with arrays
$[21{-}102{-}12]$, 
$[21{-}201{-}12]$, 
$[21{-}102{-}111{-}201{-}30]$
correspond to $3$-uniform tilings with parameters 
$(3^4{\cdot}6; 3^6; 3^6)$, 
$(3{\cdot} 6{\cdot} 3{\cdot} 6; 3^4{\cdot}6; 3^6)$, 
$(3^4{\cdot}6; 3^6; 3^6)$, respectively, in terms of \cite{GruShe}.

\section{Definitions and the main result}\label{sec:main}
A vertex partition $(V_1, \ldots, V_k)$ of a graph $G$
is called a {\em perfect coloring} (equitable partition, regular partition, partition design)
if for every $i,j\in\{1,\ldots,k\}$ there is a number $a_{ij}$ such that
every vertex from  $V_i$ has exactly $a_{ij}$ neighbors from $V_j$.
The matrix $A = (a_{ij})$ will be called the {\em parameter matrix} of the coloring.
A perfect coloring $(V_1,\ldots, V_k)$ is called {\em distance regular} if its
parameter matrix is tridiagonal (equivalently, if the perfect coloring $(V_1,\ldots, V_k)$
is the distance coloring with respect to $V_1$; the set $V_1$ is known to be a {\em distance regular code} in this case).
We study the distance regular colorings of the infinite cubic graph of the hexagonal grid.

We define the hexagonal grid as the Cayley graph 
of the group generated by three elements $x$, $y$, and $z$ 
and defined by the identities $xx=yy=zz=xyzxyz=o$, where 
$o$ stands for the identity element. That is, 
the vertices (nodes) are the elements of the group; 
two elements $u$ and $v$ are adjacent if and only if one of $u=vx$, $u=vy$,  $u=vz$ holds.
In most of the figures, $x$ will correspond to the horizontal direction, $y$ to the right-diagonal, and $z$ to the left-diagonal direction:
$$\input{000} $$
In the other figures, the grid is pictured rotated by 90 degrees; in those cases, we will not refer to the nodes as elements of the group.

A tridiagonal parameter matrix $(a_{ij})_{i,j=1}^k$ of a distance regular coloring will be written 
in the form of the array
$$[a_{11}a_{12}{-}
a_{21}a_{22}a_{23}{-}
a_{32}a_{33}a_{34}{-}\ldots{-}
a_{k\,k{-}1}a_{kk}],$$
which is called the {\em intersection array} of the corresponding completely regular code (we use a non-standard but convenient way to list the elements of the array).

\begin{theorem}
Up to the central symmetry, the parameter $k\times k$ matrices of nontrivial ($k>1$) 
distance regular colorings of the infinite hexagonal grid are divided into:\\
\mbox{}\,\,\,\,$\bullet$\,\,%
   $5$ matrices $[03{-}30]$, $[03{-}12]$, $[12{-}21]$, $[12{-}12]$, $[21{-}12]$, $k=2$;\\
\mbox{}\,\,\,\,$\bullet$\,\,%
  $6$ infinite classes \\
\mbox{}\,\,\,\,\,\,\,%
  $[12{-}111{-}...{-}111{-}12]$, 
  $[12{-}111{-}...{-}111{-}21]$,  
  $[21{-}111{-}...{-}111{-}12]$,  $k=3,4,5,...$;\\
\mbox{}\,\,\,\,\,\,\,%
 $[12{-}201{-}102{-}201{-}...{-}201{-}12]$,  $k=3,5,7,...$; \\
\mbox{}\,\,\,\,\,\,\,%
 $[12{-}201{-}102{-}201{-}...{-}102{-}21]$,
 $[21{-}102{-}201{-}102{-}...{-}201{-}12]$, $k=4,6,8,...$;\\
\mbox{}\,\,\,\,$\bullet$\,\,%
 $10$ ``sporadic'' matrices
  $[03{-}102{-}30]$, $[03{-}111{-}12]$,
  $[12{-}102{-}12]$, $[21{-}102{-}12]$, \\ 
\mbox{}\,\,\,\,\,\,\,\,\,\,\,\,%
  $[03{-}102{-}102{-}30]$, $[03{-}102{-}201{-}30]$,
  $[03{-}111{-}111{-}30]$, $[12{-}102{-}111{-}21]$, \\ 
\mbox{}\,\,\,\,\,\,\,\,\,\,\,\,%
  $[03{-}102{-}102{-}201{-}30]$, $[03{-}102{-}111{-}201{-}12]$.
 
Each of the matrices $[03{-}111{-}12]$, $[12{-}201{-}12]$ corresponds to two nonequivalent colorings; 
each of $[03{-}12]$, $[12{-}21]$, $[21{-}12]$, $[12{-}111{-}21]$, $[03{-}102{-}201{-}30]$, to infinite number of colorings;
every other matrix corresponds to one colorings, up to equivalence.
\end{theorem}

All possible distance regular colorings of the infinite hexagonal grid are described in the proof of the theorem;
the summary can be found in the appendices.

\section{A proof of the theorem}\label{sec:proof}
We will consider arrays $[a_0b_0{-}c_1a_1b_1{-}c_2a_2b_2{-}...]$,
considering possible prefixes in the lexicographic order.
We divide the proof into the following cases:\\
\def\qwerty{}
\newcommand\zacherk[1]{\raisebox{0pt}[0pt][0pt]{\raisebox{0.40em}{\underline{\raisebox{0pt}[0pt][0pt]{\raisebox{-0.40em}{#1}}}}}}
\begin{minipage}[t]{1.5in}
 {\qwerty${[03{-}102{-}102{-}...]}$},\\
{\qwerty$[03{-}102{-}111{-}...]$},\\
\zacherk{$[03{-}102{-}12]$},\\
\zacherk{$[03{-}102{-}201{-}1...]$},\\
{\qwerty$[03{-}102{-}201{-}201{-}...]$},\\
\zacherk{$[03{-}102{-}201{-}21]$},\\
{\qwerty$[03{-}102{-}201{-}30]$},\\
\zacherk{${[03{-}102{-}21]}$},\\
{\qwerty${[03{-}102{-}30]}$},
\end{minipage}
\begin{minipage}[t]{1.2in}
{\qwerty${[03{-}111{-}...]}$},\\
{\qwerty${[03{-}12]}$},\\
{\qwerty${[03{-}201{-}...]}$},\\
\zacherk{${[03{-}21]}$},\\
{\qwerty${[03{-}30]}$},\\
{\qwerty$[12{-}102{-}...]$},\\
\zacherk{${[12{-}111{-}102{-}...]}$},\\
{\qwerty${[12{-}111{-}111{-}...]}$},\\
{\qwerty${[12{-}111{-}201{-}...]}$},\\
\end{minipage}
\begin{minipage}[t]{1.2in}
{\qwerty${[12{-}111{-}21]}$},\\
\zacherk{${[12{-}111{-}30]}$},\\
{\qwerty${[12{-}12]}$},\\
{\qwerty${[12{-}201{-}102{-}...]}$},\\
\zacherk{${[12{-}201{-}111{-}...]}$},\\
{\qwerty${[12{-}201{-}12]}$},\\
\zacherk{${[12{-}201{-}201{-}...]}$},\\
\zacherk{${[12{-}201{-}21]}$},\\
\zacherk{${[12{-}201{-}30]}$},\\
\end{minipage}
\begin{minipage}[t]{1.0in}
{\qwerty${[12{-}21]}$},\\
\zacherk{${[12{-}30]}$},\\
{\qwerty${[21{-}102{-}...]}$},\\
{\qwerty${[21{-}111{-}...]}$},\\
{\qwerty${[21{-}12]}$},\\
{\qwerty${[21{-}201{-}...]}$},\\
{\qwerty${[21{-}21]}$},\\
{\qwerty${[21{-}30]}$}.
\end{minipage}

We will consider the cases below, step by step, dividing into subcases if necessary.
Some of the cases will lead to a contradiction 
(the corresponding prefixes are strikeout in the table above); 
the other result in one or more colorings.
Each prefix from this list will occur once in bold, 
in the place where we start to consider the corresponding case.

Before we start to consider the cases, 
we describe the notation we use in the figures.
The colors are indicated by the numbers, starting from $0$. 
The subindices that accompany the color numbers 
indicate the ``argument steps'', i.e., 
the order in which the nodes are considered 
to check that the situation 
pictured in the figure takes place; 
the color of a cell with such subindex 
can be uniquely determined 
from the colors of the cells with smaller 
subindices by trivial arguments and, 
sometimes, assumption up to a symmetry.
To illustrate this notation, we consider the proof of the following lemma in details.

\begin{lemma}\label{l:111-102} The fragments $[...{-}111{-}102{-}...]$
and $[...{-}201{-}111{-}...]$ cannot occur in a feasible array.
\end{lemma}

\begin{proof} By the example of $c_2=a_2=b_2=c_3=1$, $b_3=2$,
we see a contradiction:
$$\def\db#1{\dcolor{yellow!60!white}{$1_{\mathrm{#1}}$}}%
\def\0{++(1,0)}%
\def\uc#1{\ucolor{orange!65!white}{$2_{\mathrm{#1}}$}}%
\def\dc#1{\dcolor{orange!65!white}{$2_{\mathrm{#1}}$}}%
\def\ud#1{\ucolor{red!85!brown!70!white}{$3_{\mathrm{#1}}$}}%
\def\dd#1{\dcolor{red!85!brown!70!white}{$3_{\mathrm{#1}}$}}%
\def\ue#1{\ucolor{brown!80!black}{$4_{\mathrm{#1}}$}}%
\def\de#1{\dcolor{brown!80!black}{$4_{\mathrm{#1}}$}}%
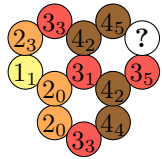
\begin{tikzpicture}[scale=0.65, rotate=60]
\draw (0, 3.464)\0\0  \db1\uc3\dd3;
\draw (0, 1.732)\0\dc0\uc0\dd1\ue2\de5;
\draw (0,-0.000)\0\ud3\de4\ue2\dd5\uq;
\end{tikzpicture}$$

As an example illustrating the notation,
let us describe the steps shown at the figure in details.

Step 0. As $a_2=1$, there are two neighbor nodes of color $2$. 
We consider such two nodes and marked them by ``$2_0$'', 
where the subindex $0$ indicates the step number.

$$
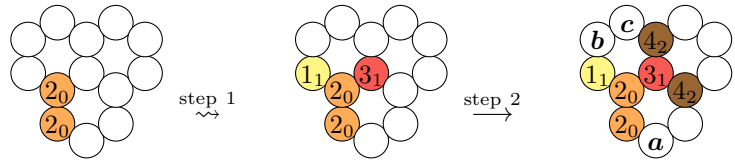
\raisebox{-10mm}{\def\uo{\ucolor{white}{$\ $}}
\def\do{\dcolor{white}{$\ $}}
\def\0{++(1,0)}%
\def\uc#1{\ucolor{orange!65!white}{$2_{\mathrm{#1}}$}}%
\def\dc#1{\dcolor{orange!65!white}{$2_{\mathrm{#1}}$}}%
\begin{tikzpicture}[scale=0.65, rotate=60]
\draw (0, 3.464)\0\0  \do \uo\do; 
\draw (0, 1.732)\0\dc0\uc0\do\uo\do ;
\draw (0,-0.000)\0\uo \do\uo\do\uo;
\end{tikzpicture}}
\stackrel{\mathrm{step\ 1}}\rightsquigarrow
\raisebox{-10mm}{\def\uo{\ucolor{white}{$\ $}}
\def\do{\dcolor{white}{$\ $}}
\def\db#1{\dcolor{yellow!60!white}{$1_{\mathrm{#1}}$}}%
\def\0{++(1,0)}%
\def\uc#1{\ucolor{orange!65!white}{$2_{\mathrm{#1}}$}}%
\def\dc#1{\dcolor{orange!65!white}{$2_{\mathrm{#1}}$}}%
\def\dd#1{\dcolor{red!85!brown!70!white}{$3_{\mathrm{#1}}$}}%
\begin{tikzpicture}[scale=0.65, rotate=60]
\draw (0, 3.464)\0\0  \db1\uo\do;
\draw (0, 1.732)\0\dc0\uc0\dd1\uo\do;
\draw (0,-0.000)\0\uo\do\uo\do\uo;
\end{tikzpicture}}
\stackrel{\mathrm{step\ 2}}\longrightarrow
\raisebox{-10mm}{\def\uo{\ucolor{white}{$\ $}}
\def\do{\dcolor{white}{$\ $}}
\def\db#1{\dcolor{yellow!60!white}{$1_{\mathrm{#1}}$}}%
\def\0{++(1,0)}%
\def\uc#1{\ucolor{orange!65!white}{$2_{\mathrm{#1}}$}}%
\def\dc#1{\dcolor{orange!65!white}{$2_{\mathrm{#1}}$}}%
\def\ud#1{\ucolor{red!85!brown!70!white}{$3_{\mathrm{#1}}$}}%
\def\dd#1{\dcolor{red!85!brown!70!white}{$3_{\mathrm{#1}}$}}%
\def\ue#1{\ucolor{brown!80!black}{$4_{\mathrm{#1}}$}}%
\def\de#1{\dcolor{brown!80!black}{$4_{\mathrm{#1}}$}}%
\begin{tikzpicture}[scale=0.65, rotate=60]
\draw (0, 3.464)\0\0  \db1{\ucolor{white}{$\boldsymbol b$}}{\dcolor{white}{$\boldsymbol c$}};
\draw (0, 1.732)\0\dc0\uc0\dd1\ue2\do;
\draw (0,-0.000)\0{\ucolor{white}{$\boldsymbol a$}}\do\ue2\do\uo;
\end{tikzpicture}}
$$

Step 1. Now consider one of these two color-$2$ nodes, e.g., the upper one. 
Since $c_2=b_2=1$, its two unmarked neighbors have colors $1$ and $3$.
Without loss of generality (up to a symmetry), 
we mark the left neighbor by ``$1_1$'' and  the right neighbor by ``$3_1$'',
where the subindex $1$ indicates the step number.

Step 2. Since $b_3=2$, the node marked by ``$3_1$'' has two color-$4$ neighbors.
We mark them by ``$4_2$'',
where the subindex $2$ indicates the step number.

Step 3. Now consider the node {$\boldsymbol a$} (see the last figure).
It has color $1$ or $3$. But $1$ is not possible because {$\boldsymbol a$}
is at distance $2$ from a color-$4$ node.
So, we mark the node {$\boldsymbol a$} by ``$3_3$'',
where the subindex $3$ indicates the step number.

Independently, we consider the nodes {$\boldsymbol b$} and {$\boldsymbol c$}.
They lay on a length-$3$ path from a color-$1$ node to a color-$4$ node;
so, they must have colors $2$ and $3$, respectively. 
We mark them by these numbers 
with the subindex $3$, indicating the step number.

$$\stackrel{\mathrm{step\ 3}}\longrightarrow
\raisebox{-10mm}{\def\uo{\ucolor{white}{$\ $}}
\def\do{\dcolor{white}{$\ $}}
\def\db#1{\dcolor{yellow!60!white}{$1_{\mathrm{#1}}$}}%
\def\0{++(1,0)}%
\def\uc#1{\ucolor{orange!65!white}{$2_{\mathrm{#1}}$}}%
\def\dc#1{\dcolor{orange!65!white}{$2_{\mathrm{#1}}$}}%
\def\ud#1{\ucolor{red!85!brown!70!white}{$3_{\mathrm{#1}}$}}%
\def\dd#1{\dcolor{red!85!brown!70!white}{$3_{\mathrm{#1}}$}}%
\def\ue#1{\ucolor{brown!80!black}{$4_{\mathrm{#1}}$}}%
\def\de#1{\dcolor{brown!80!black}{$4_{\mathrm{#1}}$}}%
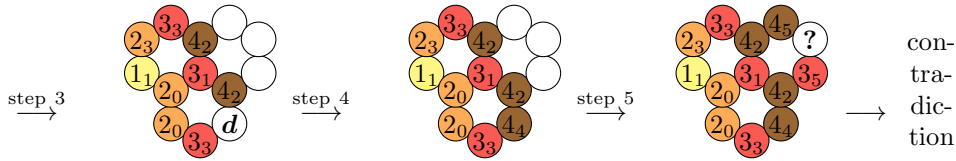
\begin{tikzpicture}[scale=0.65, rotate=60]
\draw (0, 3.464)\0\0  \db1\uc3\dd3;
\draw (0, 1.732)\0\dc0\uc0\dd1\ue2\do;
\draw (0,-0.000)\0\ud3{\dcolor{white}{$\boldsymbol d$}}\ue2\do\uo;
\end{tikzpicture}}
\stackrel{\mathrm{step\ 4}}\longrightarrow
\raisebox{-10mm}{\def\uo{\ucolor{white}{$\ $}}
\def\do{\dcolor{white}{$\ $}}
\def\db#1{\dcolor{yellow!60!white}{$1_{\mathrm{#1}}$}}%
\def\0{++(1,0)}%
\def\uc#1{\ucolor{orange!65!white}{$2_{\mathrm{#1}}$}}%
\def\dc#1{\dcolor{orange!65!white}{$2_{\mathrm{#1}}$}}%
\def\ud#1{\ucolor{red!85!brown!70!white}{$3_{\mathrm{#1}}$}}%
\def\dd#1{\dcolor{red!85!brown!70!white}{$3_{\mathrm{#1}}$}}%
\def\ue#1{\ucolor{brown!80!black}{$4_{\mathrm{#1}}$}}%
\def\de#1{\dcolor{brown!80!black}{$4_{\mathrm{#1}}$}}%
\begin{tikzpicture}[scale=0.65, rotate=60]
\draw (0, 3.464)\0\0  \db1\uc3\dd3;
\draw (0, 1.732)\0\dc0\uc0\dd1\ue2\do;
\draw (0,-0.000)\0\ud3\de4\ue2\do\uo;
\end{tikzpicture}}
\stackrel{\mathrm{step\ 5}}\longrightarrow
\raisebox{-10mm}{\def\db#1{\dcolor{yellow!60!white}{$1_{\mathrm{#1}}$}}%
\def\0{++(1,0)}%
\def\uc#1{\ucolor{orange!65!white}{$2_{\mathrm{#1}}$}}%
\def\dc#1{\dcolor{orange!65!white}{$2_{\mathrm{#1}}$}}%
\def\ud#1{\ucolor{red!85!brown!70!white}{$3_{\mathrm{#1}}$}}%
\def\dd#1{\dcolor{red!85!brown!70!white}{$3_{\mathrm{#1}}$}}%
\def\ue#1{\ucolor{brown!80!black}{$4_{\mathrm{#1}}$}}%
\def\de#1{\dcolor{brown!80!black}{$4_{\mathrm{#1}}$}}%
\begin{tikzpicture}[scale=0.65, rotate=60]
\draw (0, 3.464)\0\0  \db1\uc3\dd3;
\draw (0, 1.732)\0\dc0\uc0\dd1\ue2\de5;
\draw (0,-0.000)\0\ud3\de4\ue2\dd5\uq;
\end{tikzpicture}}
\longrightarrow
\raisebox{3mm}{%
\begin{tabular}{c}
 \mbox{con-}\\ \mbox{tra-}\\ \mbox{dic-}\\ \mbox{tion}\\ 
\end{tabular}}
$$

Step 4. At this step, we mark the node {$\boldsymbol d$} by ``$4_4$'' since it has a color-$3$ neighbor and
$b_3=2$.

Step 5. After the previous step, 
we see that a color-$4$ node has two color-$3$ neighbors. 
Another color-$4$ node has a color-$4$ neighbor. 
By the definition of a perfect coloring, both assertions must be satisfied by any 
color-$4$ node (that is, $c_4 = 2$ and $a_4=1$). 
Using this fact, we mark two nodes by ``$3_5$'' and ``$4_5$'', 
as shown at the figure.

Now, both unmarked neighbors of the node ``$4_5$'' must be colored by $3$.
But the one indicated by ``?'' cannot have this color because it has a color-$3$ neighbor and $a_3=0$.
A contradiction.
\end{proof}

We start with {\boldmath${[03{-}102{-}102{-}...]}$}.
We assume that the origin is colored by $0$.
As we see from the figure below, the color
of any node at distance at most $3$ from the origin
is uniquely defined. 

$$\def\0{++(1,0)}%
\def\uo{\ucolor{white}{$\ $}}%
\def\do{\dcolor{white}{$\ $}}%
\def\dxa{\dcolor{white}{$\boldsymbol a$}}%
\def\uxa{\ucolor{white}{$\boldsymbol a$}}%
\def\dxb{\dcolor{white}{$\boldsymbol b$}}%
\def\uxb{\ucolor{white}{$\boldsymbol b$}}%
\def\ua#1{\ucolor{yellow!20!white}{$0_{\mathrm #1}$}}%
\def\da#1{\dcolor{yellow!20!white}{$0_{\mathrm #1}$}}%
\def\ub#1{\ucolor{yellow!60!white}{$1_{\mathrm #1}$}}%
\def\db#1{\dcolor{yellow!60!white}{$1_{\mathrm #1}$}}%
\def\uc#1{\ucolor{orange!65!white}{$2_{\mathrm #1}$}}%
\def\dc#1{\dcolor{orange!65!white}{$2_{\mathrm #1}$}}%
\def\ud#1{\ucolor{red!85!brown!70!white}{$3_{\mathrm #1}$}}%
\def\dd#1{\dcolor{red!85!brown!70!white}{$3_{\mathrm #1}$}}%
\def\ue#1{\ucolor{brown!80!black}{$4_{\mathrm #1}$}}%
\def\de#1{\dcolor{brown!80!black}{$4_{\mathrm #1}$}}%
\def\ut{ ++(1, 0.29) node  {\raisebox{-0.25em}[0.2em][0em]{\makebox[0mm][c]{$\cdots$}}}  ++(0,-0.29) }%
\def\dt{ ++(1,-0.29) node  {\raisebox{-0.25em}[0.2em][0em]{\makebox[0mm][c]{$\cdots$}}}  ++(0, 0.29) }%
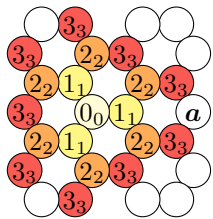
\begin{tikzpicture}[scale=0.65, rotate=-90]
\draw (0,-0.000)\do \uo \dd3\uxa\dd3\uo \do ;
\draw (0,-1.732)\uo \dd3\uc2\db1\uc2\dd3\uo ;
\draw (0,-3.464)\dd3\uc2\db1\ua0\db1\uc2\dd3;
\draw (0,-5.192)\uo \dd3\uc2\dd3\uc2\dd3\uo ;
\end{tikzpicture}$$
The node $\boldsymbol a$ can be colored by either $2$ or $4$.
Each of the cases leads to a unique coloring, up to a symmetry: 
$$\input{03-102-102a}$$
(a period in the subindex denotes that the color of the corresponding cell was determined in the previous figure).
The corresponding arrays are $[03{-}102{-}102{-}30]$ 
and $[03{-}102{-}102{-}201{-}30]$, respectively.

The next case is {\boldmath$[03{-}102{-}111{-}...]$}.
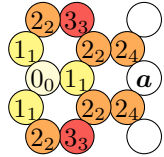
$$\def\0{++(1,0)}%
\def\uo{\ucolor{white}{$\ $}}%
\def\do{\dcolor{white}{$\ $}}%
\def\dxa{\dcolor{white}{$\boldsymbol a$}}%
\def\uxa{\ucolor{white}{$\boldsymbol a$}}%
\def\dxb{\dcolor{white}{$\boldsymbol b$}}%
\def\uxb{\ucolor{white}{$\boldsymbol b$}}%
\def\ua#1{\ucolor{yellow!20!white}{$0_{\mathrm #1}$}}%
\def\da#1{\dcolor{yellow!20!white}{$0_{\mathrm #1}$}}%
\def\ub#1{\ucolor{yellow!60!white}{$1_{\mathrm #1}$}}%
\def\db#1{\dcolor{yellow!60!white}{$1_{\mathrm #1}$}}%
\def\uc#1{\ucolor{orange!65!white}{$2_{\mathrm #1}$}}%
\def\dc#1{\dcolor{orange!65!white}{$2_{\mathrm #1}$}}%
\def\ud#1{\ucolor{red!85!brown!70!white}{$3_{\mathrm #1}$}}%
\def\dd#1{\dcolor{red!85!brown!70!white}{$3_{\mathrm #1}$}}%
\def\ue#1{\ucolor{brown!80!black}{$4_{\mathrm #1}$}}%
\def\de#1{\dcolor{brown!80!black}{$4_{\mathrm #1}$}}%
\def\ut{ ++(1, 0.29) node  {\raisebox{-0.25em}[0.2em][0em]{\makebox[0mm][c]{$\cdots$}}}  ++(0,-0.29) }%
\def\dt{ ++(1,-0.29) node  {\raisebox{-0.25em}[0.2em][0em]{\makebox[0mm][c]{$\cdots$}}}  ++(0, 0.29) }
\begin{tikzpicture}[scale=0.65, rotate=-90]
\draw (0,-3.464)\uo \dc4\uxa\dc4\uo ;
\draw (0,-5.192)\dd3\uc2\db1\uc2\dd3;
\draw (0,-6.928)\uc2\db1\ua0\db1\uc2;
\end{tikzpicture}$$
The node $\boldsymbol a$ can be colored by either $3$ or $1$.
The first case leads to a contradiction; in the second case, we derive that the node $xyxzx$ has color $0$:
$$
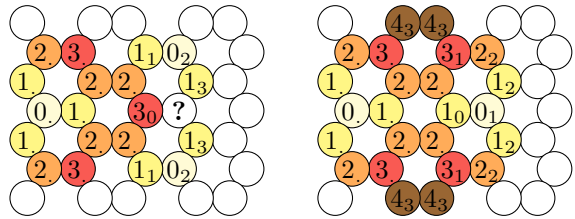
\def\0{++(1,0)}%
\def\uo{\ucolor{white}{$\ $}}%
\def\do{\dcolor{white}{$\ $}}%
\def\dxa{\dcolor{white}{$\boldsymbol a$}}%
\def\uxa{\ucolor{white}{$\boldsymbol a$}}%
\def\dxb{\dcolor{white}{$\boldsymbol b$}}%
\def\uxb{\ucolor{white}{$\boldsymbol b$}}%
\def\ua#1{\ucolor{yellow!20!white}{$0_{\mathrm #1}$}}%
\def\da#1{\dcolor{yellow!20!white}{$0_{\mathrm #1}$}}%
\def\ub#1{\ucolor{yellow!60!white}{$1_{\mathrm #1}$}}%
\def\db#1{\dcolor{yellow!60!white}{$1_{\mathrm #1}$}}%
\def\uc#1{\ucolor{orange!65!white}{$2_{\mathrm #1}$}}%
\def\dc#1{\dcolor{orange!65!white}{$2_{\mathrm #1}$}}%
\def\ud#1{\ucolor{red!85!brown!70!white}{$3_{\mathrm #1}$}}%
\def\dd#1{\dcolor{red!85!brown!70!white}{$3_{\mathrm #1}$}}%
\def\ue#1{\ucolor{brown!80!black}{$4_{\mathrm #1}$}}%
\def\de#1{\dcolor{brown!80!black}{$4_{\mathrm #1}$}}%
\def\ut{ ++(1, 0.29) node  {\raisebox{-0.25em}[0.2em][0em]{\makebox[0mm][c]{$\cdots$}}}  ++(0,-0.29) }%
\def\dt{ ++(1,-0.29) node  {\raisebox{-0.25em}[0.2em][0em]{\makebox[0mm][c]{$\cdots$}}}  ++(0, 0.29) }
\begin{tikzpicture}[scale=0.65, rotate=-90]
\draw (0,-0.000)\do\uo \do \uo \do \uo \do;
\draw (0,-1.732)\uo\da2\ub3\dq \ub3\da2\uo;
\draw (0,-3.464)\do\ub1\dc.\ud0\dc.\ub1\do;
\draw (0,-5.192)\uo\dd.\uc.\db.\uc.\dd.\uo;
\draw (0,-6.928)\do\uc.\db.\ua.\db.\uc.\do;
\end{tikzpicture}
\qquad
\def\0{++(1,0)}%
\def\uo{\ucolor{white}{$\ $}}%
\def\do{\dcolor{white}{$\ $}}%
\def\dxa{\dcolor{white}{$\boldsymbol a$}}%
\def\uxa{\ucolor{white}{$\boldsymbol a$}}%
\def\dxb{\dcolor{white}{$\boldsymbol b$}}%
\def\uxb{\ucolor{white}{$\boldsymbol b$}}%
\def\ua#1{\ucolor{yellow!20!white}{$0_{\mathrm #1}$}}%
\def\da#1{\dcolor{yellow!20!white}{$0_{\mathrm #1}$}}%
\def\ub#1{\ucolor{yellow!60!white}{$1_{\mathrm #1}$}}%
\def\db#1{\dcolor{yellow!60!white}{$1_{\mathrm #1}$}}%
\def\uc#1{\ucolor{orange!65!white}{$2_{\mathrm #1}$}}%
\def\dc#1{\dcolor{orange!65!white}{$2_{\mathrm #1}$}}%
\def\ud#1{\ucolor{red!85!brown!70!white}{$3_{\mathrm #1}$}}%
\def\dd#1{\dcolor{red!85!brown!70!white}{$3_{\mathrm #1}$}}%
\def\ue#1{\ucolor{brown!80!black}{$4_{\mathrm #1}$}}%
\def\de#1{\dcolor{brown!80!black}{$4_{\mathrm #1}$}}%
\def\ut{ ++(1, 0.29) node  {\raisebox{-0.25em}[0.2em][0em]{\makebox[0mm][c]{$\cdots$}}}  ++(0,-0.29) }%
\def\dt{ ++(1,-0.29) node  {\raisebox{-0.25em}[0.2em][0em]{\makebox[0mm][c]{$\cdots$}}}  ++(0, 0.29) }
\begin{tikzpicture}[scale=0.65, rotate=-90]
\draw (0,-0.000)\do \uo \do \uo \do \uo \do ;
\draw (0,-1.732)\uo \dc2\ub2\da1\ub2\dc2\uo ;
\draw (0,-3.464)\de3\ud1\dc.\ub0\dc.\ud1\de3;
\draw (0,-5.192)\ue3\dd.\uc.\db.\uc.\dd.\ue3;
\draw (0,-6.928)\do \uc.\db.\ua.\db.\uc.\do ;
\end{tikzpicture}
$$
Similarly, the nodes $yxyzy$ and $zxzyz$ have the same color and,
for every node $a$ of color $0$, the nodes $axyxzx$, $ayxyzy$, and $azxzyz$
have color $0$ too. Following the arguments as at the picture above, we see that
the nodes of the other colors are colored uniquely.
We get a unique coloring; the array is $[03{-}102{-}111{-}201{-}12]$:
$$\input{03-102-111c}$$

The cases {\boldmath$[03{-}102{-}12]$} and {\boldmath$[03{-}102{-}201{-}1...]$} lead to contradictions:
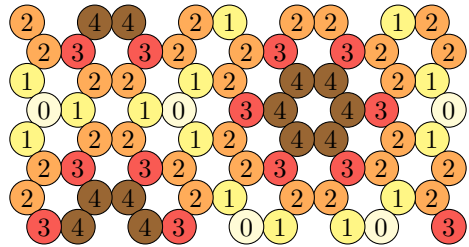
$$\def\0{++(1,0)}%
\def\uo{\ucolor{white}{$\ $}}%
\def\do{\dcolor{white}{$\ $}}%
\def\dxa{\dcolor{white}{$\boldsymbol a$}}%
\def\uxa{\ucolor{white}{$\boldsymbol a$}}%
\def\dxb{\dcolor{white}{$\boldsymbol b$}}%
\def\uxb{\ucolor{white}{$\boldsymbol b$}}%
\def\ua#1{\ucolor{yellow!20!white}{$0_{\mathrm #1}$}}%
\def\da#1{\dcolor{yellow!20!white}{$0_{\mathrm #1}$}}%
\def\ub#1{\ucolor{yellow!60!white}{$1_{\mathrm #1}$}}%
\def\db#1{\dcolor{yellow!60!white}{$1_{\mathrm #1}$}}%
\def\uc#1{\ucolor{orange!65!white}{$2_{\mathrm #1}$}}%
\def\dc#1{\dcolor{orange!65!white}{$2_{\mathrm #1}$}}%
\def\ud#1{\ucolor{red!85!brown!70!white}{$3_{\mathrm #1}$}}%
\def\dd#1{\dcolor{red!85!brown!70!white}{$3_{\mathrm #1}$}}%
\def\ue#1{\ucolor{brown!80!black}{$4_{\mathrm #1}$}}%
\def\de#1{\dcolor{brown!80!black}{$4_{\mathrm #1}$}}%
\def\ut{ ++(1, 0.29) node  {\raisebox{-0.25em}[0.2em][0em]{\makebox[0mm][c]{$\cdots$}}}  ++(0,-0.29) }%
\def\dt{ ++(1,-0.29) node  {\raisebox{-0.25em}[0.2em][0em]{\makebox[0mm][c]{$\cdots$}}}  ++(0, 0.29) }
\begin{tikzpicture}[scale=0.65, rotate=-30]
\draw (3,-3.464)            \db4\uc5\dc6\uq \db9;
\draw (0,-5.192)\do \uc3\dc2\uc3\dc2\uc3\db7\ua8;
\draw (0,-6.928)\uo \dc2\ub1\da0\ub1\dc2\uo ;
\end{tikzpicture}
 \qquad \def\0{++(1,0)}%
\def\uo{\ucolor{white}{$\ $}}%
\def\do{\dcolor{white}{$\ $}}%
\def\dxa{\dcolor{white}{$\boldsymbol a$}}%
\def\uxa{\ucolor{white}{$\boldsymbol a$}}%
\def\dxb{\dcolor{white}{$\boldsymbol b$}}%
\def\uxb{\ucolor{white}{$\boldsymbol b$}}%
\def\ua#1{\ucolor{yellow!20!white}{$0_{\mathrm #1}$}}%
\def\da#1{\dcolor{yellow!20!white}{$0_{\mathrm #1}$}}%
\def\ub#1{\ucolor{yellow!60!white}{$1_{\mathrm #1}$}}%
\def\db#1{\dcolor{yellow!60!white}{$1_{\mathrm #1}$}}%
\def\uc#1{\ucolor{orange!65!white}{$2_{\mathrm #1}$}}%
\def\dc#1{\dcolor{orange!65!white}{$2_{\mathrm #1}$}}%
\def\ud#1{\ucolor{red!85!brown!70!white}{$3_{\mathrm #1}$}}%
\def\dd#1{\dcolor{red!85!brown!70!white}{$3_{\mathrm #1}$}}%
\def\ue#1{\ucolor{brown!80!black}{$4_{\mathrm #1}$}}%
\def\de#1{\dcolor{brown!80!black}{$4_{\mathrm #1}$}}%
\def\ut{ ++(1, 0.29) node  {\raisebox{-0.25em}[0.2em][0em]{\makebox[0mm][c]{$\cdots$}}}  ++(0,-0.29) }%
\def\dt{ ++(1,-0.29) node  {\raisebox{-0.25em}[0.2em][0em]{\makebox[0mm][c]{$\cdots$}}}  ++(0, 0.29) }
\begin{tikzpicture}[scale=0.65, rotate=-30]
\draw (1,-3.464)    \da4\ub5\dq;
\draw (0,-5.192)\dc2\ub3\dc2\ud4;
\draw (0,-6.928)\ub1\da0\ub1;
\end{tikzpicture}
$$
(the last subcase starts the series of four subcases with prefix $[03{-}102{-}201{-}...]$).

In the case {\boldmath$[03{-}102{-}201{-}201{-}...]$}, we consider two subcases, 
in accordance with the color of the node $\boldsymbol a$ at the picture below.
$$\def\uua{\ucolor{white}{\boldmath $a$}}%
\def\da#1{\dcolor{yellow!20!white}{$0_#1$}}%
\def\ub#1{\ucolor{yellow!60!white}{$1_#1$}}%
\def\dc#1{\dcolor{orange!65!white}{$2_#1$}}%
\begin{tikzpicture}[scale=0.65, rotate=-90]
\draw (0,-5.192) \dc2\uua\dc2;
\draw (0,-6.928) \ub1\da0\ub1;
\end{tikzpicture} \qquad \def\do{\dcolor{white}{$\ $}}%
\def\ua#1{\ucolor{yellow!20!white}{$0_#1$}}%
\def\da#1{\dcolor{yellow!20!white}{$0_#1$}}%
\def\ub#1{\ucolor{yellow!60!white}{$1_#1$}}%
\def\db#1{\dcolor{yellow!60!white}{$1_#1$}}%
\def\uc#1{\ucolor{orange!65!white}{$2_#1$}}%
\def\dc#1{\dcolor{orange!65!white}{$2_#1$}}%
\def\ud#1{\ucolor{red!85!brown!70!white}{$3_#1$}}%
\def\de#1{\dcolor{brown!80!black}{$4_#1$}}%
\begin{tikzpicture}[scale=0.65, rotate=-90]
\draw (0,-3.464)\dc2\ud2\de1\ud2\dc2;
\draw (0,-5.192)\ub1\dc.\ud0\dc.\ub1;
\draw (0,-6.928)\do \ub.\da.\ub.\do ;
\end{tikzpicture} \qquad \def\do{\dcolor{white}{$\ $}}%
\def\ua#1{\ucolor{yellow!20!white}{$0_#1$}}%
\def\da#1{\dcolor{yellow!20!white}{$0_#1$}}%
\def\ub#1{\ucolor{yellow!60!white}{$1_#1$}}%
\def\db#1{\dcolor{yellow!60!white}{$1_#1$}}%
\def\uc#1{\ucolor{orange!65!white}{$2_#1$}}%
\def\dc#1{\dcolor{orange!65!white}{$2_#1$}}%
\def\ud#1{\ucolor{red!85!brown!70!white}{$3_#1$}}%
\begin{tikzpicture}[scale=0.65, rotate=-90]
\draw (0,-3.464)\dc2\ub2\da1;
\draw (0,-5.192)\ud1\dc.\ub0\dc.\ub1;
\draw (0,-6.928)\do \ub.\da.\ub.\do ;
\end{tikzpicture}$$
We see that if $\boldsymbol a$ is colored by $3$ (that is, there is a hexagon with nodes of colors $0$ and $3$), 
then we obtain the array $[03{-}102{-}201{-}201{-}30]$,
which is symmetrical to $[03{-}102{-}102{-}201{-}30]$, considered above. 
If $\boldsymbol a$ is colored by $1$, then the situation of the first case takes place too, 
as we find another hexagon with nodes of both colors $0$ and $3$.

The case {\boldmath$[03{-}102{-}201{-}21]$} leads to a contradiction. $\def\db#1{\dcolor{yellow!60!white}{$1_#1$}}%
\def\uc#1{\ucolor{orange!65!white}{$2_#1$}}%
\def\dc#1{\dcolor{orange!65!white}{$2_#1$}}%
\def\ud#1{\ucolor{red!85!brown!70!white}{$3_#1$}}%
\def\dd#1{\dcolor{red!85!brown!70!white}{$3_#1$}}%
\begin{tikzpicture}[scale=0.65, rotate=-90]
\draw (0,-5.192) \dc2\uq \db2;
\draw (0,-6.928) \ud1\dd0\uc1;
\end{tikzpicture}$

Any coloring with array {\boldmath$[03{-}102{-}201{-}30]$} 
corresponds to a partition of the node set into balls of radius $1$ colored
as $\def\0{++(1,0)}%
\def\da{\dcolor{yellow!35!white}{$0$}}%
\def\ub{\ucolor{yellow!60!white}{$1$}}%
\begin{tikzpicture}[scale=0.50, rotate=-90]
\dk=\dktmp
\draw (0,-0    )\ub\da\ub;
\draw (0,-1.732)\0\ub\0;
\end{tikzpicture}
$ and $\def\0{++(1,0)}%
\def\ud{\ucolor{red!35!brown!80!white}{$3$}}%
\def\dc{\dcolor{orange!65!white}{$2$}}%
\begin{tikzpicture}[scale=0.50, rotate=-90]
\dk=\dktmp
\draw (0,-0    )\0 \dc;
\draw (0,-1.732)\dc\ud\dc;
\end{tikzpicture}
$. 
Moreover, the nodes at distance $1$ from a ball of first type belong to balls of second type, and vice versa.
The considered array is the first one for which the number of colorings is infinite. 
We will use the following lemma.
\begin{lemma}\label{l:infin}
Assume that we have a perfect coloring $\phi$ of the hexagonal grid such that
the set $X=\{(yz)^i, (yz)^ix \mid i=0,\pm 1,\pm 2,\ldots\}$ of nodes is colored periodically with period $yzyz$;
that is, for every node $v$ from $X$ the colors of $v$ and $yzyzv$ coincide. Then

{\rm (1)} All the coloring is periodic with period $yzyz$ 
{\rm (}i.e., for every node $v$ of the grid the colors of $v$ and $yzyzv$ coincide{\rm )}.

{\rm (2)} In each coset $uX$ of $X$, the coloring can be shifted by $yz$ resulting in a new perfect coloring 
$$ \psi(v) = \left\{ \begin{array}{ll}
\phi(yzv) & \mbox{if $v \in uX$}\\
\phi(v) & \mbox{if $v \not\in uX$}
\end{array}
\right.
$$
with the same parameter matrix. 
In particular, if $\phi$ is not periodic with period $yz$, 
then the number of perfect colorings with the same parameter matrix is infinite.

{\rm (3)}  Every perfect coloring that coincides with $\phi$ on $X$ and has the same parameters is obtained
from $\phi$ by shifting described in {\rm (2)}  on some, may be infinite, number of cosets of $X$.
\end{lemma}
\begin{proof}
(1) Assume that the set $X$ is colored by the colors $\boldsymbol a$, $\boldsymbol b$, $\boldsymbol c$, $\boldsymbol d$, 
as in the figure. 
$$\def\uo{\ucolor{white}{$\ $}}%
\def\do{\dcolor{white}{$\ $}}%
\def\dxa{\dcolor{white}{$\boldsymbol a$}}%
\def\dxb{\dcolor{white}{$\boldsymbol b$}}%
\def\uxc{\ucolor{white}{$\boldsymbol c$}}%
\def\uxd{\ucolor{white}{$\boldsymbol d$}}%
\def\dxe{\dcolor{white}{$\boldsymbol e$}}%
\def\dxf{\dcolor{white}{$\boldsymbol f$}}%
\def\dxg{\dcolor{white}{$\boldsymbol g$}}%
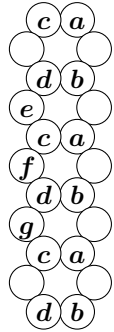
\begin{tikzpicture}[scale=0.65, rotate=-90]
\draw (0,-3.464)\dxa\uo \dxb\uo \dxa\uo \dxb\uo \dxa\uo \dxb;
\draw (0,-5.192)\uxc\do \uxd\dxe\uxc\dxf\uxd\dxg\uxc\do \uxd;
\draw (0,-6.928);
\end{tikzpicture}
$$
In the figure, there are two vertices of color $\boldsymbol c$ that have neighbors of colors 
$\boldsymbol a$, $\boldsymbol e$, $\boldsymbol f$
and $\boldsymbol a$, $\boldsymbol g$, respectively.
We conclude that 
$\boldsymbol g \in \{\boldsymbol e,\boldsymbol f\}$. Similarly, 
$\boldsymbol e \in \{\boldsymbol f,\boldsymbol g\}$.
It follows that $\boldsymbol g =\boldsymbol e$. 
Similarly, all nodes at distance $1$ from $X$ are colored periodically with period $yzyz$.
Straightforwardly, the same is true 
for the nodes at distance $2$ from $X$.
Proceeding by induction on the distance from $X$,
we find that the coloring is periodic 
 with period $yzyz$.
 
(2) It is straightforward that after shifting the colors of the nodes of $uX$, 
the resulting coloring still satisfies the definition of a perfect coloring:
$$\def\uo{\ucolor{white}{$\ $}}%
\def\do{\dcolor{white}{$\ $}}%
\def\dxa{\Dcolor{white}{$\boldsymbol a$}}%
\def\dxb{\Dcolor{white}{$\boldsymbol b$}}%
\def\uxc{\Ucolor{white}{$\boldsymbol c$}}%
\def\uxd{\Ucolor{white}{$\boldsymbol d$}}%
\def\dxe{\dcolor{white}{$\boldsymbol e$}}%
\def\dxf{\dcolor{white}{$\boldsymbol f$}}%
\def\uxg{\ucolor{white}{$\boldsymbol g$}}%
\def\uxh{\ucolor{white}{$\boldsymbol h$}}%
\def\uxi{\ucolor{white}{$\boldsymbol i$}}%
\def\uxj{\ucolor{white}{$\boldsymbol j$}}%
\def\dxk{\dcolor{white}{$\boldsymbol k$}}%
\def\dxl{\dcolor{white}{$\boldsymbol l$}}%
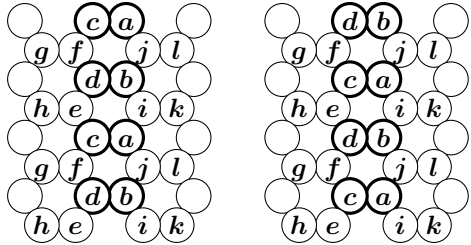
\begin{tikzpicture}[scale=0.65, rotate=-90]
\draw (0,-1.732)\uo \dxl\uo \dxk\uo \dxl\uo \dxk;
\draw (0,-3.464)\dxa\uxj\dxb\uxi\dxa\uxj\dxb\uxi;
\draw (0,-5.192)\uxc\dxf\uxd\dxe\uxc\dxf\uxd\dxe;
\draw (0,-6.928)\do \uxg\do \uxh\do \uxg\do \uxh;
\end{tikzpicture}
\def\dxa{\Dcolor{white}{$\boldsymbol b$}}%
\def\dxb{\Dcolor{white}{$\boldsymbol a$}}%
\def\uxc{\Ucolor{white}{$\boldsymbol d$}}%
\def\uxd{\Ucolor{white}{$\boldsymbol c$}}%
\qquad
\begin{tikzpicture}[scale=0.65, rotate=-90]
\draw (0,-1.732)\uo \dxl\uo \dxk\uo \dxl\uo \dxk;
\draw (0,-3.464)\dxa\uxj\dxb\uxi\dxa\uxj\dxb\uxi;
\draw (0,-5.192)\uxc\dxf\uxd\dxe\uxc\dxf\uxd\dxe;
\draw (0,-6.928)\do \uxg\do \uxh\do \uxg\do \uxh;
\end{tikzpicture}$$
If $zy$ is not a period of the coloring, then such shifts give different colorings. 
The number of the cosets $(xy)^iX$, $i=0,\pm 1, \pm 2, \ldots$,
of $X$ is infinite.
To show that the number of colorings is continuum, 
it is sufficient to prove that the following:

(*) \emph{The number of cosets of $X$ that are not colored periodically
with the period $zy$ is infinite.} 
Seeking a contradiction, assume the contrary, i.e., that there exists 
$l$ such that for all $j\ge l$ the cosets $(x y)^jX$ are colored periodically
with the period $zy$.
Since the number of colors is finite 
and the number of cosets is infinite, 
there are two cosets $(x y)^j X$ and $(x y)^{j'}X$, $j'>j\ge l$, colored identically,
i.e., for some $a$ and $b$ and for all $i=0,\pm 1,\ldots$  we have 
$\phi((x y)^j (yz)^i)= \phi((x y)^{j'} (yz)^i)= a$ 
and $\phi((x y)^j (yz)^i x)= \phi((x y)^{j'} (yz)^i x)= b$.
Since $j'>l$, the coset $(x y)^{j'-1}X$ 
is colored periodically with the period $zy$,
that is, for some $c$ and $d$ and for all $i=0,\pm 1,\ldots$ we have 
$\phi((x y)^{j'-1} (yz)^i)= c$ and $\phi((x y)^{j'-1} (yz)^i x)= d$.
Now we see that every $a$-node from  $(x y)^{j'}X$ 
has two $d$-neighbors and, additionally, one $b$-neighbor.
Since the coloring is perfect, 
the same is true for all $a$-nodes from  $(x y)^{j}X$.
We conclude that $\phi((x y)^{j-1} (yz)^i x)= d$.
Considering the neighborhood of the $d$-nodes, 
we also see that $\phi((x y)^{j-1} (yz)^i x)= c$.
So, the cosets $(x y)^{j-1} X$ and $(x y)^{j'-1}X$ 
are colored identically and periodically 
with period $zy$.
By induction, 
the same is true 
for the cosets $(x y)^{j-s} X$ 
and $(x y)^{j'-s}X$, $s=0,1,2,\ldots$.
Consequently, all the coloring is periodic with  period $zy$, which contradicts our assumption.

(3) is straightforward.
\end{proof}

In the case $[03{-}102{-}201{-}30]$, we consider two subcases:
(1) every hexagon that has a node of color $0$ has a node of color $3$, and vice versa;
(2) there is a hexagon that has only one node of color $0$ or $3$; 
without loss of generality we assume that $o$ has color $0$ while $yxz$ has color $1$.
In both cases, the set $X=\{(yz)^i, (yz)^ix \mid i=0,\pm 1,\pm 2,\ldots\}$ is uniquely colored with period $yzyz$:
$$\mbox{(1)}\def\0{++(1,0)}%
\def\uo{\ucolor{white}{$\ $}}%
\def\do{\dcolor{white}{$\ $}}%
\def\dxa{\dcolor{white}{$\boldsymbol a$}}%
\def\uxa{\ucolor{white}{$\boldsymbol a$}}%
\def\dxb{\dcolor{white}{$\boldsymbol b$}}%
\def\uxb{\ucolor{white}{$\boldsymbol b$}}%
\def\ua#1{\ucolor{yellow!20!white}{$0_{\mathrm #1}$}}%
\def\da#1{\dcolor{yellow!20!white}{$0_{\mathrm #1}$}}%
\def\ub#1{\ucolor{yellow!60!white}{$1_{\mathrm #1}$}}%
\def\db#1{\dcolor{yellow!60!white}{$1_{\mathrm #1}$}}%
\def\uc#1{\ucolor{orange!65!white}{$2_{\mathrm #1}$}}%
\def\dc#1{\dcolor{orange!65!white}{$2_{\mathrm #1}$}}%
\def\ud#1{\ucolor{red!85!brown!70!white}{$3_{\mathrm #1}$}}%
\def\dd#1{\dcolor{red!85!brown!70!white}{$3_{\mathrm #1}$}}%
\def\ue#1{\ucolor{brown!80!black}{$4_{\mathrm #1}$}}%
\def\de#1{\dcolor{brown!80!black}{$4_{\mathrm #1}$}}%
\def\ut{ ++(1, 0.29) node  {\raisebox{-0.25em}[0.2em][0em]{\makebox[0mm][c]{$\cdots$}}}  ++(0,-0.29) }%
\def\dt{ ++(1,-0.29) node  {\raisebox{-0.25em}[0.2em][0em]{\makebox[0mm][c]{$\cdots$}}}  ++(0, 0.29) }
\def\Uc#1{\Ucolor{orange!65!white}{$2_{\mathrm #1}$}}%
\def\Ua#1{\Ucolor{yellow!20!white}{$0_{\mathrm #1}$}}%
\def\Db#1{\Dcolor{yellow!60!white}{$1_{\mathrm #1}$}}%
\def\Dd#1{\Dcolor{red!85!brown!70!white}{$3_{\mathrm #1}$}}%
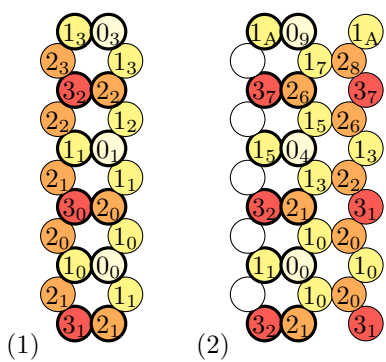
\begin{tikzpicture}[scale=0.65, rotate=90]
\draw (0, 0    ) \Dd1\uc1\Db0\uc0\Dd0\uc1\Db1\uc2\Dd2\uc3\Db3; 
\draw (0,-1.732) \Uc1\db1\Ua0\db0\Uc0\db1\Ua1\db2\Uc2\db3\Ua3;
\end{tikzpicture}\qquad \mbox{(2)}\def\0{++(1,0)}%
\def\uo{\ucolor{white}{$\ $}}%
\def\do{\dcolor{white}{$\ $}}%
\def\dxa{\dcolor{white}{$\boldsymbol a$}}%
\def\uxa{\ucolor{white}{$\boldsymbol a$}}%
\def\dxb{\dcolor{white}{$\boldsymbol b$}}%
\def\uxb{\ucolor{white}{$\boldsymbol b$}}%
\def\ua#1{\ucolor{yellow!20!white}{$0_{\mathrm #1}$}}%
\def\da#1{\dcolor{yellow!20!white}{$0_{\mathrm #1}$}}%
\def\ub#1{\ucolor{yellow!60!white}{$1_{\mathrm #1}$}}%
\def\db#1{\dcolor{yellow!60!white}{$1_{\mathrm #1}$}}%
\def\uc#1{\ucolor{orange!65!white}{$2_{\mathrm #1}$}}%
\def\dc#1{\dcolor{orange!65!white}{$2_{\mathrm #1}$}}%
\def\ud#1{\ucolor{red!85!brown!70!white}{$3_{\mathrm #1}$}}%
\def\dd#1{\dcolor{red!85!brown!70!white}{$3_{\mathrm #1}$}}%
\def\ue#1{\ucolor{brown!80!black}{$4_{\mathrm #1}$}}%
\def\de#1{\dcolor{brown!80!black}{$4_{\mathrm #1}$}}%
\def\ut{ ++(1, 0.29) node  {\raisebox{-0.25em}[0.2em][0em]{\makebox[0mm][c]{$\cdots$}}}  ++(0,-0.29) }%
\def\dt{ ++(1,-0.29) node  {\raisebox{-0.25em}[0.2em][0em]{\makebox[0mm][c]{$\cdots$}}}  ++(0, 0.29) }
\def\Uc#1{\Ucolor{orange!65!white}{$2_{\mathrm #1}$}}%
\def\Ua#1{\Ucolor{yellow!20!white}{$0_{\mathrm #1}$}}%
\def\Db#1{\Dcolor{yellow!60!white}{$1_{\mathrm #1}$}}%
\def\Dd#1{\Dcolor{red!85!brown!70!white}{$3_{\mathrm #1}$}}%
\begin{tikzpicture}[scale=0.65, rotate=90]
\draw (0, 0    ) \Dd2\uo \Db1\uo \Dd2\uo \Db5\uo \Dd7\uo \Db A; 
\draw (0,-1.732) \Uc1\db0\Ua0\db0\Uc1\db3\Ua4\db5\Uc6\db7\Ua9;
\draw (0,-3.464) \dd1\uc0\db0\uc0\dd1\uc2\db3\uc6\dd7\uc8\db A;      
\end{tikzpicture}$$
By Lemma~\ref{l:infin}, we have that the number of colorings is infinite and each of them is obtained from 
one (say, from the unique coloring that corresponds to subcase (1)) by shifting coloring in some cosets of $X$.

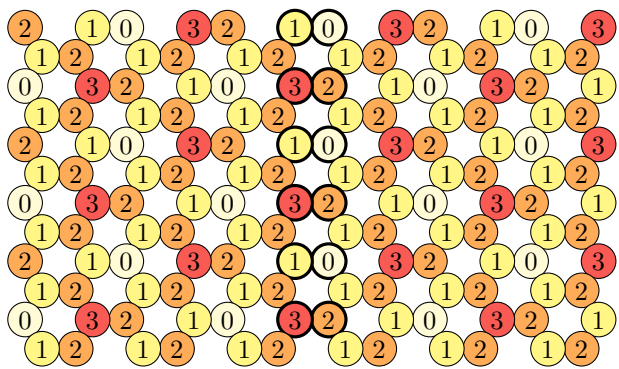
$$\input{03-102-201-30ab}$$

The case {\boldmath${[03{-}102{-}21]}$} leads to a contradiction: $$\def\ua#1{\ucolor{yellow!20!white}{$0_#1$}}%
\def\da#1{\dcolor{yellow!20!white}{$0_#1$}}%
\def\ub#1{\ucolor{yellow!60!white}{$1_#1$}}%
\def\db#1{\dcolor{yellow!60!white}{$1_#1$}}%
\def\uc#1{\ucolor{orange!65!white}{$2_#1$}}%
\def\dc#1{\dcolor{orange!65!white}{$2_#1$}}%
\def\ud#1{\ucolor{red!85!brown!70!white}{$3_#1$}}%
\def\dd#1{\dcolor{red!85!brown!70!white}{$3_#1$}}%
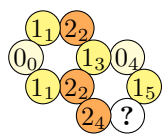
\begin{tikzpicture}[scale=0.65, rotate=-90]
\draw (1,-3.464)     \da4\ub5\dq ;
\draw (0,-5.192) \dc2\ub3\dc2\uc4;
\draw (0,-6.928) \ub1\da0\ub1;
\end{tikzpicture}$$

In the case {\boldmath${[03{-}102{-}30]}$}, the coloring is unique:
$$\def\ua#1{\ucolor{yellow!20!white}{$0_#1$}}%
\def\da#1{\dcolor{yellow!20!white}{$0_#1$}}%
\def\ub#1{\ucolor{yellow!60!white}{$1_#1$}}%
\def\db#1{\dcolor{yellow!60!white}{$1_#1$}}%
\def\uc#1{\ucolor{orange!65!white}{$2_#1$}}%
\def\dc#1{\dcolor{orange!65!white}{$2_#1$}}%
\def\ud#1{\ucolor{red!85!brown!70!white}{$3_#1$}}%
\def\dd#1{\dcolor{red!85!brown!70!white}{$3_#1$}}%
\def\ue#1{\ucolor{brown!80!black}{$4_#1$}}%
\def\de#1{\dcolor{brown!80!black}{$4_#1$}}%
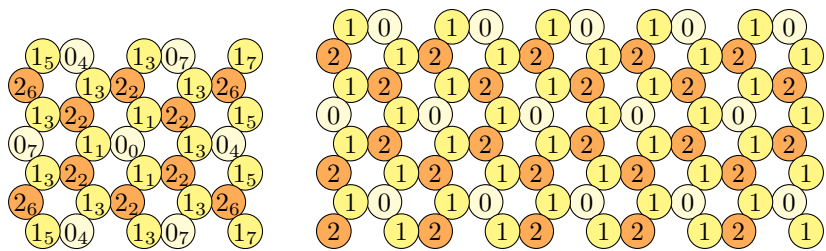
\begin{tikzpicture}[scale=0.65, rotate=-90]
\draw (0,-1.732)\ub7\dc6\ub5\da4\ub5\dc6\ub7;
\draw (0,-3.464)\da7\ub3\dc2\ub3\dc2\ub3\da7;
\draw (0,-5.192)\ub3\dc2\ub1\da0\ub1\dc2\ub3;
\draw (0,-6.928)\da4\ub3\dc2\ub1\dc2\ub3\da4;
\draw (0,-8.660)\ub5\dc6\ub3\da7\ub3\dc6\ub5;
\end{tikzpicture}
\qquad \input{03-102-30b}$$

In the case {\boldmath${[03{-}111{-}...]}$},  there are, up to rotation and reflection, 
two ways to color the nodes at distance $2$ from the original color-$0$ node:
$$\mbox{(a)} \raisebox{-2em}{\def\ua#1{\ucolor{yellow!20!white}{$0_#1$}}%
\def\da#1{\dcolor{yellow!20!white}{$0_#1$}}%
\def\ub#1{\ucolor{yellow!60!white}{$1_#1$}}%
\def\db#1{\dcolor{yellow!60!white}{$1_#1$}}%
\def\uc#1{\ucolor{orange!65!white}{$2_#1$}}%
\def\dc#1{\dcolor{orange!65!white}{$2_#1$}}%
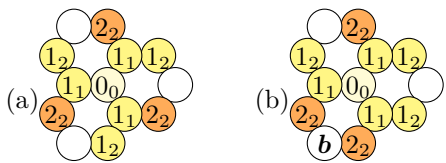
\begin{tikzpicture}[scale=0.65, rotate=-90]
\draw (1,-5.192)    \db2\uo \dc2 ;
\draw (0,-6.928)\dc2\ub1\da0\ub1\db2 ;
\draw (0,-8.660)
                \uo \db2\ub1\dc2\uo ;
\end{tikzpicture}
} \qquad \mbox{(b)}\raisebox{-2em}{\def\ua#1{\ucolor{yellow!20!white}{$0_#1$}}%
\def\da#1{\dcolor{yellow!20!white}{$0_#1$}}%
\def\ub#1{\ucolor{yellow!60!white}{$1_#1$}}%
\def\db#1{\dcolor{yellow!60!white}{$1_#1$}}%
\def\uc#1{\ucolor{orange!65!white}{$2_#1$}}%
\def\dc#1{\dcolor{orange!65!white}{$2_#1$}}%
\begin{tikzpicture}[scale=0.65, rotate=-90]
\draw (1,-5.192)    \db2\uo \db2 ;
\draw (0,-6.928)\dc2\ub1\da0\ub1\dc2 ;
\draw (0,-8.660)\uo\db2\ub1\dc2\ucolor{white}{$\boldsymbol b$} ;
\end{tikzpicture}
}$$
Let us complete the first variant:
$$\raisebox{-2em}{\def\uo{\ucolor{white}{$\ $}}%
\def\do{\dcolor{white}{$\ $}}%
\def\dxa{\dcolor{white}{$\boldsymbol a$}}%
\def\uxa{\ucolor{white}{$\boldsymbol a$}}%
\def\ua#1{\ucolor{yellow!20!white}{$0_#1$}}%
\def\da#1{\dcolor{yellow!20!white}{$0_#1$}}%
\def\ub#1{\ucolor{yellow!60!white}{$1_#1$}}%
\def\db#1{\dcolor{yellow!60!white}{$1_#1$}}%
\def\uc#1{\ucolor{orange!65!white}{$2_#1$}}%
\def\dc#1{\dcolor{orange!65!white}{$2_#1$}}%
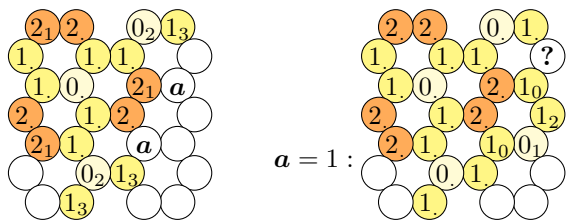
\begin{tikzpicture}[scale=0.65, rotate=-90]
\draw (0,-3.464)\db3\uo\dxa\uo\do\uo\do;
\draw (0,-5.192)\ua2\db.\uc1\dc.\uxa\db3\uo ;
\draw (0,-6.928)\dc.\ub.\da.\ub.\db.\ua2\db3;
\draw (0,-8.660)\uc1\db.\ub.\dc.\uc1\do \uo ;
\end{tikzpicture}
}\qquad {\boldsymbol a}=1:\raisebox{-2em}{\def\uo{\ucolor{white}{$\ $}}%
\def\do{\dcolor{white}{$\ $}}%
\def\dxa{\dcolor{white}{$\boldsymbol a$}}%
\def\ua#1{\ucolor{yellow!20!white}{$0_#1$}}%
\def\da#1{\dcolor{yellow!20!white}{$0_#1$}}%
\def\ub#1{\ucolor{yellow!60!white}{$1_#1$}}%
\def\db#1{\dcolor{yellow!60!white}{$1_#1$}}%
\def\uc#1{\ucolor{orange!65!white}{$2_#1$}}%
\def\dc#1{\dcolor{orange!65!white}{$2_#1$}}%
\begin{tikzpicture}[scale=0.65, rotate=-90]
\draw (0,-3.464)\db.\uq \db0\ub2\da1\uo\do;
\draw (0,-5.192)\ua.\db.\uc.\dc.\ub0\db.\uo ;
\draw (0,-6.928)\dc.\ub.\da.\ub.\db.\ua.\db.;
\draw (0,-8.660)\uc.\db.\ub.\dc.\uc.\do\uo;
\end{tikzpicture}
}$$
If the two nodes marked by {$\boldsymbol a$} have color $1$, then we get a contradiction; 
hence, their color is $2$: 
$$\def\do{\dcolor{white}{$\ $}}%
\def\ua#1{\ucolor{yellow!20!white}{$0_#1$}}%
\def\da#1{\dcolor{yellow!20!white}{$0_#1$}}%
\def\ub#1{\ucolor{yellow!60!white}{$1_#1$}}%
\def\db#1{\dcolor{yellow!60!white}{$1_#1$}}%
\def\uc#1{\ucolor{orange!65!white}{$2_#1$}}%
\def\dc#1{\dcolor{orange!65!white}{$2_#1$}}%
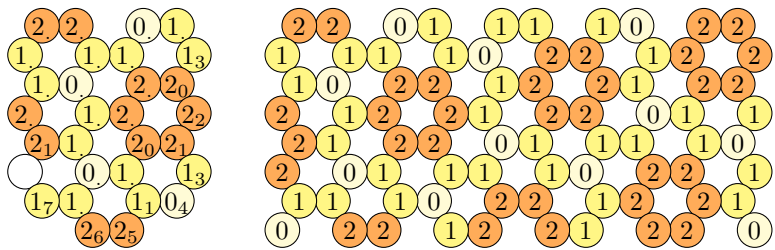
\begin{tikzpicture}[scale=0.65, rotate=-90]
\draw (0,-3.464)\db.\ub3\dc0\uc2\dc1\ub3\da4;
\draw (0,-5.192)\ua.\db.\uc.\dc.\uc0\db.\ub1\dc5;
\draw (0,-6.928)\dc.\ub.\da.\ub.\db.\ua.\db.\uc6;
\draw (0,-8.660)\uc.\db.\ub.\dc.\uc1\do\ub7;
\end{tikzpicture}
\qquad\def\uo{\ucolor{white}{$\ $}}
\def\do{\dcolor{white}{$\ $}}
\def\dxa{\dcolor{white}{$\boldsymbol a$}}
\def\uq{\ucolor{white}{$\boldsymbol ?$}}
\def\ua{\ucolor{yellow!20!white}{$0$}}
\def\da{\dcolor{yellow!20!white}{$0$}}
\def\ub{\ucolor{yellow!60!white}{$1$}}
\def\db{\dcolor{yellow!60!white}{$1$}}
\def\0{++(1,0)}
\def\uc{\ucolor{orange!65!white}{$2$}}
\def\dc{\dcolor{orange!65!white}{$2$}}
\def\ud{\ucolor{red!85!brown!70!white}{$3$}}
\def\dd{\dcolor{red!85!brown!70!white}{$3$}}
\def\ue{\ucolor{brown!80!black}{$4$}}
\def\de{\dcolor{brown!80!black}{$4$}}
\begin{tikzpicture}[scale=0.65, rotate=-90]
\draw (0, 6.928)\dc\uc\dc\ub\da\ub\db\ua;
\draw (0, 5.192)\uc\dc\uc\db\ub\dc\uc\dc;
\draw (0, 3.464)\da\ub\db\ua\db\uc\dc\uc;
\draw (0, 1.732)\ub\dc\uc\dc\ub\da\ub\db;
\draw (0,-0.000)\db\uc\dc\uc\db\ub\dc\uc;
\draw (0,-1.732)\ub\da\ub\db\ua\db\uc\dc;
\draw (0,-3.464)\db\ub\dc\uc\dc\ub\da\ub;
\draw (0,-5.192)\ua\db\uc\dc\uc\db\ub\dc;
\draw (0,-6.928)\dc\ub\da\ub\db\ua\db\uc;
\draw (0,-8.660)\uc\db\ub\dc\uc\dc\ub\da;
\end{tikzpicture}
$$
By analogy, the coloring of the whole grid is uniquely reconstructed; the cor\-res\-pond\-ing array is $[03{-}111{-}12]$.
Now, consider variant (b). In both subcases $\boldsymbol{b}=3$ and $\boldsymbol{b}=2$, 
the coloring is uniquely reconstructed,
resulting in the arrays $[03{-}111{-}111{-}30]$ and $[03{-}111{-}12]$ 
(note that we get a nonequivalent coloring with the same array as in variant (a)), respectively:
$$\def\uo{\ucolor{white}{$\ $}}
\def\do{\dcolor{white}{$\ $}}
\def\ua#1{\ucolor{yellow!20!white}{$0_{\mathrm #1}$}}
\def\da#1{\dcolor{yellow!20!white}{$0_{\mathrm #1}$}}
\def\ub#1{\ucolor{yellow!60!white}{$1_{\mathrm #1}$}}
\def\db#1{\dcolor{yellow!60!white}{$1_{\mathrm #1}$}}
\def\0{++(1,0)}
\def\uc#1{\ucolor{orange!65!white}{$2_{\mathrm #1}$}}
\def\dc#1{\dcolor{orange!65!white}{$2_{\mathrm #1}$}}
\def\ud#1{\ucolor{red!85!brown!70!white}{$3_{\mathrm #1}$}}
\def\dd#1{\dcolor{red!85!brown!70!white}{$3_{\mathrm #1}$}}
\def\ue#1{\ucolor{brown!80!black}{$4_{\mathrm #1}$}}
\def\de#1{\dcolor{brown!80!black}{$4_{\mathrm #1}$}}
\begin{tikzpicture}[scale=0.65, rotate=-90]
\draw (0,-1.732)\0  \dc D\0  \da B\0 \dc A\0 \db C\0 ;
\draw (0,-3.464)\dd C\uc B\db2\ub A\dc4\ud9\dc8\ub A\da B;
\draw (0,-5.192)\uc2\db.\ua1\db.\uc2\dd3\uc7\db7\ub8;
\draw (0,-6.928)\dc.\ub.\da.\ub.\dc.\uc2\db4\ua6\db7;
\draw (0,-8.660)\uc1\db.\ub.\dc.\ud0\dc6\ub5\da7\ub8;
\end{tikzpicture}
 \qquad \input{03-111g}$$

The array {\boldmath${[03{-}12]}$} is reduced to $[03{-}102{-}201{-}30]$, if we give new colors $2$ and $3$
to the even color-$1$ nodes and the odd color-$0$ nodes, respectively 
(in particular, there is infinite number of nonequivalent colorings with array $[03{-}12]$).

The next case is {\boldmath${[03{-}201{-}...]}$}. As in the case $[03{-}111{-}...]$, 
there are two possibilities to color the nodes at distance $2$ from the starting node:
$$\def\0{++(1,0)}%
\def\ua#1{\ucolor{yellow!20!white}{$0_#1$}}%
\def\da#1{\dcolor{yellow!20!white}{$0_#1$}}%
\def\ub#1{\ucolor{yellow!60!white}{$1_#1$}}%
\def\db#1{\dcolor{yellow!60!white}{$1_#1$}}%
\def\uc#1{\ucolor{orange!65!white}{$2_#1$}}%
\def\dc#1{\dcolor{orange!65!white}{$2_#1$}}%
\begin{tikzpicture}[scale=0.65, rotate=-90]
\draw (0,-5.192)\0  \da2\ub3\dc2;
\draw (0,-6.928)\dc2\ub1\da0\ub1\da2;
\draw (0,-8.660)\0  \da2\ub1\dc2;
\end{tikzpicture}
\qquad\def\0{++(1,0)}%
\def\ua#1{\ucolor{yellow!20!white}{$0_#1$}}%
\def\da#1{\dcolor{yellow!20!white}{$0_#1$}}%
\def\ub#1{\ucolor{yellow!60!white}{$1_#1$}}%
\def\db#1{\dcolor{yellow!60!white}{$1_#1$}}%
\def\uc#1{\ucolor{orange!65!white}{$2_#1$}}%
\def\dc#1{\dcolor{orange!65!white}{$2_#1$}}%
\begin{tikzpicture}[scale=0.65, rotate=-90]
\draw (0,-5.192)\0  \da2\ub3\dc2;
\draw (0,-6.928)\dc2\ub1\da0\ub1\da2;
\draw (0,-8.660)
                \0  \dc2\ub1\da2;
\end{tikzpicture}
$$
However, we will use the common part of both variants. As we see, a color-$2$ node has at least two 
color-$1$ neighbors; we consider three subcases:
$$ \mbox{1) }[03{-}201{-}201{-}...]:\raisebox{-2em}{\def\do{\dcolor{white}{$\ $}}%
\def\ua#1{\ucolor{yellow!20!white}{$0_#1$}}%
\def\da#1{\dcolor{yellow!20!white}{$0_#1$}}%
\def\ub#1{\ucolor{yellow!60!white}{$1_#1$}}%
\def\db#1{\dcolor{yellow!60!white}{$1_#1$}}%
\def\uc#1{\ucolor{orange!65!white}{$2_#1$}}%
\def\dc#1{\dcolor{orange!65!white}{$2_#1$}}%
\def\ud#1{\ucolor{red!85!brown!70!white}{$3_#1$}}%
\def\dd#1{\dcolor{red!85!brown!70!white}{$3_#1$}}%
\begin{tikzpicture}[scale=0.65, rotate=-90]
\draw (2,-3.464)        \da1\ub2\dc2;
\draw (1,-5.192)    \da.\ub.\dc.\ud1\dc2;
\draw (0,-6.928)\dc.\ub.\da.\ub.\da.\ub1;
\draw (1,-8.660)    \do \ub.\do ;
\end{tikzpicture}
} \qquad
   \mbox{2) }[03{-}201{-}21]:\raisebox{-2em}{\def\da#1{\dcolor{yellow!20!white}{$0_#1$}}%
\def\ub#1{\ucolor{yellow!60!white}{$1_#1$}}%
\def\uc#1{\ucolor{orange!65!white}{$2_#1$}}%
\def\dc#1{\dcolor{orange!65!white}{$2_#1$}}%
\begin{tikzpicture}[scale=0.65, rotate=-90]
\draw (0,-0.000)\dq \ub3\da2;
\draw (0,-1.732)\uc1\dc0\ub1;
\end{tikzpicture}
} $$
1) As we see, in subcase $[03{-}201{-}201{-}...]$, 
the only feasible array is $[03{-}201{-}201{-}30]$,
which symmetrical to  $[03{-}102{-}102{-}30]$, considered above;
2) the array $[03{-}201{-}21]$ is not feasible;
3) the array symmetrical to $[03{-}201{-}30]$ was considered above.

The array {\boldmath${[03{-}21]}$} is not feasible:
$$\def\ua#1{\ucolor{yellow!20!white}{$0_#1$}}%
\def\da#1{\dcolor{yellow!20!white}{$0_#1$}}%
\def\ub#1{\ucolor{yellow!60!white}{$1_#1$}}%
\def\db#1{\dcolor{yellow!60!white}{$1_#1$}}%
\begin{tikzpicture}[scale=0.65, rotate=-30]
\draw (0,-3.464)\db4\ua3 \db1\ub2\da3;
\draw (0,-5.192)\ua5\db1\ua0\db1\ub4;
\draw (1,-6.928)    \uq \db6\ua5 ;
\end{tikzpicture}
$$

The array {\boldmath${[03{-}30]}$} implies the nodes are colored in accordance with their parity.

We have finished to consider the arrays that start with $[03{-}...]$.

Next, if the array start with {\boldmath$[12{-}102{-}...]$}, then the next-but-one entry of the array, $a_2$ is not zero:
$$\def\0{++(1,0)}%
\def\uo{\ucolor{white}{$\ $}}%
\def\do{\dcolor{white}{$\ $}}%
\def\dxa{\dcolor{white}{$\boldsymbol a$}}%
\def\uxa{\ucolor{white}{$\boldsymbol a$}}%
\def\dxb{\dcolor{white}{$\boldsymbol b$}}%
\def\uxb{\ucolor{white}{$\boldsymbol b$}}%
\def\ua#1{\ucolor{yellow!20!white}{$0_{\mathrm #1}$}}%
\def\da#1{\dcolor{yellow!20!white}{$0_{\mathrm #1}$}}%
\def\ub#1{\ucolor{yellow!60!white}{$1_{\mathrm #1}$}}%
\def\db#1{\dcolor{yellow!60!white}{$1_{\mathrm #1}$}}%
\def\uc#1{\ucolor{orange!65!white}{$2_{\mathrm #1}$}}%
\def\dc#1{\dcolor{orange!65!white}{$2_{\mathrm #1}$}}%
\def\ud#1{\ucolor{red!85!brown!70!white}{$3_{\mathrm #1}$}}%
\def\dd#1{\dcolor{red!85!brown!70!white}{$3_{\mathrm #1}$}}%
\def\ue#1{\ucolor{brown!80!black}{$4_{\mathrm #1}$}}%
\def\de#1{\dcolor{brown!80!black}{$4_{\mathrm #1}$}}%
\def\ut{ ++(1, 0.29) node  {\raisebox{-0.25em}[0.2em][0em]{\makebox[0mm][c]{$\cdots$}}}  ++(0,-0.29) }%
\def\dt{ ++(1,-0.29) node  {\raisebox{-0.25em}[0.2em][0em]{\makebox[0mm][c]{$\cdots$}}}  ++(0, 0.29) }
\begin{tikzpicture}[scale=0.65, rotate=-30]
\draw (0,-5.192)\dc2\uc2\db1;
\draw (0,-6.928)\ub1\da0\ua0;
\end{tikzpicture}
$$
So, it is sufficient to consider the following three possibilities for the next three entries of the array.

1) Consider case $[12{-}102{-}111{-}...]$. Starting from two neighbor color-$0$ nodes,
we get a unique coloring with array $[12{-}102{-}111{-}21]$:
$$\input{12-102-111}$$

2) Consider the array $[12{-}102{-}12]$. Up to symmetry,
we get a unique coloring:
$$\def\ut{ ++(1, 0.29) node  {\raisebox{-0.25em}[0.2em][0em]{\makebox[0mm][c]{$\cdots$}}}  ++(0,-0.29) }
\def\dt{ ++(1,-0.29) node  {\raisebox{-0.25em}[0.2em][0em]{\makebox[0mm][c]{$\cdots$}}}  ++(0, 0.29) }
\def\uo{\ucolor{white}{$\ $}}
\def\do{\dcolor{white}{$\ $}}
\def\ux#1{\ucolor{white}{$#1$}}
\def\dx#1{\dcolor{white}{$#1$}}
\def\ua#1{\ucolor{yellow!20!white}{$0_{\mathrm{#1}}$}}
\def\da#1{\dcolor{yellow!20!white}{$0_{\mathrm{#1}}$}}
\def\ub#1{\ucolor{yellow!60!white}{$1_{\mathrm{#1}}$}}
\def\db#1{\dcolor{yellow!60!white}{$1_{\mathrm{#1}}$}}
\def\0{++(1,0)}
\def\uc#1{\ucolor{orange!65!white}{$2_{\mathrm{#1}}$}}
\def\dc#1{\dcolor{orange!65!white}{$2_{\mathrm{#1}}$}}
\def\ud#1{\ucolor{red!85!brown!70!white}{$3_{\mathrm{#1}}$}}
\def\dd#1{\dcolor{red!85!brown!70!white}{$3_{\mathrm{#1}}$}}
\def\ue#1{\ucolor{brown!80!black}{$4_{\mathrm{#1}}$}}
\def\de#1{\dcolor{brown!80!black}{$4_{\mathrm{#1}}$}}
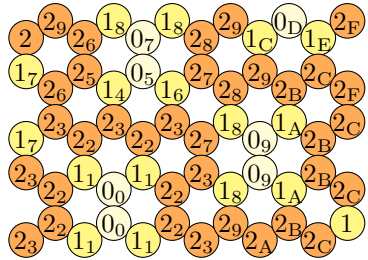
\begin{tikzpicture}[scale=0.65, rotate=0]
\draw (0,-1.732)\dc{}\uc9 \dc6\ub8 \da7\ub8 \dc8 \uc9 \db C\ua D\db E\uc F;
\draw (0,-3.464)\ub7 \dc6 \uc5\db4 \ua5\db6 \uc7 \dc8 \uc9 \dc B\uc C\dc F;
\draw (0,-5.192)\db7 \uc3 \dc2\uc3 \dc2\uc3 \dc7 \ub8 \da9 \ub A\dc B\uc C;
\draw (0,-6.928)\uc3 \dc2 \ub1\da0 \ub1\dc2 \uc3 \db8 \ua9 \db A\uc B\dc C;
\draw (0,-8.660)\dc3 \uc2 \db1\ua0 \db1\uc2 \dc3 \uc9 \dc A\uc B\dc C\ub{};
\end{tikzpicture}$$

3) The array $[12{-}102{-}21]$ is not feasible:
$$\def\ut{ ++(1, 0.29) node  {\raisebox{-0.25em}[0.2em][0em]{\makebox[0mm][c]{$\cdots$}}}  ++(0,-0.29) }
\def\dt{ ++(1,-0.29) node  {\raisebox{-0.25em}[0.2em][0em]{\makebox[0mm][c]{$\cdots$}}}  ++(0, 0.29) }
\def\uo{\ucolor{white}{$\ $}}
\def\do{\dcolor{white}{$\ $}}
\def\ux#1{\ucolor{white}{$#1$}}
\def\dx#1{\dcolor{white}{$#1$}}
\def\ua#1{\ucolor{yellow!20!white}{$0_{\mathrm{#1}}$}}
\def\da#1{\dcolor{yellow!20!white}{$0_{\mathrm{#1}}$}}
\def\ub#1{\ucolor{yellow!60!white}{$1_{\mathrm{#1}}$}}
\def\db#1{\dcolor{yellow!60!white}{$1_{\mathrm{#1}}$}}
\def\0{++(1,0)}
\def\uc#1{\ucolor{orange!65!white}{$2_{\mathrm{#1}}$}}
\def\dc#1{\dcolor{orange!65!white}{$2_{\mathrm{#1}}$}}
\def\ud#1{\ucolor{red!85!brown!70!white}{$3_{\mathrm{#1}}$}}
\def\dd#1{\dcolor{red!85!brown!70!white}{$3_{\mathrm{#1}}$}}
\def\ue#1{\ucolor{brown!80!black}{$4_{\mathrm{#1}}$}}
\def\de#1{\dcolor{brown!80!black}{$4_{\mathrm{#1}}$}}
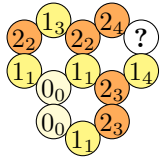
\begin{tikzpicture}[scale=0.65, rotate=0]
\draw (0, 3.464)\dc2\ub3\dc2\uc4\dcolor{white}{\bf?} ;
\draw (0, 1.732)\ub1\da0\ub1\dc3\ub4;
\draw (0,-0.000)\0  \ua0\db1\uc3;
\end{tikzpicture}$$

The arrays starting with {\boldmath${[12{-}111{-}102{-}...]}$}
are not feasible Lemma~\ref{l:111-102}.

Now consider  {\boldmath${[12{-}111{-}111{-}...]}$}. 
Taking into account the forbidden configuration
$$\def\0{++(1,0)}%
\def\uo{\ucolor{white}{$\ $}}%
\def\do{\dcolor{white}{$\ $}}%
\def\dxa{\dcolor{white}{$\boldsymbol a$}}%
\def\uxa{\ucolor{white}{$\boldsymbol a$}}%
\def\dxb{\dcolor{white}{$\boldsymbol b$}}%
\def\uxb{\ucolor{white}{$\boldsymbol b$}}%
\def\ua{\ucolor{yellow!20!white}{$0$}}%
\def\da{\dcolor{yellow!20!white}{$0$}}%
\def\ub{\ucolor{yellow!60!white}{$1$}}%
\def\db{\dcolor{yellow!60!white}{$1$}}%
\def\uc{\ucolor{orange!65!white}{$2$}}%
\def\dc{\dcolor{orange!65!white}{$2$}}%
\def\ud{\ucolor{red!85!brown!70!white}{$3$}}%
\def\dd{\dcolor{red!85!brown!70!white}{$3$}}%
\def\ue{\ucolor{brown!80!black}{$4$}}%
\def\de{\dcolor{brown!80!black}{$4$}}%
\def\um{\ucolor{brown!80!black}{\ }}%
\def\dm{\dcolor{brown!80!black}{\ }}%
\def\un{\ucolor{black}{\ }}%
\def\dn{\dcolor{black}{\ }}%
\def\Ua{\Ucolor{yellow!20!white}{$0$}}%
\def\Da{\Dcolor{yellow!20!white}{$0$}}%
\def\Ub{\Ucolor{yellow!60!white}{$1$}}%
\def\Db{\Dcolor{yellow!60!white}{$1$}}%
\def\Uc{\Ucolor{orange!65!white}{$2$}}%
\def\Dc{\Dcolor{orange!65!white}{$2$}}%
\def\Ud{\Ucolor{red!85!brown!70!white}{$3$}}%
\def\Dd{\Dcolor{red!85!brown!70!white}{$3$}}%
\def\ut{ ++(1, 0.29) node  {\raisebox{-0.25em}[0.2em][0em]{\makebox[0mm][c]{$\cdots$}}}  ++(0,-0.29) }%
\def\dt{ ++(1,-0.29) node  {\raisebox{-0.25em}[0.2em][0em]{\makebox[0mm][c]{$\cdots$}}}  ++(0, 0.29) }
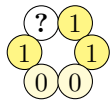
\begin{tikzpicture}[scale=0.65, rotate=-30]
\draw (0, 1.732)\dq\ub\db;
\draw (0,-0.000)\ub\da\ua;
\end{tikzpicture}$$
we see that there is an infinite sequence of nodes colored by $0$:\\
$I_0=\{ (yxz)^i, (yxz)^ix, x(yxz)^i, x(yxz)^ix \,:\, i=0,1,2,... \}$.
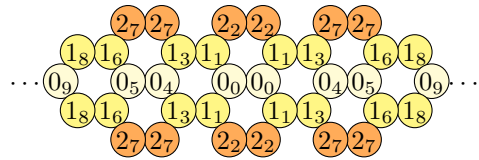
$$\input{12-111-111b}$$
The nodes at distance one from this sequence, i.e., 
$I_1=\{vy,vz \,:\, v\in I_0\}$,
are colored by $1$; the next rows of nodes $I_2=\{vyz,vzy \,:\, v\in I_0\}$, by $2$;
then $I_3=\{vyzy,vzyz \,:\, v\in I_0\}$ by $3$.
The color of $I_4=\{vyzyz,vzyzy \,:\, v\in I_0\}$ is defined by $c_3$, $a_3$, $b_3$.
It can be $2$, $3$, or $4$ for the cases $[12{-}111{-}111{-}21]$, $[12{-}111{-}111{-}12]$,
or $[12{-}111{-}111{-}111{-}...]$, respectively. In the first and the second cases, 
the grid will be uniquely colored by the colors 
$0$, $1$, $2$, $3$, $2$, $1$, $0$, $1$, $2$, $3$, $2$, $1$, $0$, $1$, $\ldots$ 
and $0$, $1$, $2$, $3$, $3$, $2$, $1$, $0$, $1$, $2$, $3$, $3$, $2$, $\ldots$,
respectively, in accordance with the distance to $I_0$.
In the last case, we repeat the arguments for $I_5$, resulting in $[12{-}111{-}111{-}111{-}21]$, $[12{-}111{-}111{-}111{-}12]$,
or $[12{-}111{-}111{-}111{-}111{-}...]$, and so on. As the number of colors is finite, at some step the color number will stop
to increase, the array will be determined, and the rest of the grid will be colored uniquely.
The array will be either $[12{-}111{-}...{-}111{-}21]$, or $[12{-}111{-}...{-}111{-}12]$.
$$\def\ut{ ++(1, 0.29) node  {\raisebox{-0.25em}[0.2em][0em]{\makebox[0mm][c]{$\cdots$}}}  ++(0,-0.29) }
\def\dt{ ++(1,-0.29) node  {\raisebox{-0.25em}[0.2em][0em]{\makebox[0mm][c]{$\cdots$}}}  ++(0, 0.29) }
\def\uo{\ucolor{white}{$\ $}}
\def\do{\dcolor{white}{$\ $}}
\def\ux{\ucolor{white}{$$}}
\def\dx{\dcolor{white}{$$}}
\def\ua{\ucolor{yellow!20!white}{$0$}}
\def\da{\dcolor{yellow!20!white}{$0$}}
\def\ub{\ucolor{yellow!60!white}{$1$}}
\def\db{\dcolor{yellow!60!white}{$1$}}
\def\0{++(1,0)}
\def\uc{\ucolor{orange!65!white}{$2$}}
\def\dc{\dcolor{orange!65!white}{$2$}}
\def\ud{\ucolor{red!85!brown!70!white}{$3$}}
\def\dd{\dcolor{red!85!brown!70!white}{$3$}}
\def\ue{\ucolor{brown!80!black}{$\ $}}
\def\de{\dcolor{brown!80!black}{$\ $}}
\def\uf{\ucolor{black}{$\ $}}
\def\df{\dcolor{black}{$\ $}}
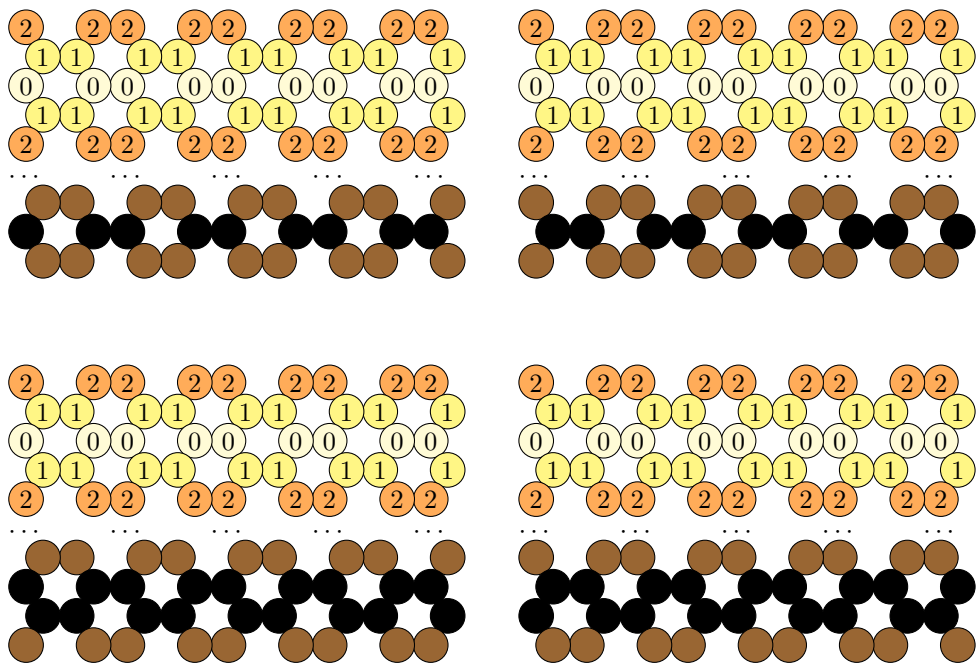
\begin{tikzpicture}[scale=0.65, rotate=-90]
\draw (0, 5.192) \dc\ub\da\ub\dc\dt\ue\df\ue;
\draw (0, 3.464) \uc\db\ua\db\uc\0 \de\uf\de;
\draw (0, 1.732) \dc\ub\da\ub\dc\dt\ue\df\ue;
\draw (0,-0.000) \uc\db\ua\db\uc\0 \de\uf\de;
\draw (0,-1.732) \dc\ub\da\ub\dc\dt\ue\df\ue;
\draw (0,-3.464) \uc\db\ua\db\uc\0 \de\uf\de;
\draw (0,-5.192) \dc\ub\da\ub\dc\dt\ue\df\ue;
\draw (0,-6.928) \uc\db\ua\db\uc\0 \de\uf\de;
\draw (0,-8.660) \dc\ub\da\ub\dc\dt\ue\df\ue;
\end{tikzpicture}\qquad\def\ut{ ++(1, 0.29) node  {\raisebox{-0.25em}[0.2em][0em]{\makebox[0mm][c]{$\cdots$}}}  ++(0,-0.29) }
\def\dt{ ++(1,-0.29) node  {\raisebox{-0.25em}[0.2em][0em]{\makebox[0mm][c]{$\cdots$}}}  ++(0, 0.29) }
\def\uo{\ucolor{white}{$\ $}}
\def\do{\dcolor{white}{$\ $}}
\def\ux{\ucolor{white}{$$}}
\def\dx{\dcolor{white}{$$}}
\def\ua{\ucolor{yellow!20!white}{$0$}}
\def\da{\dcolor{yellow!20!white}{$0$}}
\def\ub{\ucolor{yellow!60!white}{$1$}}
\def\db{\dcolor{yellow!60!white}{$1$}}
\def\0{++(1,0)}
\def\uc{\ucolor{orange!65!white}{$2$}}
\def\dc{\dcolor{orange!65!white}{$2$}}
\def\ud{\ucolor{red!85!brown!70!white}{$3$}}
\def\dd{\dcolor{red!85!brown!70!white}{$3$}}
\def\ue{\ucolor{brown!80!black}{$\ $}}
\def\de{\dcolor{brown!80!black}{$\ $}}
\def\uf{\ucolor{black}{$\ $}}
\def\df{\dcolor{black}{$\ $}}
\begin{tikzpicture}[scale=0.65, rotate=-90]
\draw (0, 5.192) \dc\ub\da\ub\dc\dt\de\uf\de;
\draw (0, 3.464) \uc\db\ua\db\uc\0 \ue\df\ue;
\draw (0, 1.732) \dc\ub\da\ub\dc\dt\de\uf\de;
\draw (0,-0.000) \uc\db\ua\db\uc\0 \ue\df\ue;
\draw (0,-1.732) \dc\ub\da\ub\dc\dt\de\uf\de;
\draw (0,-3.464) \uc\db\ua\db\uc\0 \ue\df\ue;
\draw (0,-5.192) \dc\ub\da\ub\dc\dt\de\uf\de;
\draw (0,-6.928) \uc\db\ua\db\uc\0 \ue\df\ue;
\draw (0,-8.660) \dc\ub\da\ub\dc\dt\de\uf\de;
\end{tikzpicture}$$
$$\def\ut{ ++(1, 0.29) node  {\raisebox{-0.25em}[0.2em][0em]{\makebox[0mm][c]{$\cdots$}}}  ++(0,-0.29) }
\def\dt{ ++(1,-0.29) node  {\raisebox{-0.25em}[0.2em][0em]{\makebox[0mm][c]{$\cdots$}}}  ++(0, 0.29) }
\def\uo{\ucolor{white}{$\ $}}
\def\do{\dcolor{white}{$\ $}}
\def\ux{\ucolor{white}{$$}}
\def\dx{\dcolor{white}{$$}}
\def\ua{\ucolor{yellow!20!white}{$0$}}
\def\da{\dcolor{yellow!20!white}{$0$}}
\def\ub{\ucolor{yellow!60!white}{$1$}}
\def\db{\dcolor{yellow!60!white}{$1$}}
\def\0{++(1,0)}
\def\uc{\ucolor{orange!65!white}{$2$}}
\def\dc{\dcolor{orange!65!white}{$2$}}
\def\ud{\ucolor{red!85!brown!70!white}{$3$}}
\def\dd{\dcolor{red!85!brown!70!white}{$3$}}
\def\ue{\ucolor{brown!80!black}{$\ $}}
\def\de{\dcolor{brown!80!black}{$\ $}}
\def\uf{\ucolor{black}{$\ $}}
\def\df{\dcolor{black}{$\ $}}
\begin{tikzpicture}[scale=0.65, rotate=-90]
\draw (0, 5.192) \dc\ub\da\ub\dc\dt\ue\df\uf\de;
\draw (0, 3.464) \uc\db\ua\db\uc\0 \de\uf\df\ue;
\draw (0, 1.732) \dc\ub\da\ub\dc\dt\ue\df\uf\de;
\draw (0,-0.000) \uc\db\ua\db\uc\0 \de\uf\df\ue;
\draw (0,-1.732) \dc\ub\da\ub\dc\dt\ue\df\uf\de;
\draw (0,-3.464) \uc\db\ua\db\uc\0 \de\uf\df\ue;
\draw (0,-5.192) \dc\ub\da\ub\dc\dt\ue\df\uf\de;
\draw (0,-6.928) \uc\db\ua\db\uc\0 \de\uf\df\ue;
\draw (0,-8.660) \dc\ub\da\ub\dc\dt\ue\df\uf\de;
\end{tikzpicture}\qquad\def\ut{ ++(1, 0.29) node  {\raisebox{-0.25em}[0.2em][0em]{\makebox[0mm][c]{$\cdots$}}}  ++(0,-0.29) }
\def\dt{ ++(1,-0.29) node  {\raisebox{-0.25em}[0.2em][0em]{\makebox[0mm][c]{$\cdots$}}}  ++(0, 0.29) }
\def\uo{\ucolor{white}{$\ $}}
\def\do{\dcolor{white}{$\ $}}
\def\ux{\ucolor{white}{$$}}
\def\dx{\dcolor{white}{$$}}
\def\ua{\ucolor{yellow!20!white}{$0$}}
\def\da{\dcolor{yellow!20!white}{$0$}}
\def\ub{\ucolor{yellow!60!white}{$1$}}
\def\db{\dcolor{yellow!60!white}{$1$}}
\def\0{++(1,0)}
\def\uc{\ucolor{orange!65!white}{$2$}}
\def\dc{\dcolor{orange!65!white}{$2$}}
\def\ud{\ucolor{red!85!brown!70!white}{$3$}}
\def\dd{\dcolor{red!85!brown!70!white}{$3$}}
\def\ue{\ucolor{brown!80!black}{$\ $}}
\def\de{\dcolor{brown!80!black}{$\ $}}
\def\uf{\ucolor{black}{$\ $}}
\def\df{\dcolor{black}{$\ $}}
\begin{tikzpicture}[scale=0.65, rotate=-90]
\draw (0, 5.192) \dc\ub\da\ub\dc\dt\de\uf\df\ue;
\draw (0, 3.464) \uc\db\ua\db\uc\0 \ue\df\uf\de;
\draw (0, 1.732) \dc\ub\da\ub\dc\dt\de\uf\df\ue;
\draw (0,-0.000) \uc\db\ua\db\uc\0 \ue\df\uf\de;
\draw (0,-1.732) \dc\ub\da\ub\dc\dt\de\uf\df\ue;
\draw (0,-3.464) \uc\db\ua\db\uc\0 \ue\df\uf\de;
\draw (0,-5.192) \dc\ub\da\ub\dc\dt\de\uf\df\ue;
\draw (0,-6.928) \uc\db\ua\db\uc\0 \ue\df\uf\de;
\draw (0,-8.660) \dc\ub\da\ub\dc\dt\de\uf\df\ue;
\end{tikzpicture}$$

The case {\boldmath${[12{-}111{-}201{-}...]}$} leads to the array 
symmetrical to ${[12{-}102{-}111{-}21]}$, considered above:
$$\def\0{++(1,0)}%
\def\uo{\ucolor{white}{$\ $}}%
\def\do{\dcolor{white}{$\ $}}%
\def\dxa{\dcolor{white}{$\boldsymbol a$}}%
\def\uxa{\ucolor{white}{$\boldsymbol a$}}%
\def\dxb{\dcolor{white}{$\boldsymbol b$}}%
\def\uxb{\ucolor{white}{$\boldsymbol b$}}%
\def\ua#1{\ucolor{yellow!20!white}{$0_{\mathrm #1}$}}%
\def\da#1{\dcolor{yellow!20!white}{$0_{\mathrm #1}$}}%
\def\ub#1{\ucolor{yellow!60!white}{$1_{\mathrm #1}$}}%
\def\db#1{\dcolor{yellow!60!white}{$1_{\mathrm #1}$}}%
\def\uc#1{\ucolor{orange!65!white}{$2_{\mathrm #1}$}}%
\def\dc#1{\dcolor{orange!65!white}{$2_{\mathrm #1}$}}%
\def\ud#1{\ucolor{red!85!brown!70!white}{$3_{\mathrm #1}$}}%
\def\dd#1{\dcolor{red!85!brown!70!white}{$3_{\mathrm #1}$}}%
\def\ue#1{\ucolor{brown!80!black}{$4_{\mathrm #1}$}}%
\def\de#1{\dcolor{brown!80!black}{$4_{\mathrm #1}$}}%
\def\ut{ ++(1, 0.29) node  {\raisebox{-0.25em}[0.2em][0em]{\makebox[0mm][c]{$\cdots$}}}  ++(0,-0.29) }%
\def\dt{ ++(1,-0.29) node  {\raisebox{-0.25em}[0.2em][0em]{\makebox[0mm][c]{$\cdots$}}}  ++(0, 0.29) }
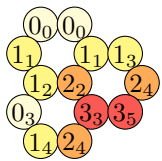
\begin{tikzpicture}[scale=0.65, rotate=-30]
\draw (0,-3.464)\da0\ua0\db1\ub3\dc4;
\draw (0,-5.192)\ub1\db2\uc2\dd3\ud5;
\draw (1,-6.928)    \ua3\db4\uc4;
\end{tikzpicture}$$

Consider the array {\boldmath${[12{-}111{-}21]}$}.
Identifying the colors $0$ and $2$ in any coloring with this array, 
we obtain a coloring with array ${[12{-}21]}$.
Inversely, a coloring with array ${[12{-}21]}$ can be obtained from 
a coloring with array ${[12{-}111{-}21]}$ by identifying the colors $0$ and $2$
if and only if the set of pairs of neighbor nodes of color $0$ can be partitioned to two parts 
such that no two pairs from the same part are at distance $2$ from each other.
So, the result for {${[12{-}111{-}21]}$} follows from that for ${[12{-}21]}$ and will be considered later
(see the case ${[21{-}12]}$).

The array symmetrical to {\boldmath${[12{-}111{-}30]}$} was considered above.

The array {\boldmath${[12{-}12]}$} is feasible with a unique coloring, up to equivalence:
$$\def\0{++(1,0)}%
\def\uo{\ucolor{white}{$\ $}}%
\def\do{\dcolor{white}{$\ $}}%
\def\dxa{\dcolor{white}{$\boldsymbol a$}}%
\def\uxa{\ucolor{white}{$\boldsymbol a$}}%
\def\dxb{\dcolor{white}{$\boldsymbol b$}}%
\def\uxb{\ucolor{white}{$\boldsymbol b$}}%
\def\ua#1{\ucolor{yellow!20!white}{$0_{\mathrm #1}$}}%
\def\da#1{\dcolor{yellow!20!white}{$0_{\mathrm #1}$}}%
\def\ub#1{\ucolor{yellow!60!white}{$1_{\mathrm #1}$}}%
\def\db#1{\dcolor{yellow!60!white}{$1_{\mathrm #1}$}}%
\def\uc#1{\ucolor{orange!65!white}{$2_{\mathrm #1}$}}%
\def\dc#1{\dcolor{orange!65!white}{$2_{\mathrm #1}$}}%
\def\ud#1{\ucolor{red!85!brown!70!white}{$3_{\mathrm #1}$}}%
\def\dd#1{\dcolor{red!85!brown!70!white}{$3_{\mathrm #1}$}}%
\def\ue#1{\ucolor{brown!80!black}{$4_{\mathrm #1}$}}%
\def\de#1{\dcolor{brown!80!black}{$4_{\mathrm #1}$}}%
\def\ut{ ++(1, 0.29) node  {\raisebox{-0.25em}[0.2em][0em]{\makebox[0mm][c]{$\cdots$}}}  ++(0,-0.29) }%
\def\dt{ ++(1,-0.29) node  {\raisebox{-0.25em}[0.2em][0em]{\makebox[0mm][c]{$\cdots$}}}  ++(0, 0.29) }
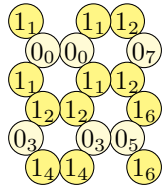
\begin{tikzpicture}[scale=0.65, rotate=-90]
\draw (0,-0.000)\db2\ua7\db2\ub6\da5\ub6;
\draw (0,-1.732)\ub1\da0\ub1\db2\ua3\db4;
\draw (0,-3.464)\db1\ua0\db1\ub2\da3\ub4;
\end{tikzpicture}$$

Case {\boldmath${[12{-}201{-}102{-}...]}$}. If the origin and $x$ are colored by $0$,
then the only possibility for $yz$ and $yzx$ is both having color $0$ too.
Similarly, $zy$, $zyx$, 
$$\def\0{++(1,0)}%
\def\uo{\ucolor{white}{$\ $}}%
\def\do{\dcolor{white}{$\ $}}%
\def\dxa{\dcolor{white}{$\boldsymbol a$}}%
\def\uxa{\ucolor{white}{$\boldsymbol a$}}%
\def\dxb{\dcolor{white}{$\boldsymbol b$}}%
\def\uxb{\ucolor{white}{$\boldsymbol b$}}%
\def\ua#1{\ucolor{yellow!20!white}{$0_{\mathrm #1}$}}%
\def\da#1{\dcolor{yellow!20!white}{$0_{\mathrm #1}$}}%
\def\ub#1{\ucolor{yellow!60!white}{$1_{\mathrm #1}$}}%
\def\db#1{\dcolor{yellow!60!white}{$1_{\mathrm #1}$}}%
\def\uc#1{\ucolor{orange!65!white}{$2_{\mathrm #1}$}}%
\def\dc#1{\dcolor{orange!65!white}{$2_{\mathrm #1}$}}%
\def\ud#1{\ucolor{red!85!brown!70!white}{$3_{\mathrm #1}$}}%
\def\dd#1{\dcolor{red!85!brown!70!white}{$3_{\mathrm #1}$}}%
\def\ue#1{\ucolor{brown!80!black}{$4_{\mathrm #1}$}}%
\def\de#1{\dcolor{brown!80!black}{$4_{\mathrm #1}$}}%
\def\ut{ ++(1, 0.29) node  {\raisebox{-0.25em}[0.2em][0em]{\makebox[0mm][c]{$\cdots$}}}  ++(0,-0.29) }%
\def\dt{ ++(1,-0.29) node  {\raisebox{-0.25em}[0.2em][0em]{\makebox[0mm][c]{$\cdots$}}}  ++(0, 0.29) }
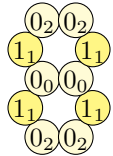
\begin{tikzpicture}[scale=0.65, rotate=90]
\draw (0,-3.464)\da2 \ub1 \da0 \ub1 \da2 ;
\draw (0,-5.192)\ua2 \db1 \ua0 \db1 \ua2 ;
\end{tikzpicture}$$
and, by induction, 
all nodes from $I_0=\{(zy)^i, (zy)^i x, (yz)^i, (yz)^i x \,:\,i=0,1,2,\ldots\}$ have color $0$. 
Denote by $I_d$ the set of nodes at distance $d$ from $I_0$. Then the nodes from $I_1$ are colored by $1$;
from $I_2$, by $2$; from $I_3$, by $3$.
The nodes from $I_4$ are colored by the same color, and it is not $2$:
$$\def\0{++(1,0)}%
\def\uo{\ucolor{white}{$\ $}}%
\def\do{\dcolor{white}{$\ $}}%
\def\dxa{\dcolor{white}{$\boldsymbol a$}}%
\def\uxa{\ucolor{white}{$\boldsymbol a$}}%
\def\dxb{\dcolor{white}{$\boldsymbol b$}}%
\def\uxb{\ucolor{white}{$\boldsymbol b$}}%
\def\ua#1{\ucolor{yellow!20!white}{$0_{\mathrm #1}$}}%
\def\da#1{\dcolor{yellow!20!white}{$0_{\mathrm #1}$}}%
\def\ub#1{\ucolor{yellow!60!white}{$1_{\mathrm #1}$}}%
\def\db#1{\dcolor{yellow!60!white}{$1_{\mathrm #1}$}}%
\def\uc#1{\ucolor{orange!65!white}{$2_{\mathrm #1}$}}%
\def\dc#1{\dcolor{orange!65!white}{$2_{\mathrm #1}$}}%
\def\ud#1{\ucolor{red!85!brown!70!white}{$3_{\mathrm #1}$}}%
\def\dd#1{\dcolor{red!85!brown!70!white}{$3_{\mathrm #1}$}}%
\def\ue#1{\ucolor{brown!80!black}{$4_{\mathrm #1}$}}%
\def\de#1{\dcolor{brown!80!black}{$4_{\mathrm #1}$}}%
\def\ut{ ++(1, 0.29) node  {\raisebox{-0.25em}[0.2em][0em]{\makebox[0mm][c]{$\cdots$}}}  ++(0,-0.29) }%
\def\dt{ ++(1,-0.29) node  {\raisebox{-0.25em}[0.2em][0em]{\makebox[0mm][c]{$\cdots$}}}  ++(0, 0.29) }
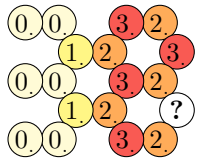
\begin{tikzpicture}[scale=0.65, rotate=-90]
\draw (0,-0.000)\dc. \ud. \dc. \uq \dc.;
\draw (0,-1.732)\ud. \dc. \ud. \dc. \ud.;
\draw (0,-3.464)\da. \ub. \da. \ub. \da.;
\draw (0,-5.192)\ua. \0   \ua. \0   \ua.;
\end{tikzpicture}$$
If this color is $3$, then the nodes of $I_5$, $I_6$, and $I_7$ have colors $2$, $1$, and $0$, respectively, 
and then the situation repeats with period $8$ (in terms of the index $d$ of $I_d$).
In the case of $4$, we have two subcases for $I_5$.
The color $4$ for $I_5$ means the colors $3$, $2$, $1$, and $0$ for $I_6$, $I_7$, $I_8$, and $I_9$, respectively;
then, the colors repeat.
If the nodes of $I_5$ have color $5$, then we repeat the arguments above as many times as we need to exhaust all the colors,
whose number is finite. Finaly, we have the array $[12{-}201{-}...{-}102{-}21]$ if the number of colors is even
and  $[12{-}201{-}...{-}201{-}12]$ if it is odd. In any case, the coloring is unique.
$$\def\0{++(1,0)}%
\def\uo{\ucolor{white}{$\ $}}%
\def\do{\dcolor{white}{$\ $}}%
\def\dxa{\dcolor{white}{$\boldsymbol a$}}%
\def\uxa{\ucolor{white}{$\boldsymbol a$}}%
\def\dxb{\dcolor{white}{$\boldsymbol b$}}%
\def\uxb{\ucolor{white}{$\boldsymbol b$}}%
\def\ua{\ucolor{yellow!20!white}{$0$}}%
\def\da{\dcolor{yellow!20!white}{$0$}}%
\def\ub{\ucolor{yellow!60!white}{$1$}}%
\def\db{\dcolor{yellow!60!white}{$1$}}%
\def\uc{\ucolor{orange!65!white}{$2$}}%
\def\dc{\dcolor{orange!65!white}{$2$}}%
\def\ud{\ucolor{red!85!brown!70!white}{$3$}}%
\def\dd{\dcolor{red!85!brown!70!white}{$3$}}%
\def\ue{\ucolor{brown!80!black}{$4$}}%
\def\de{\dcolor{brown!80!black}{$4$}}%
\def\um{\ucolor{brown!80!black}{\ }}%
\def\dm{\dcolor{brown!80!black}{\ }}%
\def\un{\ucolor{black}{\ }}%
\def\dn{\dcolor{black}{\ }}%
\def\Ua{\Ucolor{yellow!20!white}{$0$}}%
\def\Da{\Dcolor{yellow!20!white}{$0$}}%
\def\Ub{\Ucolor{yellow!60!white}{$1$}}%
\def\Db{\Dcolor{yellow!60!white}{$1$}}%
\def\Uc{\Ucolor{orange!65!white}{$2$}}%
\def\Dc{\Dcolor{orange!65!white}{$2$}}%
\def\Ud{\Ucolor{red!85!brown!70!white}{$3$}}%
\def\Dd{\Dcolor{red!85!brown!70!white}{$3$}}%
\def\ut{ ++(1, 0.29) node  {\raisebox{-0.25em}[0.2em][0em]{\makebox[0mm][c]{$\cdots$}}}  ++(0,-0.29) }%
\def\dt{ ++(1,-0.29) node  {\raisebox{-0.25em}[0.2em][0em]{\makebox[0mm][c]{$\cdots$}}}  ++(0, 0.29) }
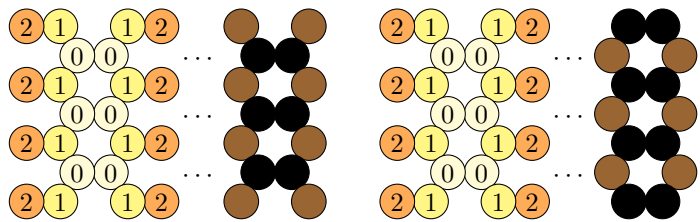
\begin{tikzpicture}[scale=0.65, rotate=-90]
\draw (0, 2.732)\um\dn\um\dn\um\dn\um;
\draw (0, 1.000)\dm\un\dm\un\dm\un\dm;
\draw (0,-1.000)\0 \ut\0 \ut\0 \ut\0 ;
\draw (0,-1.732)\dc\0 \dc\0 \dc\0 \dc;
\draw (0,-3.464)\ub\da\ub\da\ub\da\ub;
\draw (0,-5.192)\db\ua\db\ua\db\ua\db;
\draw (0,-6.928)\uc\0 \uc\0 \uc\0 \uc;
\end{tikzpicture}\qquad\def\0{++(1,0)}%
\def\uo{\ucolor{white}{$\ $}}%
\def\do{\dcolor{white}{$\ $}}%
\def\dxa{\dcolor{white}{$\boldsymbol a$}}%
\def\uxa{\ucolor{white}{$\boldsymbol a$}}%
\def\dxb{\dcolor{white}{$\boldsymbol b$}}%
\def\uxb{\ucolor{white}{$\boldsymbol b$}}%
\def\ua{\ucolor{yellow!20!white}{$0$}}%
\def\da{\dcolor{yellow!20!white}{$0$}}%
\def\ub{\ucolor{yellow!60!white}{$1$}}%
\def\db{\dcolor{yellow!60!white}{$1$}}%
\def\uc{\ucolor{orange!65!white}{$2$}}%
\def\dc{\dcolor{orange!65!white}{$2$}}%
\def\ud{\ucolor{red!85!brown!70!white}{$3$}}%
\def\dd{\dcolor{red!85!brown!70!white}{$3$}}%
\def\ue{\ucolor{brown!80!black}{$4$}}%
\def\de{\dcolor{brown!80!black}{$4$}}%
\def\um{\ucolor{brown!80!black}{\ }}%
\def\dm{\dcolor{brown!80!black}{\ }}%
\def\un{\ucolor{black}{\ }}%
\def\dn{\dcolor{black}{\ }}%
\def\Ua{\Ucolor{yellow!20!white}{$0$}}%
\def\Da{\Dcolor{yellow!20!white}{$0$}}%
\def\Ub{\Ucolor{yellow!60!white}{$1$}}%
\def\Db{\Dcolor{yellow!60!white}{$1$}}%
\def\Uc{\Ucolor{orange!65!white}{$2$}}%
\def\Dc{\Dcolor{orange!65!white}{$2$}}%
\def\Ud{\Ucolor{red!85!brown!70!white}{$3$}}%
\def\Dd{\Dcolor{red!85!brown!70!white}{$3$}}%
\def\ut{ ++(1, 0.29) node  {\raisebox{-0.25em}[0.2em][0em]{\makebox[0mm][c]{$\cdots$}}}  ++(0,-0.29) }%
\def\dt{ ++(1,-0.29) node  {\raisebox{-0.25em}[0.2em][0em]{\makebox[0mm][c]{$\cdots$}}}  ++(0, 0.29) }
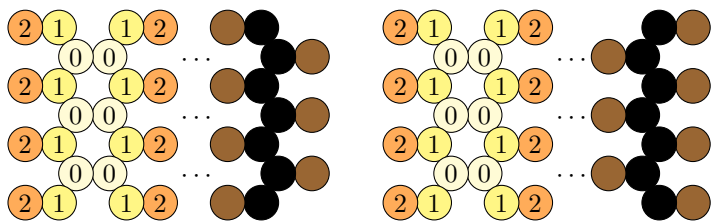
\begin{tikzpicture}[scale=0.65, rotate=-90]
\draw (0, 2.732)\dn\um\dn\um\dn\um\dn;
\draw (0, 1.000)\un\dm\un\dm\un\dm\un;
\draw (0,-1.000)\0 \ut\0 \ut\0 \ut\0 ;
\draw (0,-1.732)\dc\0 \dc\0 \dc\0 \dc;
\draw (0,-3.464)\ub\da\ub\da\ub\da\ub;
\draw (0,-5.192)\db\ua\db\ua\db\ua\db;
\draw (0,-6.928)\uc\0 \uc\0 \uc\0 \uc;
\end{tikzpicture}$$
$$\input{12-201-102e}\qquad\input{12-201-102f}$$

The case {\boldmath${[12{-}201{-}111{-}...]}$} is not feasible, by Lemma~\ref{l:111-102}.

The case {\boldmath${[12{-}201{-}12]}$} will be divided into two subcases:
(a) there is four color-$2$ nodes that belong to the same length-$6$ cycle (hexagon); 
obviously, the remaining two nodes of the cycle have the same color; 
(b) there are no four nodes belonging to the same hexagon. In both cases,
the coloring is unique, up to equivalence.
$$\def\0{++(1,0)}%
\def\uo{\ucolor{white}{$\ $}}%
\def\do{\dcolor{white}{$\ $}}%
\def\dxa{\dcolor{white}{$\boldsymbol a$}}%
\def\uxa{\ucolor{white}{$\boldsymbol a$}}%
\def\dxb{\dcolor{white}{$\boldsymbol b$}}%
\def\uxb{\ucolor{white}{$\boldsymbol b$}}%
\def\ua#1{\ucolor{yellow!20!white}{$0_{\mathrm #1}$}}%
\def\da#1{\dcolor{yellow!20!white}{$0_{\mathrm #1}$}}%
\def\ub#1{\ucolor{yellow!60!white}{$1_{\mathrm #1}$}}%
\def\db#1{\dcolor{yellow!60!white}{$1_{\mathrm #1}$}}%
\def\uc#1{\ucolor{orange!65!white}{$2_{\mathrm #1}$}}%
\def\dc#1{\dcolor{orange!65!white}{$2_{\mathrm #1}$}}%
\def\ud#1{\ucolor{red!85!brown!70!white}{$3_{\mathrm #1}$}}%
\def\dd#1{\dcolor{red!85!brown!70!white}{$3_{\mathrm #1}$}}%
\def\ue#1{\ucolor{brown!80!black}{$4_{\mathrm #1}$}}%
\def\de#1{\dcolor{brown!80!black}{$4_{\mathrm #1}$}}%
\def\ut{ ++(1, 0.29) node  {\raisebox{-0.25em}[0.2em][0em]{\makebox[0mm][c]{$\cdots$}}}  ++(0,-0.29) }%
\def\dt{ ++(1,-0.29) node  {\raisebox{-0.25em}[0.2em][0em]{\makebox[0mm][c]{$\cdots$}}}  ++(0, 0.29) }
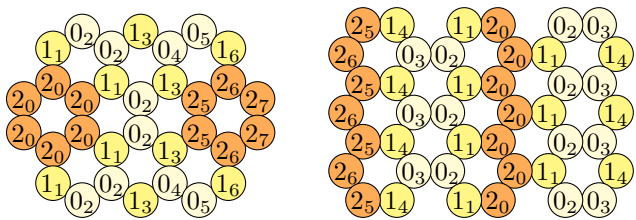
\begin{tikzpicture}[scale=0.65, rotate=0]
\draw (0,-1.732) \0  \db1\ua2\da2\ub3\da4\ua5\db6\0 ;
\draw (0,-3.464) \dc0\uc0\dc0\ub1\da2\ub3\dc5\uc6\dc7;
\draw (0,-5.192) \uc0\dc0\uc0\db1\ua2\db3\uc5\dc6\uc7;
\draw (0,-6.928) \0  \ub1\da2\ua2\db3\ua4\da5\ub6\0 ;
\end{tikzpicture}\qquad\def\0{++(1,0)}%
\def\uo{\ucolor{white}{$\ $}}%
\def\do{\dcolor{white}{$\ $}}%
\def\dxa{\dcolor{white}{$\boldsymbol a$}}%
\def\uxa{\ucolor{white}{$\boldsymbol a$}}%
\def\dxb{\dcolor{white}{$\boldsymbol b$}}%
\def\uxb{\ucolor{white}{$\boldsymbol b$}}%
\def\ua#1{\ucolor{yellow!20!white}{$0_{\mathrm #1}$}}%
\def\da#1{\dcolor{yellow!20!white}{$0_{\mathrm #1}$}}%
\def\ub#1{\ucolor{yellow!60!white}{$1_{\mathrm #1}$}}%
\def\db#1{\dcolor{yellow!60!white}{$1_{\mathrm #1}$}}%
\def\uc#1{\ucolor{orange!65!white}{$2_{\mathrm #1}$}}%
\def\dc#1{\dcolor{orange!65!white}{$2_{\mathrm #1}$}}%
\def\ud#1{\ucolor{red!85!brown!70!white}{$3_{\mathrm #1}$}}%
\def\dd#1{\dcolor{red!85!brown!70!white}{$3_{\mathrm #1}$}}%
\def\ue#1{\ucolor{brown!80!black}{$4_{\mathrm #1}$}}%
\def\de#1{\dcolor{brown!80!black}{$4_{\mathrm #1}$}}%
\def\ut{ ++(1, 0.29) node  {\raisebox{-0.25em}[0.2em][0em]{\makebox[0mm][c]{$\cdots$}}}  ++(0,-0.29) }%
\def\dt{ ++(1,-0.29) node  {\raisebox{-0.25em}[0.2em][0em]{\makebox[0mm][c]{$\cdots$}}}  ++(0, 0.29) }
\begin{tikzpicture}[scale=0.65, rotate=-90]
\draw (0, 1.732)\da3\ub4\da3\ub4\da3\ub4\da3;
\draw (0,-0.000)\ua2\db1\ua2\db1\ua2\db1\ua2;
\draw (0,-1.732)\dc0\uc0\dc0\uc0\dc0\uc0\dc0;
\draw (0,-3.464)\ub1\da2\ub1\da2\ub1\da2\ub1;
\draw (0,-5.192)\db4\ua3\db4\ua3\db4\ua3\db4;
\draw (0,-6.928)\uc5\dc6\uc5\dc6\uc5\dc6\uc5;
\end{tikzpicture}$$

The case {\boldmath${[12{-}201{-}201{-}...]}$} leads to the array $[12{-}201{-}201{-}30]$
$$\def\0{++(1,0)}%
\def\uo{\ucolor{white}{$\ $}}%
\def\do{\dcolor{white}{$\ $}}%
\def\dxa{\dcolor{white}{$\boldsymbol a$}}%
\def\uxa{\ucolor{white}{$\boldsymbol a$}}%
\def\dxb{\dcolor{white}{$\boldsymbol b$}}%
\def\uxb{\ucolor{white}{$\boldsymbol b$}}%
\def\ua#1{\ucolor{yellow!20!white}{$0_{\mathrm #1}$}}%
\def\da#1{\dcolor{yellow!20!white}{$0_{\mathrm #1}$}}%
\def\ub#1{\ucolor{yellow!60!white}{$1_{\mathrm #1}$}}%
\def\db#1{\dcolor{yellow!60!white}{$1_{\mathrm #1}$}}%
\def\uc#1{\ucolor{orange!65!white}{$2_{\mathrm #1}$}}%
\def\dc#1{\dcolor{orange!65!white}{$2_{\mathrm #1}$}}%
\def\ud#1{\ucolor{red!85!brown!70!white}{$3_{\mathrm #1}$}}%
\def\dd#1{\dcolor{red!85!brown!70!white}{$3_{\mathrm #1}$}}%
\def\ue#1{\ucolor{brown!80!black}{$4_{\mathrm #1}$}}%
\def\de#1{\dcolor{brown!80!black}{$4_{\mathrm #1}$}}%
\def\ut{ ++(1, 0.29) node  {\raisebox{-0.25em}[0.2em][0em]{\makebox[0mm][c]{$\cdots$}}}  ++(0,-0.29) }%
\def\dt{ ++(1,-0.29) node  {\raisebox{-0.25em}[0.2em][0em]{\makebox[0mm][c]{$\cdots$}}}  ++(0, 0.29) }
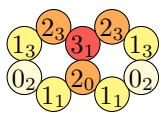
\begin{tikzpicture}[scale=0.65, rotate=0]
\draw (0,-3.464)\db3 \uc3 \dd1 \uc3 \db3;
\draw (0,-5.192)\ua2 \db1 \uc0 \db1 \ua2;
\end{tikzpicture}$$
which is symmetrical to the unfeasible array $[03{-}102{-}102{-}21]$.

The arrays symmetrical to {\boldmath${[12{-}201{-}21]}$} 
and {\boldmath${[12{-}201{-}30]}$} were considered above
and proven to be unfeasible.

The colorings with array {\boldmath${[12{-}21]}$}
are in one-to-one correspondence (inverting the color of every odd node) with the colorings with array
{${[21{-}12]}$}.
Keeping in mind this correspondence, 
it is easy to see that the statement about ${[12{-}21]}$
follows from the statement about ${[21{-}12]}$, 
which will be considered deeply below.

The array symmetrical to {\boldmath${[12{-}30]}$} was proven to be unfeasible.

It remains to consider the last possibility for the values $a_0$, $b_0$. 
Namely, {${[21{-}...]}$}. 
 We note, that we are only interested in the arrays that end with ${-}12]$, as all the other possibilities
 are considered above, up to the symmetry.
 The nodes of color $0$ generate a regular subgraph of degree $2$, 
 i.e., the union of infinite chains and cycles
 (keeping in mind the argument above, 
 we can state the same for the nodes colored by the last color).
We start with considering a partial subcase: 
assume that there is an infinite chain from color-$0$ nodes 
such that no four of these nodes belong to the same hexagon. 
Using arguments as above (case $[12{-}201{-}102{-}...]$), 
we can see that all the coloring is reconstructed, 
and the array has form $[21{-}102{-}201{-}...{-}201{-}12]$
(if the number of colors is even, including the case $[21{-}12]$ of two colors) 
or $[21{-}102{-}201{-}...{-}102{-}21]$ (if it is odd).
Remind that we are not interested in the latter.
$$\def\ut{ ++(1, 0.29) node  {\raisebox{-0.25em}[0.2em][0em]{\makebox[0mm][c]{$\cdots$}}}  ++(0,-0.29) }
\def\dt{ ++(1,-0.29) node  {\raisebox{-0.25em}[0.2em][0em]{\makebox[0mm][c]{$\cdots$}}}  ++(0, 0.29) }
\def\uo{\ucolor{white}{$\ $}}
\def\do{\dcolor{white}{$\ $}}
\def\ua{\ucolor{yellow!20!white}{$0$}}
\def\da{\dcolor{yellow!20!white}{$0$}}
\def\ub{\ucolor{yellow!60!white}{$1$}}
\def\db{\dcolor{yellow!60!white}{$1$}}
\def\0{++(1,0)}
\def\uc{\ucolor{orange!65!white}{$2$}}
\def\dc{\dcolor{orange!65!white}{$2$}}
\def\ud{\ucolor{red!85!brown!70!white}{$3$}}
\def\dd{\dcolor{red!85!brown!70!white}{$3$}}
\def\ue{\ucolor{brown!80!black}{\ }}
\def\de{\dcolor{brown!80!black}{\ }}
\def\uf{\ucolor{black}{\ }}
\def\df{\dcolor{black}{\ }}
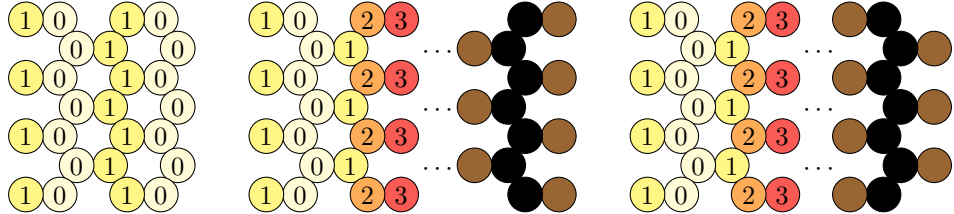
\begin{tikzpicture}[scale=0.65, rotate=-90]
\draw (0,-1.732)\da\ua\da\ua\da\ua\da;
\draw (0,-3.464)\ub\db\ub\db\ub\db\ub;
\draw (0,-5.192)\da\ua\da\ua\da\ua\da;
\draw (0,-6.928)\ub\0 \ub\0 \ub\0 \ub;
\end{tikzpicture}\qquad\input{21-102-201a}\qquad \input{21-102-201b}$$

So, we can assume that the coloring contains 
a chain of four color-$0$ vertices 
from the same hexagon
(and the same is true for the last color):
$$\raisebox{-2em}{\def\0{++(1,0)}%
\def\uo{\ucolor{white}{$\ $}}%
\def\do{\dcolor{white}{$\ $}}%
\def\dxa{\dcolor{white}{$\boldsymbol a$}}%
\def\uxa{\ucolor{white}{$\boldsymbol a$}}%
\def\dxb{\dcolor{white}{$\boldsymbol b$}}%
\def\uxb{\ucolor{white}{$\boldsymbol b$}}%
\def\ua{\ucolor{yellow!20!white}{$0$}}%
\def\da{\dcolor{yellow!20!white}{$0$}}%
\def\ub{\ucolor{yellow!60!white}{$1$}}%
\def\db{\dcolor{yellow!60!white}{$1$}}%
\def\uc{\ucolor{orange!65!white}{$2$}}%
\def\dc{\dcolor{orange!65!white}{$2$}}%
\def\ud{\ucolor{red!85!brown!70!white}{$3$}}%
\def\dd{\dcolor{red!85!brown!70!white}{$3$}}%
\def\ue{\ucolor{brown!80!black}{$4$}}%
\def\de{\dcolor{brown!80!black}{$4$}}%
\def\um{\ucolor{brown!80!black}{\ }}%
\def\dm{\dcolor{brown!80!black}{\ }}%
\def\un{\ucolor{black}{\ }}%
\def\dn{\dcolor{black}{\ }}%
\def\Ua{\Ucolor{yellow!20!white}{$0$}}%
\def\Da{\Dcolor{yellow!20!white}{$0$}}%
\def\Ub{\Ucolor{yellow!60!white}{$1$}}%
\def\Db{\Dcolor{yellow!60!white}{$1$}}%
\def\Uc{\Ucolor{orange!65!white}{$2$}}%
\def\Dc{\Dcolor{orange!65!white}{$2$}}%
\def\Ud{\Ucolor{red!85!brown!70!white}{$3$}}%
\def\Dd{\Dcolor{red!85!brown!70!white}{$3$}}%
\def\ut{ ++(1, 0.29) node  {\raisebox{-0.25em}[0.2em][0em]{\makebox[0mm][c]{$\cdots$}}}  ++(0,-0.29) }%
\def\dt{ ++(1,-0.29) node  {\raisebox{-0.25em}[0.2em][0em]{\makebox[0mm][c]{$\cdots$}}}  ++(0, 0.29) }
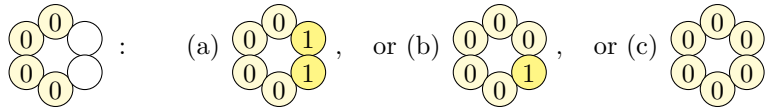
\begin{tikzpicture}[scale=0.65, rotate=0]
\draw (0,-3.464)\da \ua \do ;
\draw (0,-5.192)\ua \da \uo ;
\end{tikzpicture}}:\qquad\mbox{(a)}\raisebox{-2em}{\def\0{++(1,0)}%
\def\uo{\ucolor{white}{$\ $}}%
\def\do{\dcolor{white}{$\ $}}%
\def\dxa{\dcolor{white}{$\boldsymbol a$}}%
\def\uxa{\ucolor{white}{$\boldsymbol a$}}%
\def\dxb{\dcolor{white}{$\boldsymbol b$}}%
\def\uxb{\ucolor{white}{$\boldsymbol b$}}%
\def\ua{\ucolor{yellow!20!white}{$0$}}%
\def\da{\dcolor{yellow!20!white}{$0$}}%
\def\ub{\ucolor{yellow!60!white}{$1$}}%
\def\db{\dcolor{yellow!60!white}{$1$}}%
\def\uc{\ucolor{orange!65!white}{$2$}}%
\def\dc{\dcolor{orange!65!white}{$2$}}%
\def\ud{\ucolor{red!85!brown!70!white}{$3$}}%
\def\dd{\dcolor{red!85!brown!70!white}{$3$}}%
\def\ue{\ucolor{brown!80!black}{$4$}}%
\def\de{\dcolor{brown!80!black}{$4$}}%
\def\um{\ucolor{brown!80!black}{\ }}%
\def\dm{\dcolor{brown!80!black}{\ }}%
\def\un{\ucolor{black}{\ }}%
\def\dn{\dcolor{black}{\ }}%
\def\Ua{\Ucolor{yellow!20!white}{$0$}}%
\def\Da{\Dcolor{yellow!20!white}{$0$}}%
\def\Ub{\Ucolor{yellow!60!white}{$1$}}%
\def\Db{\Dcolor{yellow!60!white}{$1$}}%
\def\Uc{\Ucolor{orange!65!white}{$2$}}%
\def\Dc{\Dcolor{orange!65!white}{$2$}}%
\def\Ud{\Ucolor{red!85!brown!70!white}{$3$}}%
\def\Dd{\Dcolor{red!85!brown!70!white}{$3$}}%
\def\ut{ ++(1, 0.29) node  {\raisebox{-0.25em}[0.2em][0em]{\makebox[0mm][c]{$\cdots$}}}  ++(0,-0.29) }%
\def\dt{ ++(1,-0.29) node  {\raisebox{-0.25em}[0.2em][0em]{\makebox[0mm][c]{$\cdots$}}}  ++(0, 0.29) }
\begin{tikzpicture}[scale=0.65, rotate=0]
\draw (0,-3.464)\da \ua \db ;
\draw (0,-5.192)\ua \da \ub ;
\end{tikzpicture}}, \quad\mbox{or (b)} \raisebox{-2em}{\def\0{++(1,0)}%
\def\uo{\ucolor{white}{$\ $}}%
\def\do{\dcolor{white}{$\ $}}%
\def\dxa{\dcolor{white}{$\boldsymbol a$}}%
\def\uxa{\ucolor{white}{$\boldsymbol a$}}%
\def\dxb{\dcolor{white}{$\boldsymbol b$}}%
\def\uxb{\ucolor{white}{$\boldsymbol b$}}%
\def\ua{\ucolor{yellow!20!white}{$0$}}%
\def\da{\dcolor{yellow!20!white}{$0$}}%
\def\ub{\ucolor{yellow!60!white}{$1$}}%
\def\db{\dcolor{yellow!60!white}{$1$}}%
\def\uc{\ucolor{orange!65!white}{$2$}}%
\def\dc{\dcolor{orange!65!white}{$2$}}%
\def\ud{\ucolor{red!85!brown!70!white}{$3$}}%
\def\dd{\dcolor{red!85!brown!70!white}{$3$}}%
\def\ue{\ucolor{brown!80!black}{$4$}}%
\def\de{\dcolor{brown!80!black}{$4$}}%
\def\um{\ucolor{brown!80!black}{\ }}%
\def\dm{\dcolor{brown!80!black}{\ }}%
\def\un{\ucolor{black}{\ }}%
\def\dn{\dcolor{black}{\ }}%
\def\Ua{\Ucolor{yellow!20!white}{$0$}}%
\def\Da{\Dcolor{yellow!20!white}{$0$}}%
\def\Ub{\Ucolor{yellow!60!white}{$1$}}%
\def\Db{\Dcolor{yellow!60!white}{$1$}}%
\def\Uc{\Ucolor{orange!65!white}{$2$}}%
\def\Dc{\Dcolor{orange!65!white}{$2$}}%
\def\Ud{\Ucolor{red!85!brown!70!white}{$3$}}%
\def\Dd{\Dcolor{red!85!brown!70!white}{$3$}}%
\def\ut{ ++(1, 0.29) node  {\raisebox{-0.25em}[0.2em][0em]{\makebox[0mm][c]{$\cdots$}}}  ++(0,-0.29) }%
\def\dt{ ++(1,-0.29) node  {\raisebox{-0.25em}[0.2em][0em]{\makebox[0mm][c]{$\cdots$}}}  ++(0, 0.29) }
\begin{tikzpicture}[scale=0.65, rotate=0]
\draw (0,-3.464)\da \ua \da ;
\draw (0,-5.192)\ua \da \ub ;
\end{tikzpicture}}, \quad\mbox{or (c)} \raisebox{-2em}{\def\0{++(1,0)}%
\def\uo{\ucolor{white}{$\ $}}%
\def\do{\dcolor{white}{$\ $}}%
\def\dxa{\dcolor{white}{$\boldsymbol a$}}%
\def\uxa{\ucolor{white}{$\boldsymbol a$}}%
\def\dxb{\dcolor{white}{$\boldsymbol b$}}%
\def\uxb{\ucolor{white}{$\boldsymbol b$}}%
\def\ua{\ucolor{yellow!20!white}{$0$}}%
\def\da{\dcolor{yellow!20!white}{$0$}}%
\def\ub{\ucolor{yellow!60!white}{$1$}}%
\def\db{\dcolor{yellow!60!white}{$1$}}%
\def\uc{\ucolor{orange!65!white}{$2$}}%
\def\dc{\dcolor{orange!65!white}{$2$}}%
\def\ud{\ucolor{red!85!brown!70!white}{$3$}}%
\def\dd{\dcolor{red!85!brown!70!white}{$3$}}%
\def\ue{\ucolor{brown!80!black}{$4$}}%
\def\de{\dcolor{brown!80!black}{$4$}}%
\def\um{\ucolor{brown!80!black}{\ }}%
\def\dm{\dcolor{brown!80!black}{\ }}%
\def\un{\ucolor{black}{\ }}%
\def\dn{\dcolor{black}{\ }}%
\def\Ua{\Ucolor{yellow!20!white}{$0$}}%
\def\Da{\Dcolor{yellow!20!white}{$0$}}%
\def\Ub{\Ucolor{yellow!60!white}{$1$}}%
\def\Db{\Dcolor{yellow!60!white}{$1$}}%
\def\Uc{\Ucolor{orange!65!white}{$2$}}%
\def\Dc{\Dcolor{orange!65!white}{$2$}}%
\def\Ud{\Ucolor{red!85!brown!70!white}{$3$}}%
\def\Dd{\Dcolor{red!85!brown!70!white}{$3$}}%
\def\ut{ ++(1, 0.29) node  {\raisebox{-0.25em}[0.2em][0em]{\makebox[0mm][c]{$\cdots$}}}  ++(0,-0.29) }%
\def\dt{ ++(1,-0.29) node  {\raisebox{-0.25em}[0.2em][0em]{\makebox[0mm][c]{$\cdots$}}}  ++(0, 0.29) }
\begin{tikzpicture}[scale=0.65, rotate=0]
\draw (0,-3.464)\da \ua \da ;
\draw (0,-5.192)\ua \da \ua ;
\end{tikzpicture}}$$

The first subcase of ${[21{-}...]}$ is {\boldmath${[21{-}102{-}...]}$}.
We see that (a) and (b) are inadmissible, as $a_1<1$ and $c_1<2$ in the considered prefix.
Also, we can easily find that $a_2>0$:
$$\def\ut{ ++(1, 0.29) node  {\raisebox{-0.25em}[0.2em][0em]{\makebox[0mm][c]{$\cdots$}}}  ++(0,-0.29) }
\def\dt{ ++(1,-0.29) node  {\raisebox{-0.25em}[0.2em][0em]{\makebox[0mm][c]{$\cdots$}}}  ++(0, 0.29) }
\def\uo{\ucolor{white}{$\ $}}
\def\do{\dcolor{white}{$\ $}}
\def\ux#1{\ucolor{white}{$#1$}}
\def\dx#1{\dcolor{white}{$#1$}}
\def\ua#1{\ucolor{yellow!20!white}{$0_{\mathrm{#1}}$}}
\def\da#1{\dcolor{yellow!20!white}{$0_{\mathrm{#1}}$}}
\def\ub#1{\ucolor{yellow!60!white}{$1_{\mathrm{#1}}$}}
\def\db#1{\dcolor{yellow!60!white}{$1_{\mathrm{#1}}$}}
\def\uc#1{\ucolor{orange!65!white}{$2_{\mathrm{#1}}$}}
\def\dc#1{\dcolor{orange!65!white}{$2_{\mathrm{#1}}$}}
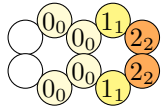
\begin{tikzpicture}[scale=0.65, rotate=0]
\draw (0,-3.464)\do \ua0 \da0 \ub1 \dc2;
\draw (0,-5.192)\uo \da0 \ua0 \db1 \uc2;
\end{tikzpicture}$$
Further, all possibilities are divided into the following four subcases
(excluding the arrays that are not finishing with ${-}12$):

$[21{-}102{-}111{-}102{-}...]$ is not feasible by Lemma~\ref{l:111-102};
$[21{-}102{-}111{-}111{-}...]$ and $[21{-}102{-}111{-}12]$ are not feasible (see the figures below);
$[21{-}102{-}12]$ corresponds to a unique coloring, up to equivalence:
$$\def\uo{\ucolor{white}{$\ $}}
\def\do{\dcolor{white}{$\ $}}
\def\ux#1{\ucolor{white}{$#1$}}
\def\dx#1{\dcolor{white}{$#1$}}
\def\ua#1{\ucolor{yellow!20!white}{$0_{\mathrm{#1}}$}}
\def\da#1{\dcolor{yellow!20!white}{$0_{\mathrm{#1}}$}}
\def\ub#1{\ucolor{yellow!60!white}{$1_{\mathrm{#1}}$}}
\def\db#1{\dcolor{yellow!60!white}{$1_{\mathrm{#1}}$}}
\def\0{++(1,0)}
\def\uc#1{\ucolor{orange!65!white}{$2_{\mathrm{#1}}$}}
\def\dc#1{\dcolor{orange!65!white}{$2_{\mathrm{#1}}$}}
\def\ud#1{\ucolor{red!85!brown!70!white}{$3_{\mathrm{#1}}$}}
\def\dd#1{\dcolor{red!85!brown!70!white}{$3_{\mathrm{#1}}$}}
\def\ue#1{\ucolor{brown!80!black}{$4_{\mathrm{#1}}$}}
\def\de#1{\dcolor{brown!80!black}{$4_{\mathrm{#1}}$}}
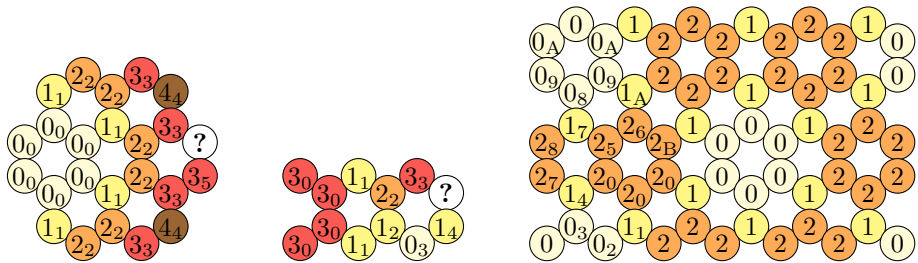
\begin{tikzpicture}[scale=0.65, rotate=0]
\draw (0,-0.000)\0  \db1\uc2\dc2\ud3\de4;
\draw (0,-1.732)\da0\ua0\da0\ub1\dc2\ud3\dcolor{white}{\bf?};
\draw (0,-3.464)\ua0\da0\ua0\db1\uc2\dd3\ud5;
\draw (0,-5.192)\0  \ub1\dc2\uc2\dd3\ue4;
\end{tikzpicture}\qquad\def\ua#1{\ucolor{yellow!20!white}{$0_{\mathrm{#1}}$}}
\def\da#1{\dcolor{yellow!20!white}{$0_{\mathrm{#1}}$}}
\def\ub#1{\ucolor{yellow!60!white}{$1_{\mathrm{#1}}$}}
\def\db#1{\dcolor{yellow!60!white}{$1_{\mathrm{#1}}$}}
\def\uc#1{\ucolor{orange!65!white}{$2_{\mathrm{#1}}$}}
\def\dc#1{\dcolor{orange!65!white}{$2_{\mathrm{#1}}$}}
\def\ud#1{\ucolor{red!85!brown!70!white}{$3_{\mathrm{#1}}$}}
\def\dd#1{\dcolor{red!85!brown!70!white}{$3_{\mathrm{#1}}$}}
\begin{tikzpicture}[scale=0.65, rotate=0]
\draw (0,-3.464) \ud0\dd0\ub1\dc2\ud3\dcolor{white}{\bf?};
\draw (0,-5.192) \dd0\ud0\db1\ub2\da3\ub4;
\end{tikzpicture}\qquad\def\ut{ ++(1, 0.29) node  {\raisebox{-0.25em}[0.2em][0em]{\makebox[0mm][c]{$\cdots$}}}  ++(0,-0.29) }
\def\dt{ ++(1,-0.29) node  {\raisebox{-0.25em}[0.2em][0em]{\makebox[0mm][c]{$\cdots$}}}  ++(0, 0.29) }
\def\ua#1{\ucolor{yellow!20!white}{$0_{\mathrm{#1}}$}}
\def\da#1{\dcolor{yellow!20!white}{$0_{\mathrm{#1}}$}}
\def\uA{\ucolor{yellow!20!white}{$0$}}
\def\dA{\dcolor{yellow!20!white}{$0$}}
\def\ub#1{\ucolor{yellow!60!white}{$1_{\mathrm{#1}}$}}
\def\db#1{\dcolor{yellow!60!white}{$1_{\mathrm{#1}}$}}
\def\uB{\ucolor{yellow!60!white}{$1$}}
\def\dB{\dcolor{yellow!60!white}{$1$}}
\def\0{++(1,0)}
\def\uc#1{\ucolor{orange!65!white}{$2_{\mathrm{#1}}$}}
\def\dc#1{\dcolor{orange!65!white}{$2_{\mathrm{#1}}$}}
\def\uC{\ucolor{orange!65!white}{$2$}}
\def\dC{\dcolor{orange!65!white}{$2$}}
\begin{tikzpicture}[scale=0.65, rotate=0]
\draw (0, 3.464)\da A\uA \da A\uB  \dC  \uC \dC \uB \dC\uC\dC\uB\dA ;
\draw (0, 1.732)\ua9 \da8\ua9 \db A\uC  \dC \uC \dB \uC\dC\uC\dB\uA ;
\draw (0,-0.000)\dc8 \ub7\dc5 \uc6 \dc B\uB \dA \uA \dA\uB\dC\uC\dC ;
\draw (0,-1.732)\uc7 \db4\uc0 \dc0 \uc0 \dB \uA \dA \uA\dB\uC\dC\uC ;
\draw (0,-3.464)\dA  \ua3\da2 \ub1 \dC  \uC \dC \uB \dC\uC\dC\uB\dA ;
\end{tikzpicture}$$
(on the step 6 of the last consideration, we use the symmetry and the obvious fact that the two remaining 
uncolored nodes in the corresponding hexagon cannot have color $1$ simultaneously, as $a_1=0$).

The next subcase of ${[21{-}...]}$ is {\boldmath${[21{-}111{-}...]}$}.
The following two pictures demon\-strate that two special kind of chains 
from color-$0$ nodes lead to arrays that do not finish with ${-}12$:
$$\def\uo{\ucolor{white}{$\ $}}
\def\do{\dcolor{white}{$\ $}}
\def\ux#1{\ucolor{white}{$#1$}}
\def\dx#1{\dcolor{white}{$#1$}}
\def\ua#1{\ucolor{yellow!20!white}{$0_{\mathrm{#1}}$}}
\def\da#1{\dcolor{yellow!20!white}{$0_{\mathrm{#1}}$}}
\def\ub#1{\ucolor{yellow!60!white}{$1_{\mathrm{#1}}$}}
\def\db#1{\dcolor{yellow!60!white}{$1_{\mathrm{#1}}$}}
\def\0{++(1,0)}
\def\dc#1{\dcolor{orange!65!white}{$2_{\mathrm{#1}}$}}
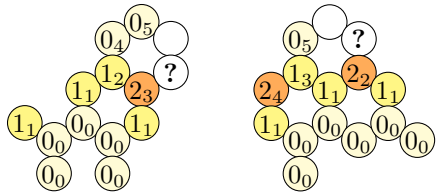
\begin{tikzpicture}[scale=0.65, rotate=0]
\draw (0, 1.732)\0  \0  \0  \da4\ua5\do;
\draw (0,-0.000)\0  \0  \db1\ub2\dc3\ucolor{white}{\bf?};
\draw (0,-1.732)\ub1\da0\ua0\da0\ub1 ;
\draw (0,-3.464)\0  \ua0\0  \ua0;
\end{tikzpicture}\qquad\def\uo{\ucolor{white}{$\ $}}
\def\do{\dcolor{white}{$\ $}}
\def\ua#1{\ucolor{yellow!20!white}{$0_{\mathrm{#1}}$}}
\def\da#1{\dcolor{yellow!20!white}{$0_{\mathrm{#1}}$}}
\def\ub#1{\ucolor{yellow!60!white}{$1_{\mathrm{#1}}$}}
\def\db#1{\dcolor{yellow!60!white}{$1_{\mathrm{#1}}$}}
\def\0{++(1,0)}
\def\uc#1{\ucolor{orange!65!white}{$2_{\mathrm{#1}}$}}
\def\dc#1{\dcolor{orange!65!white}{$2_{\mathrm{#1}}$}}
\begin{tikzpicture}[scale=0.65, rotate=0]
\draw (0,-0.000) \0  \da5\uo \dcolor{white}{\bf?};
\draw (0,-1.732) \dc4\ub3\db1\uc2\db1;
\draw (0,-3.464) \ub1\da0\ua0\da0\ua0\da0;
\draw (0,-5.192) \0  \ua0;
\end{tikzpicture}$$
In particular, this means that four consequent color-$0$ nodes in the same hexagon
are uniquely continued to an infinite color-$0$ chain, like in the figure below. 
Arguing as in the case $[12{-}111{-}111{-}...]$, 
we get a new series of parameters $[21{-}111{-}...{-}111{-}12]$: 
$$\def\ut{ ++(1, 0.29) node  {\raisebox{-0.25em}[0.2em][0em]{\makebox[0mm][c]{$\cdots$}}}  ++(0,-0.29) }%
\def\dt{ ++(1,-0.29) node  {\raisebox{-0.25em}[0.2em][0em]{\makebox[0mm][c]{$\cdots$}}}  ++(0, 0.29) }%
\def\uo{\ucolor{white}{$\ $}}%
\def\do{\dcolor{white}{$\ $}}%
\def\ux{\ucolor{white}{$$}}%
\def\dx{\dcolor{white}{$$}}%
\def\ua{\ucolor{yellow!20!white}{$0$}}%
\def\da{\dcolor{yellow!20!white}{$0$}}%
\def\ub{\ucolor{yellow!60!white}{$1$}}%
\def\db{\dcolor{yellow!60!white}{$1$}}%
\def\0{++(1,0)}%
\def\uc{\ucolor{orange!65!white}{$2$}}%
\def\dc{\dcolor{orange!65!white}{$2$}}%
\def\ud{\ucolor{red!85!brown!70!white}{$3$}}%
\def\dd{\dcolor{red!85!brown!70!white}{$3$}}%
\def\ue{\ucolor{brown!80!black}{$\ $}}%
\def\de{\dcolor{brown!80!black}{$\ $}}%
\def\uf{\ucolor{black}{$\ $}}%
\def\df{\dcolor{black}{$\ $}}%
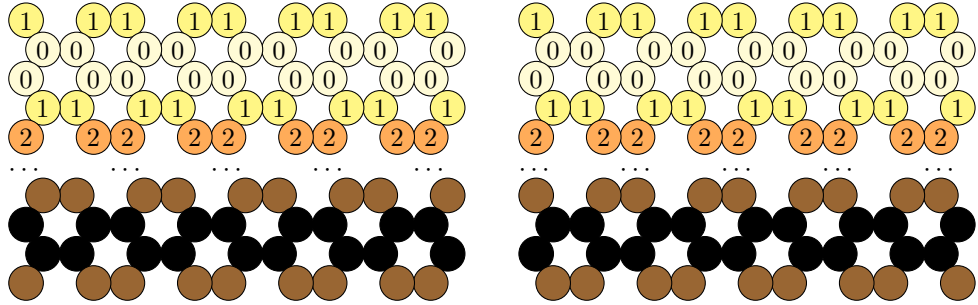
\begin{tikzpicture}[scale=0.65, rotate=-90]
\draw (0, 5.192) \db\ua\da\ub\dc\dt\ue\df\uf\de;
\draw (0, 3.464) \ub\da\ua\db\uc\0 \de\uf\df\ue;
\draw (0, 1.732) \db\ua\da\ub\dc\dt\ue\df\uf\de;
\draw (0,-0.000) \ub\da\ua\db\uc\0 \de\uf\df\ue;
\draw (0,-1.732) \db\ua\da\ub\dc\dt\ue\df\uf\de;
\draw (0,-3.464) \ub\da\ua\db\uc\0 \de\uf\df\ue;
\draw (0,-5.192) \db\ua\da\ub\dc\dt\ue\df\uf\de;
\draw (0,-6.928) \ub\da\ua\db\uc\0 \de\uf\df\ue;
\draw (0,-8.660) \db\ua\da\ub\dc\dt\ue\df\uf\de;
\end{tikzpicture}\qquad\input{21-111-b}$$

{\boldmath${[21{-}12]}$}.
We first consider the subcase (c), when there is a hexagon colored into one color, say $0$.
The rest of the coloring is colored uniquely, namely, 
the color of a node equals $\lfloor\frac{d+1}2\rfloor\bmod2$ 
where $d$ is the distance from the node
to the origin hexagon. 
By inverting the color of the odd vertices, 
we get the corresponding coloring with the array ${[12{-}21]}$.
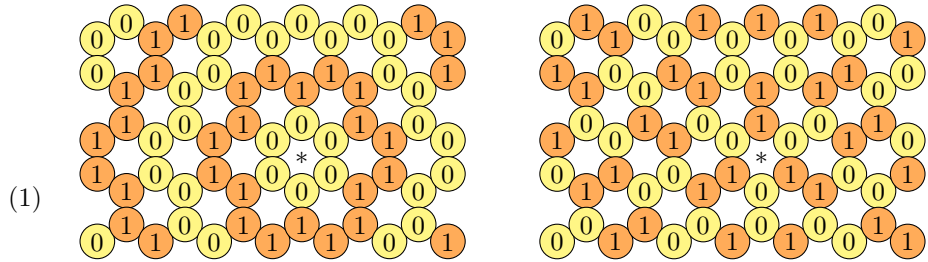
\begin{equation}\label{eq:radial}
 \raisebox{-2em}{\input{21-12a}\qquad\input{12-21a}}
\end{equation}
We also note that in the last case, 
there is a triple of color-$0$
at distance $2$ from each other; 
so, the set of color-$0$ nodes cannot be split to form a coloring with 
the intersection array $[12{-}111{-}21]$ (the same is true for the color-$1$ nodes).

It remains to consider the subcase (a), assuming that no hexagon 
colored by one color. Let $o$, $x$, $y$, and $xz$ be colored by $0$. Then $yz$, $yzx$ are colored by $1$ (otherwize we have a hexagon colored by $0$); hence the nodes $yzy$ and $yzxz$ have the same color. By induction, the nodes
$(yz)^i$, $(yz)^ix$, $(yz)^iy$, $(yz)^ixz$ are colored by $i\bmod 2$, $i=0,1,2,...$. The same holds for the nodes
$(zy)^i$, $(zy)^ix$, $(zy)^iy$, $(zy)^ixz$.

\noindent\mbox{}\hspace{-6ex}$\raisebox{-2em}{\def\ua#1{\ucolor{yellow!60!white}{$0_{\mathrm{#1}}$}}
\def\da#1{\dcolor{yellow!60!white}{$0_{\mathrm{#1}}$}}
\def\db#1{\dcolor{orange!65!white}{$1_{\mathrm{#1}}$}}
\def\ub#1{\ucolor{orange!65!white}{$1_{\mathrm{#1}}$}}
\def\uO{\ucolor{white}{$ $}}
\def\dO{\dcolor{white}{$ $}}
\def\uA{\ucolor{yellow!60!white}{$0$}}
\def\dA{\dcolor{yellow!60!white}{$0$}}
\def\uB{\ucolor{orange!65!white}{$1$}}
\def\dB{\dcolor{orange!65!white}{$1$}}
\def\dAA{\dcolor{white}{$\boldsymbol a$}}
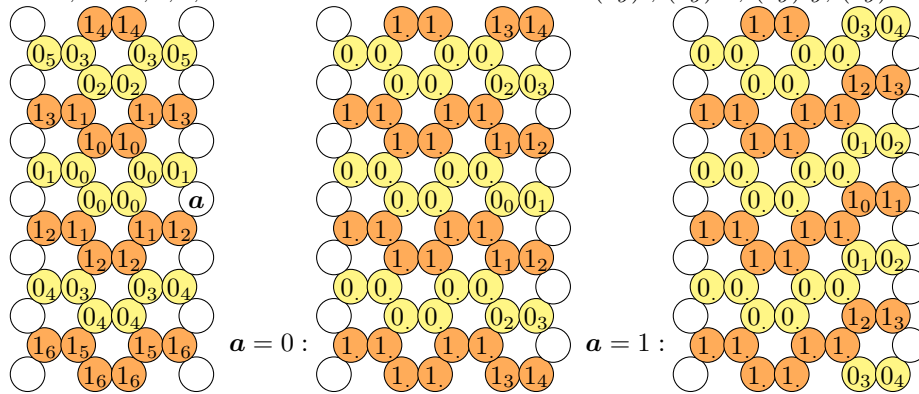
\begin{tikzpicture}[scale=0.65, rotate=90]
\draw (0, 1.732)\uO  \db6\uO  \da4 \uO  \db2\uO \da1\uO \db3\uO \da5\uO ;
\draw (0,-0.000)\db6 \ub5\da4 \ua3 \db2 \ub1\da0\ua0\db0\ub1\da2\ua3\db4;
\draw (0,-1.732)\ub6 \db5\ua4 \da3 \ub2 \db1\ua0\da0\ub0\db1\ua2\da3\ub4;
\draw (0,-3.464)\dO  \ub6\dO  \ua4 \dO  \ub2\dAA\ua1\dO \ub3\dO \ua5\dO ;
\end{tikzpicture}}\  \boldsymbol a=0: \raisebox{-2em}{\input{21-12a0}}\  \boldsymbol a=1: \raisebox{-2em}{\input{21-12a1}}$

We state that for every node $v$, the node $vx$ has the same color, while the node $vyz$ has the different color.
We already know that this holds for the nodes of the set
$I_o$, where $I_v=\{v(yz)^i,v(yz)^ix,v(zy)^i,v(zy)^ix : i=0,1,2,...\}$. In general, this statement is easy to prove by induction on the distance from $I_o$, see the figures above (subcases $\boldsymbol a=0$ and $\boldsymbol a=1$) for the example of the induction step.

As a consequence, we see that the colors of all nodes of the set
$I_v$ for some $v$ is uniquely determined by the color of any node from this set; so, there is only one way to color $I_v$, up to the inversion of the coloring.
Finaly, we find that any coloring with array $[21{-}12]$
either have a hexagon colored into one color (in which case the coloring is unique, up to equivalence and inversion of the colors, see (\ref{eq:radial})) or is obtained from the coloring
$$\input{21-12c}$$
by inversion of the colors of some sets $I_v$ (and may be by rotation). For example:
$$\input{21-12d}$$
Since there are infinite number of such sets, the number of the colorings is infinite.

The same conclusion can be done for the linked array
$[12{-}21]$: every coloring is either equivalent to the coloring (\ref{eq:radial}) or obtained from the coloring
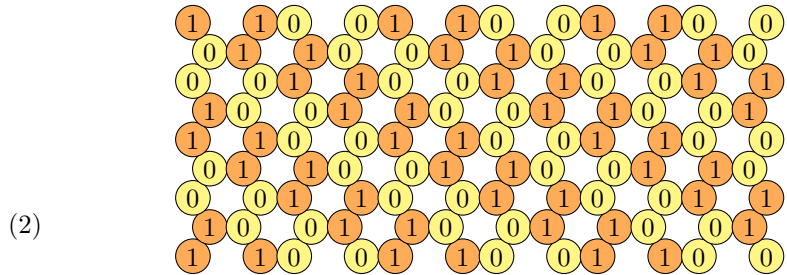
\begin{equation}\label{eq:1221}
\raisebox{-2em}{\input{12-21c}}
\end{equation}
by inversion of the colors of some sets $I_v$. For example:
$$\input{12-21d}$$
The number of the colorings is infinite.

Another linked array is $[12{-}111{-}21]$. 
It is easy to see that the color $0$ of every coloring with array $[12{-}21]$ except (\ref{eq:radial}) can be split into two colors to form a coloring with array $[12{-}111{-}21]$. However, there are two different subcases. If the coloring contains the fragment
$$\def\ud{\ucolor{red!85!brown!70!white}{$2$}}
\def\dd{\dcolor{red!85!brown!70!white}{$2$}}
\def\ua{\ucolor{yellow!60!white}{$0$}}
\def\da{\dcolor{yellow!60!white}{$0$}}
\def\Ua{\ucolor{yellow!60!white}{ }}
\def\Da{\dcolor{yellow!60!white}{ }}
\def\uA{\ucolor{yellow!60!white}{$\boldsymbol b$}}
\def\dA{\dcolor{yellow!60!white}{$\boldsymbol b$}}
\def\ub{\ucolor{orange!65!white}{$1$}}
\def\db{\dcolor{orange!65!white}{$1$}}
\def\uc{\ucolor{yellow!20!white}{$0$}}
\def\dc{\dcolor{yellow!20!white}{$0$}}
\def\Uc{\ucolor{yellow!20!white}{ }}
\def\Dc{\dcolor{yellow!20!white}{ }}
\def\dC{\dcolor{yellow!20!white}{$\boldsymbol a$}}
\def\uC{\ucolor{yellow!20!white}{$\boldsymbol a$}}
\def\uO{\ucolor{white}{$ $}}
\def\dO{\dcolor{white}{$ $}}
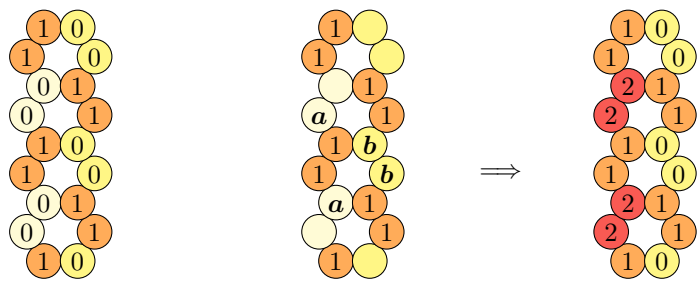
\begin{tikzpicture}[scale=0.65, rotate=90]
\draw (0, 26.928)\db\uc\dc\ub\db\uc\dc\ub\db; 
\draw (0, 25.192)\ua\db\ub\da\ua\db\ub\da\ua; 
\draw (0, 16.928)\db\Uc\dC\ub\db\uC\Dc\ub\db; 
\draw (0, 15.192)\Ua\db\ub\dA\uA\db\ub\Da\Ua; 
\draw (4,  11   ) node {$\Longrightarrow$}; 
\draw (0,  6.928)\db\ud\dd\ub\db\ud\dd\ub\db; 
\draw (0,  5.192)\ua\db\ub\da\ua\db\ub\da\ua; 
\end{tikzpicture}$$
then the splitting is unique 
(we see that the nodes marked by $\boldsymbol a$ 
must have the same color in the new coloring;
hence, all ``left'' $0$s of the fragment will have the same color; in the rest of the coloring, $0$s are split uniquely):
$$\input{12-111-21d}$$
The only coloring with array $[12{-}21]$ avoiding such fragment is (\ref{eq:1221}).
We see that for every set 
$\{v(xz)^i,v(xz)^iy : i=0,\pm 1,\pm 2,...\}$, where
$v$ is colored by $0$, we can independently choose 
one of two variants to color half of the $0$s into new color $2$:
$$\input{12-111-21e}$$
This gives another infinite set of colorings with array $[12{-}111{-}21]$. 

In the case {\boldmath${[21{-}201{-}...]}$}, we see that the number of colors is smaller than $4$; 
$$\def\0{++(1,0)}%
\def\uo{\ucolor{white}{$\ $}}%
\def\do{\dcolor{white}{$\ $}}%
\def\dxa{\dcolor{white}{$\boldsymbol a$}}%
\def\uxa{\ucolor{white}{$\boldsymbol a$}}%
\def\dxb{\dcolor{white}{$\boldsymbol b$}}%
\def\uxb{\ucolor{white}{$\boldsymbol b$}}%
\def\ua#1{\ucolor{yellow!20!white}{$0_{\mathrm #1}$}}%
\def\da#1{\dcolor{yellow!20!white}{$0_{\mathrm #1}$}}%
\def\ub#1{\ucolor{yellow!60!white}{$1_{\mathrm #1}$}}%
\def\db#1{\dcolor{yellow!60!white}{$1_{\mathrm #1}$}}%
\def\uc#1{\ucolor{orange!65!white}{$2_{\mathrm #1}$}}%
\def\dc#1{\dcolor{orange!65!white}{$2_{\mathrm #1}$}}%
\def\ud#1{\ucolor{red!85!brown!70!white}{$3_{\mathrm #1}$}}%
\def\dd#1{\dcolor{red!85!brown!70!white}{$3_{\mathrm #1}$}}%
\def\ue#1{\ucolor{brown!80!black}{$4_{\mathrm #1}$}}%
\def\de#1{\dcolor{brown!80!black}{$4_{\mathrm #1}$}}%
\def\ut{ ++(1, 0.29) node  {\raisebox{-0.25em}[0.2em][0em]{\makebox[0mm][c]{$\cdots$}}}  ++(0,-0.29) }%
\def\dt{ ++(1,-0.29) node  {\raisebox{-0.25em}[0.2em][0em]{\makebox[0mm][c]{$\cdots$}}}  ++(0, 0.29) }
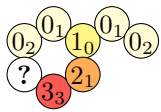
\begin{tikzpicture}[scale=0.65, rotate=0]
\draw (0,-0.000)\da2\ua1\db0\ua1\da2;
\draw (0,-1.732)\uq\dd3\uc1;
\end{tikzpicture}$$
but the three arrays with $3$ colors were considered above, up to the symmetry. 
The same is true for the remaining subcases {\boldmath${[21{-}21]}$} and {\boldmath${[21{-}30]}$}.

\renewcommand\appendixname{Appendix}
\appendix
\renewcommand\appendixname{Appendix}
\section*{Appendices}
In the appendices, we list colorings for all feasible arrays of
distance regular colorings of the infinite hexagonal grid.
If the array corresponds to a finite number of colorings, 
we show all nonequivalent ones.
If the number of colorings is infinite, 
each class obtained from one coloring in accordance with Lemma~\ref{l:infin}
is illustrated by only one picture; in this case, the set $X$ is emphasized by the bold cycles.
The small monochromatic pictures illustrate corresponding completely regular codes. 
If the two codes corresponding to the same coloring are equivalent, we picture only one of them.
\section{The colorings: two colors}   

\noindent\rotatebox{90}{$[03{-}30]$:}
\input{03-30fin}
\parbox[b]{1cm}{\def\uclr#1{ ++(1, 0.29) node [thin,circle,draw=black, fill=#1, thin, minimum size=6.5, inner sep=0] {}  ++(0,-0.29) }
\def\dclr#1{ ++(1,-0.29) node [thin,circle,draw=black, fill=#1, thin, minimum size=6.5, inner sep=0] {}  ++(0, 0.29) }
\def\ua{\uclr{white!90!black}}
\def\da{\dclr{white!90!black}}
\def\ub{\uclr{white!90!black}}
\def\db{\dclr{white!90!black}}
\def\uc{\uclr{white!90!black}}
\def\dc{\dclr{white!90!black}}
\def\ud{\uclr{white!90!black}}
\def\dd{\dclr{white!90!black}}
\def\ue{\uclr{white!90!black}}
\def\de{\dclr{white!90!black}}
\def\ua{\uclr{black}}
\def\da{\dclr{black}}
\def\Da{\da}
\def\Db{\db}
\def\Ua{\ua}
\def\Ub{\ub}
\def\Dc{\dc}
\def\Uc{\uc}
\def\Dd{\dd}
\def\Ud{\ud}
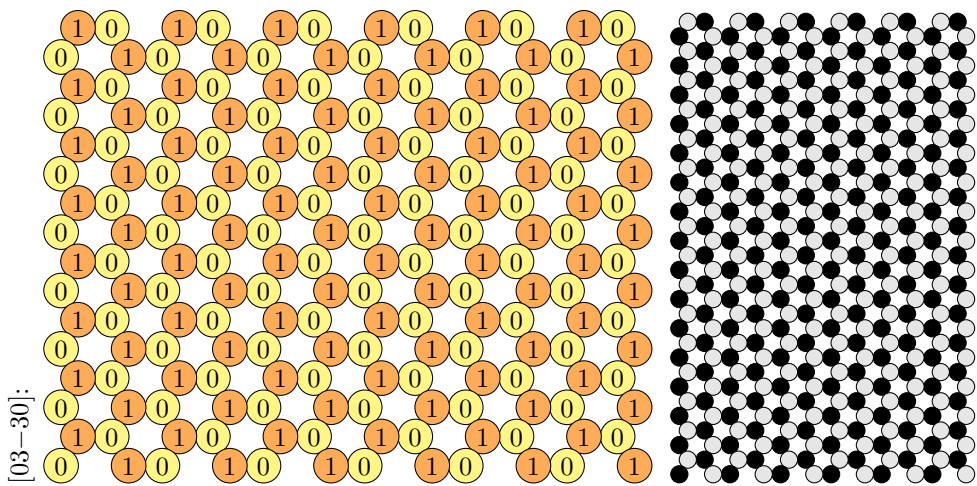
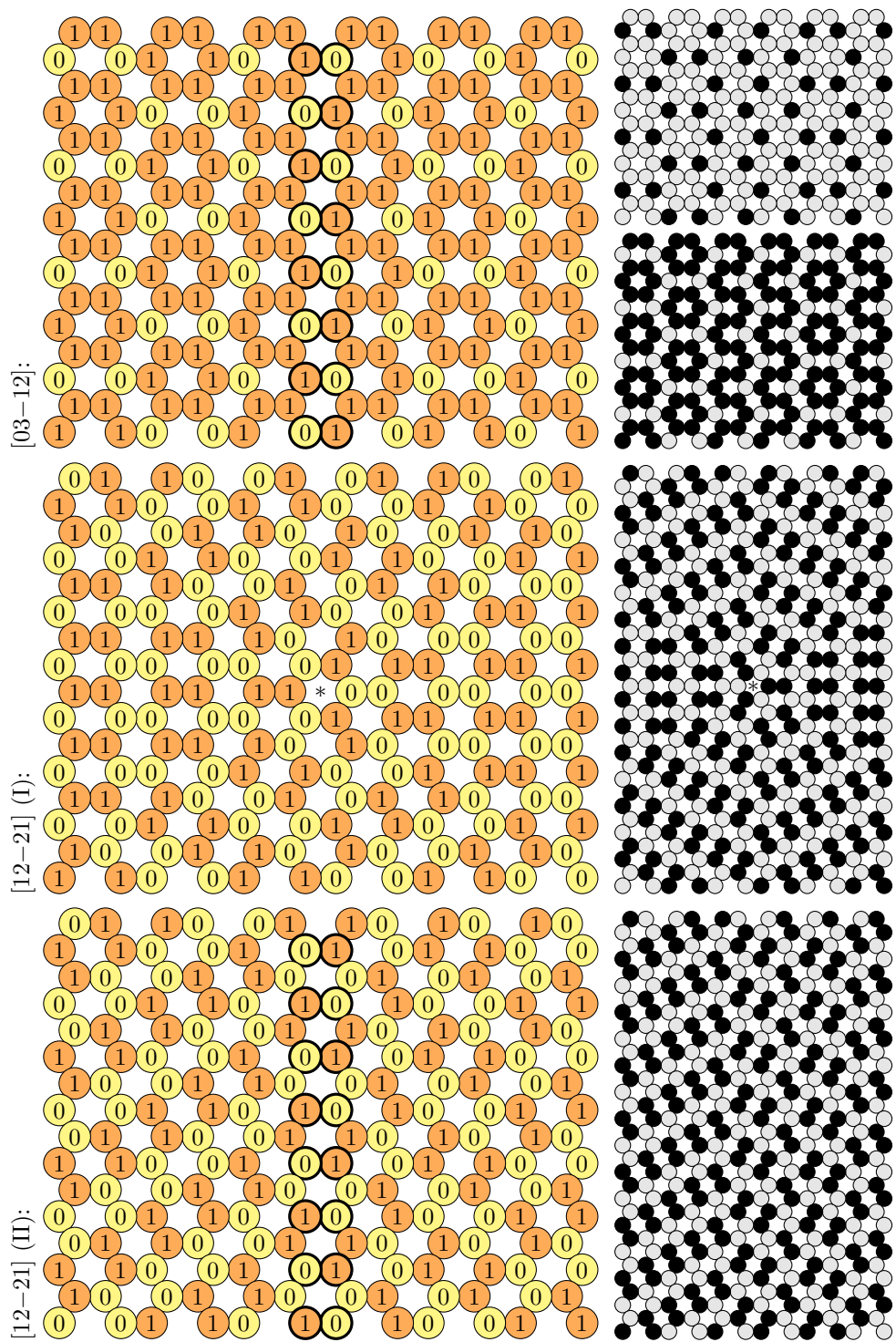
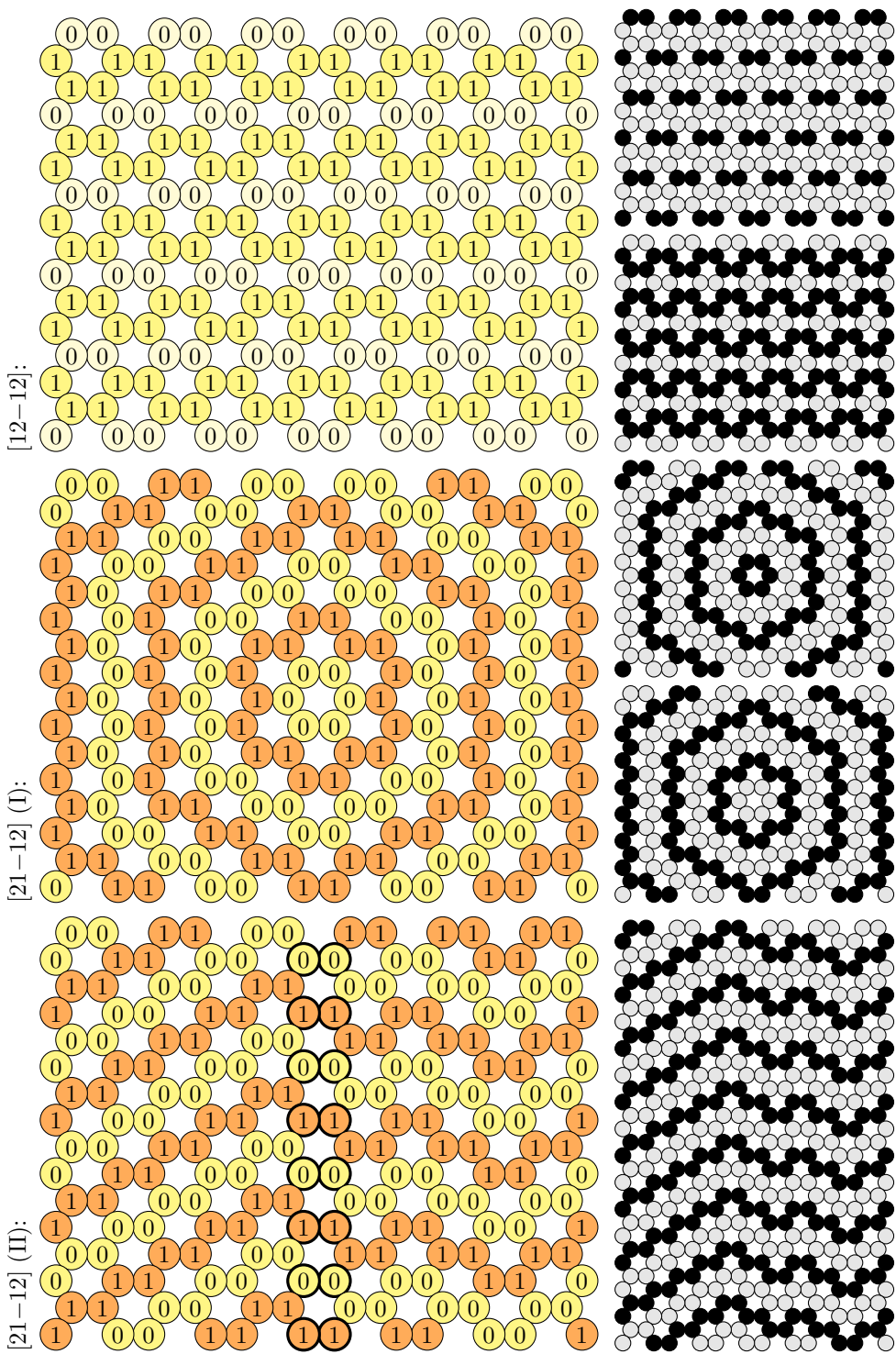
\begin{tikzpicture}[scale=0.325, rotate=-90]
\abcabcabc
\end{tikzpicture}
}

\noindent\rotatebox{90}{\noindent$[03{-}12]$:} 
\input{03-12fin}
\parbox[b]{1cm}{\def\uclr#1{ ++(1, 0.29) node [thin,circle,draw=black, fill=#1, thin, minimum size=6.5, inner sep=0] {}  ++(0,-0.29) }
\def\dclr#1{ ++(1,-0.29) node [thin,circle,draw=black, fill=#1, thin, minimum size=6.5, inner sep=0] {}  ++(0, 0.29) }
\def\ua{\uclr{white!90!black}}
\def\da{\dclr{white!90!black}}
\def\ub{\uclr{white!90!black}}
\def\db{\dclr{white!90!black}}
\def\uc{\uclr{white!90!black}}
\def\dc{\dclr{white!90!black}}
\def\ud{\uclr{white!90!black}}
\def\dd{\dclr{white!90!black}}
\def\ue{\uclr{white!90!black}}
\def\de{\dclr{white!90!black}}
\def\ua{\uclr{black}}
\def\da{\dclr{black}}
\def\Da{\da}
\def\Db{\db}
\def\Ua{\ua}
\def\Ub{\ub}
\def\Dc{\dc}
\def\Uc{\uc}
\def\Dd{\dd}
\def\Ud{\ud}
\begin{tikzpicture}[scale=0.325, rotate=-90]
\abcabcabc
\end{tikzpicture}
\\ \def\uclr#1{ ++(1, 0.29) node [thin,circle,draw=black, fill=#1, thin, minimum size=6.5, inner sep=0] {}  ++(0,-0.29) }
\def\dclr#1{ ++(1,-0.29) node [thin,circle,draw=black, fill=#1, thin, minimum size=6.5, inner sep=0] {}  ++(0, 0.29) }
\def\ua{\uclr{white!90!black}}
\def\da{\dclr{white!90!black}}
\def\ub{\uclr{white!90!black}}
\def\db{\dclr{white!90!black}}
\def\uc{\uclr{white!90!black}}
\def\dc{\dclr{white!90!black}}
\def\ud{\uclr{white!90!black}}
\def\dd{\dclr{white!90!black}}
\def\ue{\uclr{white!90!black}}
\def\de{\dclr{white!90!black}}
\def\ub{\uclr{black}}
\def\db{\dclr{black}}
\def\Da{\da}
\def\Db{\db}
\def\Ua{\ua}
\def\Ub{\ub}
\def\Dc{\dc}
\def\Uc{\uc}
\begin{tikzpicture}[scale=0.325, rotate=-90]
\abcabcabc
\end{tikzpicture}
}
   
\noindent\rotatebox{90}{\noindent$[12{-}21]$ (I):}
\input{12-21fin}
\parbox[b]{1cm}{\def\uclr#1{ ++(1, 0.29) node [thin,circle,draw=black, fill=#1, thin, minimum size=6.5, inner sep=0] {}  ++(0,-0.29) }
\def\dclr#1{ ++(1,-0.29) node [thin,circle,draw=black, fill=#1, thin, minimum size=6.5, inner sep=0] {}  ++(0, 0.29) }
\def\ua{\uclr{white!90!black}}
\def\da{\dclr{white!90!black}}
\def\ub{\uclr{white!90!black}}
\def\db{\dclr{white!90!black}}
\def\uc{\uclr{white!90!black}}
\def\dc{\dclr{white!90!black}}
\def\ud{\uclr{white!90!black}}
\def\dd{\dclr{white!90!black}}
\def\ue{\uclr{white!90!black}}
\def\de{\dclr{white!90!black}}
\def\ua{\uclr{black}}
\def\da{\dclr{black}}
\def\Da{\da}
\def\Db{\db}
\def\Ua{\ua}
\def\Ub{\ub}
\def\Dc{\dc}
\def\Uc{\uc}
\def\Dd{\dd}
\def\Ud{\ud}
\begin{tikzpicture}[scale=0.325, rotate=-90]
\abcabcabc
\end{tikzpicture}
}

\noindent\rotatebox{90}{\noindent$[12{-}21]$ (II):}
\input{12-21fin9}
\parbox[b]{1cm}{\def\uclr#1{ ++(1, 0.29) node [thin,circle,draw=black, fill=#1, thin, minimum size=6.5, inner sep=0] {}  ++(0,-0.29) }
\def\dclr#1{ ++(1,-0.29) node [thin,circle,draw=black, fill=#1, thin, minimum size=6.5, inner sep=0] {}  ++(0, 0.29) }
\def\ua{\uclr{white!90!black}}
\def\da{\dclr{white!90!black}}
\def\ub{\uclr{white!90!black}}
\def\db{\dclr{white!90!black}}
\def\uc{\uclr{white!90!black}}
\def\dc{\dclr{white!90!black}}
\def\ud{\uclr{white!90!black}}
\def\dd{\dclr{white!90!black}}
\def\ue{\uclr{white!90!black}}
\def\de{\dclr{white!90!black}}
\def\ua{\uclr{black}}
\def\da{\dclr{black}}
\def\Da{\da}
\def\Db{\db}
\def\Ua{\ua}
\def\Ub{\ub}
\def\Dc{\dc}
\def\Uc{\uc}
\def\Dd{\dd}
\def\Ud{\ud}
\begin{tikzpicture}[scale=0.325, rotate=-90]
\abcabcabc
\end{tikzpicture}
}

\noindent\rotatebox{90}{$[12{-}12]$:} 
\input{12-12fin}
\parbox[b]{1cm}{\def\uclr#1{ ++(1, 0.29) node [thin,circle,draw=black, fill=#1, thin, minimum size=6.5, inner sep=0] {}  ++(0,-0.29) }
\def\dclr#1{ ++(1,-0.29) node [thin,circle,draw=black, fill=#1, thin, minimum size=6.5, inner sep=0] {}  ++(0, 0.29) }
\def\ua{\uclr{white!90!black}}
\def\da{\dclr{white!90!black}}
\def\ub{\uclr{white!90!black}}
\def\db{\dclr{white!90!black}}
\def\uc{\uclr{white!90!black}}
\def\dc{\dclr{white!90!black}}
\def\ud{\uclr{white!90!black}}
\def\dd{\dclr{white!90!black}}
\def\ue{\uclr{white!90!black}}
\def\de{\dclr{white!90!black}}
\def\ua{\uclr{black}}
\def\da{\dclr{black}}
\def\Da{\da}
\def\Db{\db}
\def\Ua{\ua}
\def\Ub{\ub}
\def\Dc{\dc}
\def\Uc{\uc}
\def\Dd{\dd}
\def\Ud{\ud}
\begin{tikzpicture}[scale=0.325, rotate=-90]
\abcabcabc
\end{tikzpicture}
\\ \def\uclr#1{ ++(1, 0.29) node [thin,circle,draw=black, fill=#1, thin, minimum size=6.5, inner sep=0] {}  ++(0,-0.29) }
\def\dclr#1{ ++(1,-0.29) node [thin,circle,draw=black, fill=#1, thin, minimum size=6.5, inner sep=0] {}  ++(0, 0.29) }
\def\ua{\uclr{white!90!black}}
\def\da{\dclr{white!90!black}}
\def\ub{\uclr{white!90!black}}
\def\db{\dclr{white!90!black}}
\def\uc{\uclr{white!90!black}}
\def\dc{\dclr{white!90!black}}
\def\ud{\uclr{white!90!black}}
\def\dd{\dclr{white!90!black}}
\def\ue{\uclr{white!90!black}}
\def\de{\dclr{white!90!black}}
\def\ub{\uclr{black}}
\def\db{\dclr{black}}
\def\Da{\da}
\def\Db{\db}
\def\Ua{\ua}
\def\Ub{\ub}
\def\Dc{\dc}
\def\Uc{\uc}
\begin{tikzpicture}[scale=0.325, rotate=-90]
\abcabcabc
\end{tikzpicture}
}
   
\noindent\rotatebox{90}{\noindent$[21{-}12]$ (I):}
\input{21-12fin}
\parbox[b]{1cm}{\def\uclr#1{ ++(1, 0.29) node [thin,circle,draw=black, fill=#1, thin, minimum size=6.5, inner sep=0] {}  ++(0,-0.29) }
\def\dclr#1{ ++(1,-0.29) node [thin,circle,draw=black, fill=#1, thin, minimum size=6.5, inner sep=0] {}  ++(0, 0.29) }
\def\ua{\uclr{white!90!black}}
\def\da{\dclr{white!90!black}}
\def\ub{\uclr{white!90!black}}
\def\db{\dclr{white!90!black}}
\def\uc{\uclr{white!90!black}}
\def\dc{\dclr{white!90!black}}
\def\ud{\uclr{white!90!black}}
\def\dd{\dclr{white!90!black}}
\def\ue{\uclr{white!90!black}}
\def\de{\dclr{white!90!black}}
\def\ua{\uclr{black}}
\def\da{\dclr{black}}
\def\Da{\da}
\def\Db{\db}
\def\Ua{\ua}
\def\Ub{\ub}
\def\Dc{\dc}
\def\Uc{\uc}
\def\Dd{\dd}
\def\Ud{\ud}
\begin{tikzpicture}[scale=0.325, rotate=-90]
\abcabcabc
\end{tikzpicture}
\\ \def\uclr#1{ ++(1, 0.29) node [thin,circle,draw=black, fill=#1, thin, minimum size=6.5, inner sep=0] {}  ++(0,-0.29) }
\def\dclr#1{ ++(1,-0.29) node [thin,circle,draw=black, fill=#1, thin, minimum size=6.5, inner sep=0] {}  ++(0, 0.29) }
\def\ua{\uclr{white!90!black}}
\def\da{\dclr{white!90!black}}
\def\ub{\uclr{white!90!black}}
\def\db{\dclr{white!90!black}}
\def\uc{\uclr{white!90!black}}
\def\dc{\dclr{white!90!black}}
\def\ud{\uclr{white!90!black}}
\def\dd{\dclr{white!90!black}}
\def\ue{\uclr{white!90!black}}
\def\de{\dclr{white!90!black}}
\def\ub{\uclr{black}}
\def\db{\dclr{black}}
\def\Da{\da}
\def\Db{\db}
\def\Ua{\ua}
\def\Ub{\ub}
\def\Dc{\dc}
\def\Uc{\uc}
\begin{tikzpicture}[scale=0.325, rotate=-90]
\abcabcabc
\end{tikzpicture}
}

\noindent\rotatebox{90}{\noindent$[21{-}12]$ (II):}
\input{21-12fin9}
\parbox[b]{1cm}{\def\uclr#1{ ++(1, 0.29) node [thin,circle,draw=black, fill=#1, thin, minimum size=6.5, inner sep=0] {}  ++(0,-0.29) }
\def\dclr#1{ ++(1,-0.29) node [thin,circle,draw=black, fill=#1, thin, minimum size=6.5, inner sep=0] {}  ++(0, 0.29) }
\def\ua{\uclr{white!90!black}}
\def\da{\dclr{white!90!black}}
\def\ub{\uclr{white!90!black}}
\def\db{\dclr{white!90!black}}
\def\uc{\uclr{white!90!black}}
\def\dc{\dclr{white!90!black}}
\def\ud{\uclr{white!90!black}}
\def\dd{\dclr{white!90!black}}
\def\ue{\uclr{white!90!black}}
\def\de{\dclr{white!90!black}}
\def\ua{\uclr{black}}
\def\da{\dclr{black}}
\def\Da{\da}
\def\Db{\db}
\def\Ua{\ua}
\def\Ub{\ub}
\def\Dc{\dc}
\def\Uc{\uc}
\def\Dd{\dd}
\def\Ud{\ud}
\begin{tikzpicture}[scale=0.325, rotate=-90]
\abcabcabc
\end{tikzpicture}
}

\section{The colorings: 6 infinite classes of arrays}   

\noindent
\rotatebox{90}{$[12{-}111{-}...{-}111{-}12]$:}
\rotatebox{90}{$k=3,5,7,...$:}
\def\ut{ ++(1, 0.29) node  {\raisebox{-0.25em}[0.2em][0em]{\makebox[0mm][c]{$\cdots$}}}  ++(0,-0.29) }%
\def\dt{ ++(1,-0.29) node  {\raisebox{-0.25em}[0.2em][0em]{\makebox[0mm][c]{$\cdots$}}}  ++(0, 0.29) }%
\def\ua{\ucolor{yellow!20!white}{$0$}}%
\def\da{\dcolor{yellow!20!white}{$0$}}%
\def\ub{\ucolor{yellow!60!white}{$1$}}%
\def\db{\dcolor{yellow!60!white}{$1$}}%
\def\0{++(1,0)}%
\def\uc{\ucolor{orange!65!white}{$2$}}%
\def\dc{\dcolor{orange!65!white}{$2$}}%
\def\ud{\ucolor{red!85!brown!70!white}{$3$}}%
\def\dd{\dcolor{red!85!brown!70!white}{$3$}}%
\def\ue{\ucolor{brown!80!black}{$\ $}}%
\def\de{\dcolor{brown!80!black}{$\ $}}%
\def\uf{\ucolor{black}{$\ $}}%
\def\df{\dcolor{black}{$\ $}}%
\begin{tikzpicture}[scale=0.65, rotate=-90]
\draw (0, 5.192) \dc\ub\da\ub\dc\dt\ue\df\uf\de\0 \uc\db\ua\db\uc;
\draw (0, 3.464) \uc\db\ua\db\uc\0 \de\uf\df\ue\dt\dc\ub\da\ub\dc;
\draw (0, 1.732) \dc\ub\da\ub\dc\dt\ue\df\uf\de\0 \uc\db\ua\db\uc;
\draw (0,-0.000) \uc\db\ua\db\uc\0 \de\uf\df\ue\dt\dc\ub\da\ub\dc;
\draw (0,-1.732) \dc\ub\da\ub\dc\dt\ue\df\uf\de\0 \uc\db\ua\db\uc;
\draw (0,-3.464) \uc\db\ua\db\uc\0 \de\uf\df\ue\dt\dc\ub\da\ub\dc;
\draw (0,-5.192) \dc\ub\da\ub\dc\dt\ue\df\uf\de\0 \uc\db\ua\db\uc;
\draw (0,-6.928) \uc\db\ua\db\uc\0 \de\uf\df\ue\dt\dc\ub\da\ub\dc;
\draw (0,-8.660) \dc\ub\da\ub\dc\dt\ue\df\uf\de\0 \uc\db\ua\db\uc;
\end{tikzpicture}\ 
\rotatebox{90}{$k=4,6,8,...$:}
\def\ut{ ++(1, 0.29) node  {\raisebox{-0.25em}[0.2em][0em]{\makebox[0mm][c]{$\cdots$}}}  ++(0,-0.29) }%
\def\dt{ ++(1,-0.29) node  {\raisebox{-0.25em}[0.2em][0em]{\makebox[0mm][c]{$\cdots$}}}  ++(0, 0.29) }%
\def\uo{\ucolor{white}{$\ $}}%
\def\do{\dcolor{white}{$\ $}}%
\def\ux{\ucolor{white}{$$}}%
\def\dx{\dcolor{white}{$$}}%
\def\ua{\ucolor{yellow!20!white}{$0$}}%
\def\da{\dcolor{yellow!20!white}{$0$}}%
\def\ub{\ucolor{yellow!60!white}{$1$}}%
\def\db{\dcolor{yellow!60!white}{$1$}}%
\def\0{++(1,0)}%
\def\uc{\ucolor{orange!65!white}{$2$}}%
\def\dc{\dcolor{orange!65!white}{$2$}}%
\def\ud{\ucolor{red!85!brown!70!white}{$3$}}%
\def\dd{\dcolor{red!85!brown!70!white}{$3$}}%
\def\ue{\ucolor{brown!80!black}{$\ $}}%
\def\de{\dcolor{brown!80!black}{$\ $}}%
\def\uf{\ucolor{black}{$\ $}}%
\def\df{\dcolor{black}{$\ $}}%
\begin{tikzpicture}[scale=0.65, rotate=-90]
\draw (0, 5.192) \dc\ub\da\ub\dc\dt\de\uf\df\ue\0 \uc\db\ua\db\uc;
\draw (0, 3.464) \uc\db\ua\db\uc\0 \ue\df\uf\de\dt\dc\ub\da\ub\dc;
\draw (0, 1.732) \dc\ub\da\ub\dc\dt\de\uf\df\ue\0 \uc\db\ua\db\uc;
\draw (0,-0.000) \uc\db\ua\db\uc\0 \ue\df\uf\de\dt\dc\ub\da\ub\dc;
\draw (0,-1.732) \dc\ub\da\ub\dc\dt\de\uf\df\ue\0 \uc\db\ua\db\uc;
\draw (0,-3.464) \uc\db\ua\db\uc\0 \ue\df\uf\de\dt\dc\ub\da\ub\dc;
\draw (0,-5.192) \dc\ub\da\ub\dc\dt\de\uf\df\ue\0 \uc\db\ua\db\uc;
\draw (0,-6.928) \uc\db\ua\db\uc\0 \ue\df\uf\de\dt\dc\ub\da\ub\dc;
\draw (0,-8.660) \dc\ub\da\ub\dc\dt\de\uf\df\ue\0 \uc\db\ua\db\uc;
\end{tikzpicture}

\noindent
\rotatebox{90}{$[12{-}111{-}...{-}111{-}21]$:}
\rotatebox{90}{$k=3,5,7,...$:}
\def\ut{ ++(1, 0.29) node  {\raisebox{-0.25em}[0.2em][0em]{\makebox[0mm][c]{$\cdots$}}}  ++(0,-0.29) }%
\def\dt{ ++(1,-0.29) node  {\raisebox{-0.25em}[0.2em][0em]{\makebox[0mm][c]{$\cdots$}}}  ++(0, 0.29) }%
\def\ua{\ucolor{yellow!20!white}{$0$}}%
\def\da{\dcolor{yellow!20!white}{$0$}}%
\def\ub{\ucolor{yellow!60!white}{$1$}}%
\def\db{\dcolor{yellow!60!white}{$1$}}%
\def\0{++(1,0)}%
\def\uc{\ucolor{orange!65!white}{$2$}}%
\def\dc{\dcolor{orange!65!white}{$2$}}%
\def\ud{\ucolor{red!85!brown!70!white}{$3$}}%
\def\dd{\dcolor{red!85!brown!70!white}{$3$}}%
\def\ue{\ucolor{brown!80!black}{$\ $}}%
\def\de{\dcolor{brown!80!black}{$\ $}}%
\def\uf{\ucolor{black}{$\ $}}%
\def\df{\dcolor{black}{$\ $}}%
\begin{tikzpicture}[scale=0.65, rotate=-90]
\draw (0, 5.192) \dc\ub\da\ub\dc\dt\ue\df\ue\dt\dc\ub\da\ub\dc;
\draw (0, 3.464) \uc\db\ua\db\uc\0 \de\uf\de\0 \uc\db\ua\db\uc;
\draw (0, 1.732) \dc\ub\da\ub\dc\dt\ue\df\ue\dt\dc\ub\da\ub\dc;
\draw (0,-0.000) \uc\db\ua\db\uc\0 \de\uf\de\0 \uc\db\ua\db\uc;
\draw (0,-1.732) \dc\ub\da\ub\dc\dt\ue\df\ue\dt\dc\ub\da\ub\dc;
\draw (0,-3.464) \uc\db\ua\db\uc\0 \de\uf\de\0 \uc\db\ua\db\uc;
\draw (0,-5.192) \dc\ub\da\ub\dc\dt\ue\df\ue\dt\dc\ub\da\ub\dc;
\draw (0,-6.928) \uc\db\ua\db\uc\0 \de\uf\de\0 \uc\db\ua\db\uc;
\draw (0,-8.660) \dc\ub\da\ub\dc\dt\ue\df\ue\dt\dc\ub\da\ub\dc;
\end{tikzpicture}\ 
\rotatebox{90}{$k=4,6,8,...$:}
\def\ut{ ++(1, 0.29) node  {\raisebox{-0.25em}[0.2em][0em]{\makebox[0mm][c]{$\cdots$}}}  ++(0,-0.29) }%
\def\dt{ ++(1,-0.29) node  {\raisebox{-0.25em}[0.2em][0em]{\makebox[0mm][c]{$\cdots$}}}  ++(0, 0.29) }%
\def\ua{\ucolor{yellow!20!white}{$0$}}%
\def\da{\dcolor{yellow!20!white}{$0$}}%
\def\ub{\ucolor{yellow!60!white}{$1$}}%
\def\db{\dcolor{yellow!60!white}{$1$}}%
\def\0{++(1,0)}%
\def\uc{\ucolor{orange!65!white}{$2$}}%
\def\dc{\dcolor{orange!65!white}{$2$}}%
\def\ud{\ucolor{red!85!brown!70!white}{$3$}}%
\def\dd{\dcolor{red!85!brown!70!white}{$3$}}%
\def\ue{\ucolor{brown!80!black}{$\ $}}%
\def\de{\dcolor{brown!80!black}{$\ $}}%
\def\uf{\ucolor{black}{$\ $}}%
\def\df{\dcolor{black}{$\ $}}%
\begin{tikzpicture}[scale=0.65, rotate=-90]
\draw (0, 5.192) \dc\ub\da\ub\dc\dt\de\uf\de\dt\dc\ub\da\ub\dc;
\draw (0, 3.464) \uc\db\ua\db\uc\0 \ue\df\ue\0 \uc\db\ua\db\uc;
\draw (0, 1.732) \dc\ub\da\ub\dc\dt\de\uf\de\dt\dc\ub\da\ub\dc;
\draw (0,-0.000) \uc\db\ua\db\uc\0 \ue\df\ue\0 \uc\db\ua\db\uc;
\draw (0,-1.732) \dc\ub\da\ub\dc\dt\de\uf\de\dt\dc\ub\da\ub\dc;
\draw (0,-3.464) \uc\db\ua\db\uc\0 \ue\df\ue\0 \uc\db\ua\db\uc;
\draw (0,-5.192) \dc\ub\da\ub\dc\dt\de\uf\de\dt\dc\ub\da\ub\dc;
\draw (0,-6.928) \uc\db\ua\db\uc\0 \ue\df\ue\0 \uc\db\ua\db\uc;
\draw (0,-8.660) \dc\ub\da\ub\dc\dt\de\uf\de\dt\dc\ub\da\ub\dc;
\end{tikzpicture}

\noindent
\rotatebox{90}{$[12{-}111{-}21]$ (I):}
\input{12-111-21fin}
\parbox[b]{1cm}{\def\uclr#1{ ++(1, 0.29) node [thin,circle,draw=black, fill=#1, thin, minimum size=6.5, inner sep=0] {}  ++(0,-0.29) }
\def\dclr#1{ ++(1,-0.29) node [thin,circle,draw=black, fill=#1, thin, minimum size=6.5, inner sep=0] {}  ++(0, 0.29) }
\def\ua{\uclr{white!90!black}}
\def\da{\dclr{white!90!black}}
\def\ub{\uclr{white!90!black}}
\def\db{\dclr{white!90!black}}
\def\uc{\uclr{white!90!black}}
\def\dc{\dclr{white!90!black}}
\def\ud{\uclr{white!90!black}}
\def\dd{\dclr{white!90!black}}
\def\ue{\uclr{white!90!black}}
\def\de{\dclr{white!90!black}}
\def\ua{\uclr{black}}
\def\da{\dclr{black}}
\def\Da{\da}
\def\Db{\db}
\def\Ua{\ua}
\def\Ub{\ub}
\def\Dc{\dc}
\def\Uc{\uc}
\def\Dd{\dd}
\def\Ud{\ud}
\begin{tikzpicture}[scale=0.325, rotate=-90]
\abcabcabc
\end{tikzpicture}
}

\noindent
\rotatebox{90}{$[12{-}111{-}21]$ (II):}
\input{12-111-21fin0}
\parbox[b]{1cm}{\def\uclr#1{ ++(1, 0.29) node [thin,circle,draw=black, fill=#1, thin, minimum size=6.5, inner sep=0] {}  ++(0,-0.29) }
\def\dclr#1{ ++(1,-0.29) node [thin,circle,draw=black, fill=#1, thin, minimum size=6.5, inner sep=0] {}  ++(0, 0.29) }
\def\ua{\uclr{white!90!black}}
\def\da{\dclr{white!90!black}}
\def\ub{\uclr{white!90!black}}
\def\db{\dclr{white!90!black}}
\def\uc{\uclr{white!90!black}}
\def\dc{\dclr{white!90!black}}
\def\ud{\uclr{white!90!black}}
\def\dd{\dclr{white!90!black}}
\def\ue{\uclr{white!90!black}}
\def\de{\dclr{white!90!black}}
\def\ua{\uclr{black}}
\def\da{\dclr{black}}
\def\Da{\da}
\def\Db{\db}
\def\Ua{\ua}
\def\Ub{\ub}
\def\Dc{\dc}
\def\Uc{\uc}
\def\Dd{\dd}
\def\Ud{\ud}
\begin{tikzpicture}[scale=0.325, rotate=-90]
\abcabcabc
\end{tikzpicture}
}

\noindent
\rotatebox{90}{$[21{-}111{-}...{-}111{-}12]$:}
\rotatebox{90}{$k=3,5,7,...$:}
\def\ut{ ++(1, 0.29) node  {\raisebox{-0.25em}[0.2em][0em]{\makebox[0mm][c]{$\cdots$}}}  ++(0,-0.29) }%
\def\dt{ ++(1,-0.29) node  {\raisebox{-0.25em}[0.2em][0em]{\makebox[0mm][c]{$\cdots$}}}  ++(0, 0.29) }%
\def\uo{\ucolor{white}{$\ $}}%
\def\do{\dcolor{white}{$\ $}}%
\def\ux{\ucolor{white}{$$}}%
\def\dx{\dcolor{white}{$$}}%
\def\ua{\ucolor{yellow!20!white}{$0$}}%
\def\da{\dcolor{yellow!20!white}{$0$}}%
\def\ub{\ucolor{yellow!60!white}{$1$}}%
\def\db{\dcolor{yellow!60!white}{$1$}}%
\def\0{++(1,0)}%
\def\uc{\ucolor{orange!65!white}{$2$}}%
\def\dc{\dcolor{orange!65!white}{$2$}}%
\def\ud{\ucolor{red!85!brown!70!white}{$3$}}%
\def\dd{\dcolor{red!85!brown!70!white}{$3$}}%
\def\ue{\ucolor{brown!80!black}{$\ $}}%
\def\de{\dcolor{brown!80!black}{$\ $}}%
\def\uf{\ucolor{black}{$\ $}}%
\def\df{\dcolor{black}{$\ $}}%
\begin{tikzpicture}[scale=0.65, rotate=-90]
\draw (0, 5.192) \db\ua\da\ub\dc\dt\de\uf\df\ue\dt\dc\db\ua\da\ub;
\draw (0, 3.464) \ub\da\ua\db\uc\0 \ue\df\uf\de\0 \uc\ub\da\ua\db;
\draw (0, 1.732) \db\ua\da\ub\dc\dt\de\uf\df\ue\dt\dc\db\ua\da\ub;
\draw (0,-0.000) \ub\da\ua\db\uc\0 \ue\df\uf\de\0 \uc\ub\da\ua\db;
\draw (0,-1.732) \db\ua\da\ub\dc\dt\de\uf\df\ue\dt\dc\db\ua\da\ub;
\draw (0,-3.464) \ub\da\ua\db\uc\0 \ue\df\uf\de\0 \uc\ub\da\ua\db;
\draw (0,-5.192) \db\ua\da\ub\dc\dt\de\uf\df\ue\dt\dc\db\ua\da\ub;
\draw (0,-6.928) \ub\da\ua\db\uc\0 \ue\df\uf\de\0 \uc\ub\da\ua\db;
\draw (0,-8.660) \db\ua\da\ub\dc\dt\de\uf\df\ue\dt\dc\db\ua\da\ub;
\end{tikzpicture}\ 
\rotatebox{90}{$k=4,6,8,...$:}
\def\ut{ ++(1, 0.29) node  {\raisebox{-0.25em}[0.2em][0em]{\makebox[0mm][c]{$\cdots$}}}  ++(0,-0.29) }%
\def\dt{ ++(1,-0.29) node  {\raisebox{-0.25em}[0.2em][0em]{\makebox[0mm][c]{$\cdots$}}}  ++(0, 0.29) }%
\def\uo{\ucolor{white}{$\ $}}%
\def\do{\dcolor{white}{$\ $}}%
\def\ux{\ucolor{white}{$$}}%
\def\dx{\dcolor{white}{$$}}%
\def\ua{\ucolor{yellow!20!white}{$0$}}%
\def\da{\dcolor{yellow!20!white}{$0$}}%
\def\ub{\ucolor{yellow!60!white}{$1$}}%
\def\db{\dcolor{yellow!60!white}{$1$}}%
\def\0{++(1,0)}%
\def\uc{\ucolor{orange!65!white}{$2$}}%
\def\dc{\dcolor{orange!65!white}{$2$}}%
\def\ud{\ucolor{red!85!brown!70!white}{$3$}}%
\def\dd{\dcolor{red!85!brown!70!white}{$3$}}%
\def\ue{\ucolor{brown!80!black}{$\ $}}%
\def\de{\dcolor{brown!80!black}{$\ $}}%
\def\uf{\ucolor{black}{$\ $}}%
\def\df{\dcolor{black}{$\ $}}%
\begin{tikzpicture}[scale=0.65, rotate=-90]
\draw (0, 5.192) \db\ua\da\ub\dc\dt\ue\df\uf\de\dt\dc\db\ua\da\ub;
\draw (0, 3.464) \ub\da\ua\db\uc\0 \de\uf\df\ue\0 \uc\ub\da\ua\db;
\draw (0, 1.732) \db\ua\da\ub\dc\dt\ue\df\uf\de\dt\dc\db\ua\da\ub;
\draw (0,-0.000) \ub\da\ua\db\uc\0 \de\uf\df\ue\0 \uc\ub\da\ua\db;
\draw (0,-1.732) \db\ua\da\ub\dc\dt\ue\df\uf\de\dt\dc\db\ua\da\ub;
\draw (0,-3.464) \ub\da\ua\db\uc\0 \de\uf\df\ue\0 \uc\ub\da\ua\db;
\draw (0,-5.192) \db\ua\da\ub\dc\dt\ue\df\uf\de\dt\dc\db\ua\da\ub;
\draw (0,-6.928) \ub\da\ua\db\uc\0 \de\uf\df\ue\0 \uc\ub\da\ua\db;
\draw (0,-8.660) \db\ua\da\ub\dc\dt\ue\df\uf\de\dt\dc\db\ua\da\ub;
\end{tikzpicture}

\noindent
\rotatebox{90}{$[12{-}201{-}102{-}201{-}...{-}201{-}12]$:}
\rotatebox{90}{$k=3,7,11,...$:} 
\input{12-201--12fin1}
\rotatebox{90}{$k=5,9,13,...$:} 
\input{12-201--12fin2}

\noindent
\rotatebox{90}{$[12{-}201{-}12]$ (II):}
\input{12-201-12fin2}
\parbox[b]{1cm}{\def\uclr#1{ ++(1, 0.29) node [thin,circle,draw=black, fill=#1, thin, minimum size=6.5, inner sep=0] {}  ++(0,-0.29) }
\def\dclr#1{ ++(1,-0.29) node [thin,circle,draw=black, fill=#1, thin, minimum size=6.5, inner sep=0] {}  ++(0, 0.29) }
\def\ua{\uclr{white!90!black}}
\def\da{\dclr{white!90!black}}
\def\ub{\uclr{white!90!black}}
\def\db{\dclr{white!90!black}}
\def\uc{\uclr{white!90!black}}
\def\dc{\dclr{white!90!black}}
\def\ud{\uclr{white!90!black}}
\def\dd{\dclr{white!90!black}}
\def\ue{\uclr{white!90!black}}
\def\de{\dclr{white!90!black}}
\def\ua{\uclr{black}}
\def\da{\dclr{black}}
\def\Da{\da}
\def\Db{\db}
\def\Ua{\ua}
\def\Ub{\ub}
\def\Dc{\dc}
\def\Uc{\uc}
\def\Dd{\dd}
\def\Ud{\ud}
\begin{tikzpicture}[scale=0.325, rotate=-90]
\abcabcabc
\end{tikzpicture}
\\ \def\uclr#1{ ++(1, 0.29) node [thin,circle,draw=black, fill=#1, thin, minimum size=6.5, inner sep=0] {}  ++(0,-0.29) }
\def\dclr#1{ ++(1,-0.29) node [thin,circle,draw=black, fill=#1, thin, minimum size=6.5, inner sep=0] {}  ++(0, 0.29) }
\def\ua{\uclr{white!90!black}}
\def\da{\dclr{white!90!black}}
\def\ub{\uclr{white!90!black}}
\def\db{\dclr{white!90!black}}
\def\uc{\uclr{white!90!black}}
\def\dc{\dclr{white!90!black}}
\def\ud{\uclr{white!90!black}}
\def\dd{\dclr{white!90!black}}
\def\ue{\uclr{white!90!black}}
\def\de{\dclr{white!90!black}}
\def\uc{\uclr{black}}
\def\dc{\dclr{black}}
\begin{tikzpicture}[scale=0.325, rotate=-90]
\abcabcabc
\end{tikzpicture}
}

\noindent
\rotatebox{90}{$[12{-}201{-}102{-}201{-}...{-}102{-}21]$:}
\rotatebox{90}{$k=4,8,12,...$:} 
\input{12-201--21fin1}
\rotatebox{90}{$k=6,10,14,...$:} 
\input{12-201--21fin2}

\noindent
\rotatebox{90}{$[21{-}102{-}201{-}102{-}...{-}201{-}12]$:}
\rotatebox{90}{$k=4,8,12,...$:} 
\input{21-102--12fin2}
\rotatebox{90}{$k=6,10,14,...$:} 
\input{21-102--12fin1}
 
\section{The colorings: 10 ``sporadic'' arrays} 

\noindent
\rotatebox{90}{$[03{-}102{-}30]$:}
\input{03-102-30fin}
\parbox[b]{1cm}{\def\uclr#1{ ++(1, 0.29) node [thin,circle,draw=black, fill=#1, thin, minimum size=6.5, inner sep=0] {}  ++(0,-0.29) }
\def\dclr#1{ ++(1,-0.29) node [thin,circle,draw=black, fill=#1, thin, minimum size=6.5, inner sep=0] {}  ++(0, 0.29) }
\def\ua{\uclr{white!90!black}}
\def\da{\dclr{white!90!black}}
\def\ub{\uclr{white!90!black}}
\def\db{\dclr{white!90!black}}
\def\uc{\uclr{white!90!black}}
\def\dc{\dclr{white!90!black}}
\def\ud{\uclr{white!90!black}}
\def\dd{\dclr{white!90!black}}
\def\ue{\uclr{white!90!black}}
\def\de{\dclr{white!90!black}}
\def\ua{\uclr{black}}
\def\da{\dclr{black}}
\def\Da{\da}
\def\Db{\db}
\def\Ua{\ua}
\def\Ub{\ub}
\def\Dc{\dc}
\def\Uc{\uc}
\def\Dd{\dd}
\def\Ud{\ud}
\begin{tikzpicture}[scale=0.325, rotate=-90]
\abcabcabc
\end{tikzpicture}
\\ \def\uclr#1{ ++(1, 0.29) node [thin,circle,draw=black, fill=#1, thin, minimum size=6.5, inner sep=0] {}  ++(0,-0.29) }
\def\dclr#1{ ++(1,-0.29) node [thin,circle,draw=black, fill=#1, thin, minimum size=6.5, inner sep=0] {}  ++(0, 0.29) }
\def\ua{\uclr{white!90!black}}
\def\da{\dclr{white!90!black}}
\def\ub{\uclr{white!90!black}}
\def\db{\dclr{white!90!black}}
\def\uc{\uclr{white!90!black}}
\def\dc{\dclr{white!90!black}}
\def\ud{\uclr{white!90!black}}
\def\dd{\dclr{white!90!black}}
\def\ue{\uclr{white!90!black}}
\def\de{\dclr{white!90!black}}
\def\uc{\uclr{black}}
\def\dc{\dclr{black}}
\begin{tikzpicture}[scale=0.325, rotate=-90]
\abcabcabc
\end{tikzpicture}
}

\noindent
\rotatebox{90}{$[03{-}111{-}12]$ (I):}
\input{03-111-12fin0}
\parbox[b]{1cm}{\def\uclr#1{ ++(1, 0.29) node [thin,circle,draw=black, fill=#1, thin, minimum size=6.5, inner sep=0] {}  ++(0,-0.29) }
\def\dclr#1{ ++(1,-0.29) node [thin,circle,draw=black, fill=#1, thin, minimum size=6.5, inner sep=0] {}  ++(0, 0.29) }
\def\ua{\uclr{white!90!black}}
\def\da{\dclr{white!90!black}}
\def\ub{\uclr{white!90!black}}
\def\db{\dclr{white!90!black}}
\def\uc{\uclr{white!90!black}}
\def\dc{\dclr{white!90!black}}
\def\ud{\uclr{white!90!black}}
\def\dd{\dclr{white!90!black}}
\def\ue{\uclr{white!90!black}}
\def\de{\dclr{white!90!black}}
\def\ua{\uclr{black}}
\def\da{\dclr{black}}
\def\Da{\da}
\def\Db{\db}
\def\Ua{\ua}
\def\Ub{\ub}
\def\Dc{\dc}
\def\Uc{\uc}
\def\Dd{\dd}
\def\Ud{\ud}
\begin{tikzpicture}[scale=0.325, rotate=-90]
\abcabcabc
\end{tikzpicture}
\\ \def\uclr#1{ ++(1, 0.29) node [thin,circle,draw=black, fill=#1, thin, minimum size=6.5, inner sep=0] {}  ++(0,-0.29) }
\def\dclr#1{ ++(1,-0.29) node [thin,circle,draw=black, fill=#1, thin, minimum size=6.5, inner sep=0] {}  ++(0, 0.29) }
\def\ua{\uclr{white!90!black}}
\def\da{\dclr{white!90!black}}
\def\ub{\uclr{white!90!black}}
\def\db{\dclr{white!90!black}}
\def\uc{\uclr{white!90!black}}
\def\dc{\dclr{white!90!black}}
\def\ud{\uclr{white!90!black}}
\def\dd{\dclr{white!90!black}}
\def\ue{\uclr{white!90!black}}
\def\de{\dclr{white!90!black}}
\def\uc{\uclr{black}}
\def\dc{\dclr{black}}
\begin{tikzpicture}[scale=0.325, rotate=-90]
\abcabcabc
\end{tikzpicture}
}

\noindent
\rotatebox{90}{$[03{-}111{-}12]$ (II):}
\input{03-111-12fin}
\parbox[b]{1cm}{\def\uclr#1{ ++(1, 0.29) node [thin,circle,draw=black, fill=#1, thin, minimum size=6.5, inner sep=0] {}  ++(0,-0.29) }
\def\dclr#1{ ++(1,-0.29) node [thin,circle,draw=black, fill=#1, thin, minimum size=6.5, inner sep=0] {}  ++(0, 0.29) }
\def\ua{\uclr{white!90!black}}
\def\da{\dclr{white!90!black}}
\def\ub{\uclr{white!90!black}}
\def\db{\dclr{white!90!black}}
\def\uc{\uclr{white!90!black}}
\def\dc{\dclr{white!90!black}}
\def\ud{\uclr{white!90!black}}
\def\dd{\dclr{white!90!black}}
\def\ue{\uclr{white!90!black}}
\def\de{\dclr{white!90!black}}
\def\ua{\uclr{black}}
\def\da{\dclr{black}}
\def\Da{\da}
\def\Db{\db}
\def\Ua{\ua}
\def\Ub{\ub}
\def\Dc{\dc}
\def\Uc{\uc}
\def\Dd{\dd}
\def\Ud{\ud}
\begin{tikzpicture}[scale=0.325, rotate=-90]
\abcabcabc
\end{tikzpicture}
\\ \def\uclr#1{ ++(1, 0.29) node [thin,circle,draw=black, fill=#1, thin, minimum size=6.5, inner sep=0] {}  ++(0,-0.29) }
\def\dclr#1{ ++(1,-0.29) node [thin,circle,draw=black, fill=#1, thin, minimum size=6.5, inner sep=0] {}  ++(0, 0.29) }
\def\ua{\uclr{white!90!black}}
\def\da{\dclr{white!90!black}}
\def\ub{\uclr{white!90!black}}
\def\db{\dclr{white!90!black}}
\def\uc{\uclr{white!90!black}}
\def\dc{\dclr{white!90!black}}
\def\ud{\uclr{white!90!black}}
\def\dd{\dclr{white!90!black}}
\def\ue{\uclr{white!90!black}}
\def\de{\dclr{white!90!black}}
\def\uc{\uclr{black}}
\def\dc{\dclr{black}}
\begin{tikzpicture}[scale=0.325, rotate=-90]
\abcabcabc
\end{tikzpicture}
}

\noindent
\rotatebox{90}{$[12{-}102{-}12]$:}
\input{12-102-12fin}
\parbox[b]{1cm}{\def\uclr#1{ ++(1, 0.29) node [thin,circle,draw=black, fill=#1, thin, minimum size=6.5, inner sep=0] {}  ++(0,-0.29) }
\def\dclr#1{ ++(1,-0.29) node [thin,circle,draw=black, fill=#1, thin, minimum size=6.5, inner sep=0] {}  ++(0, 0.29) }
\def\ua{\uclr{white!90!black}}
\def\da{\dclr{white!90!black}}
\def\ub{\uclr{white!90!black}}
\def\db{\dclr{white!90!black}}
\def\uc{\uclr{white!90!black}}
\def\dc{\dclr{white!90!black}}
\def\ud{\uclr{white!90!black}}
\def\dd{\dclr{white!90!black}}
\def\ue{\uclr{white!90!black}}
\def\de{\dclr{white!90!black}}
\def\ua{\uclr{black}}
\def\da{\dclr{black}}
\def\Da{\da}
\def\Db{\db}
\def\Ua{\ua}
\def\Ub{\ub}
\def\Dc{\dc}
\def\Uc{\uc}
\def\Dd{\dd}
\def\Ud{\ud}
\begin{tikzpicture}[scale=0.325, rotate=-90]
\abcabcabc
\end{tikzpicture}
\\ \def\uclr#1{ ++(1, 0.29) node [thin,circle,draw=black, fill=#1, thin, minimum size=6.5, inner sep=0] {}  ++(0,-0.29) }
\def\dclr#1{ ++(1,-0.29) node [thin,circle,draw=black, fill=#1, thin, minimum size=6.5, inner sep=0] {}  ++(0, 0.29) }
\def\ua{\uclr{white!90!black}}
\def\da{\dclr{white!90!black}}
\def\ub{\uclr{white!90!black}}
\def\db{\dclr{white!90!black}}
\def\uc{\uclr{white!90!black}}
\def\dc{\dclr{white!90!black}}
\def\ud{\uclr{white!90!black}}
\def\dd{\dclr{white!90!black}}
\def\ue{\uclr{white!90!black}}
\def\de{\dclr{white!90!black}}
\def\uc{\uclr{black}}
\def\dc{\dclr{black}}
\begin{tikzpicture}[scale=0.325, rotate=-90]
\abcabcabc
\end{tikzpicture}
}

\noindent
\rotatebox{90}{$[21{-}102{-}12]$:}
\input{21-102-12fin}
\parbox[b]{1cm}{\def\uclr#1{ ++(1, 0.29) node [thin,circle,draw=black, fill=#1, thin, minimum size=6.5, inner sep=0] {}  ++(0,-0.29) }
\def\dclr#1{ ++(1,-0.29) node [thin,circle,draw=black, fill=#1, thin, minimum size=6.5, inner sep=0] {}  ++(0, 0.29) }
\def\ua{\uclr{white!90!black}}
\def\da{\dclr{white!90!black}}
\def\ub{\uclr{white!90!black}}
\def\db{\dclr{white!90!black}}
\def\uc{\uclr{white!90!black}}
\def\dc{\dclr{white!90!black}}
\def\ud{\uclr{white!90!black}}
\def\dd{\dclr{white!90!black}}
\def\ue{\uclr{white!90!black}}
\def\de{\dclr{white!90!black}}
\def\ua{\uclr{black}}
\def\da{\dclr{black}}
\def\Da{\da}
\def\Db{\db}
\def\Ua{\ua}
\def\Ub{\ub}
\def\Dc{\dc}
\def\Uc{\uc}
\def\Dd{\dd}
\def\Ud{\ud}
\begin{tikzpicture}[scale=0.325, rotate=-90]
\abcabcabc
\end{tikzpicture}
\\ \def\uclr#1{ ++(1, 0.29) node [thin,circle,draw=black, fill=#1, thin, minimum size=6.5, inner sep=0] {}  ++(0,-0.29) }
\def\dclr#1{ ++(1,-0.29) node [thin,circle,draw=black, fill=#1, thin, minimum size=6.5, inner sep=0] {}  ++(0, 0.29) }
\def\ua{\uclr{white!90!black}}
\def\da{\dclr{white!90!black}}
\def\ub{\uclr{white!90!black}}
\def\db{\dclr{white!90!black}}
\def\uc{\uclr{white!90!black}}
\def\dc{\dclr{white!90!black}}
\def\ud{\uclr{white!90!black}}
\def\dd{\dclr{white!90!black}}
\def\ue{\uclr{white!90!black}}
\def\de{\dclr{white!90!black}}
\def\uc{\uclr{black}}
\def\dc{\dclr{black}}
\begin{tikzpicture}[scale=0.325, rotate=-90]
\abcabcabc
\end{tikzpicture}
}

\noindent
\rotatebox{90}{$[03{-}102{-}102{-}30]$:}
\input{03-102-102-30fin}
\parbox[b]{1cm}{\def\uclr#1{ ++(1, 0.29) node [thin,circle,draw=black, fill=#1, thin, minimum size=6.5, inner sep=0] {}  ++(0,-0.29) }
\def\dclr#1{ ++(1,-0.29) node [thin,circle,draw=black, fill=#1, thin, minimum size=6.5, inner sep=0] {}  ++(0, 0.29) }
\def\ua{\uclr{white!90!black}}
\def\da{\dclr{white!90!black}}
\def\ub{\uclr{white!90!black}}
\def\db{\dclr{white!90!black}}
\def\uc{\uclr{white!90!black}}
\def\dc{\dclr{white!90!black}}
\def\ud{\uclr{white!90!black}}
\def\dd{\dclr{white!90!black}}
\def\ue{\uclr{white!90!black}}
\def\de{\dclr{white!90!black}}
\def\ua{\uclr{black}}
\def\da{\dclr{black}}
\def\Da{\da}
\def\Db{\db}
\def\Ua{\ua}
\def\Ub{\ub}
\def\Dc{\dc}
\def\Uc{\uc}
\def\Dd{\dd}
\def\Ud{\ud}
\begin{tikzpicture}[scale=0.325, rotate=-90]
\abcabcabc
\end{tikzpicture}
\\ \def\uclr#1{ ++(1, 0.29) node [thin,circle,draw=black, fill=#1, thin, minimum size=6.5, inner sep=0] {}  ++(0,-0.29) }
\def\dclr#1{ ++(1,-0.29) node [thin,circle,draw=black, fill=#1, thin, minimum size=6.5, inner sep=0] {}  ++(0, 0.29) }
\def\ua{\uclr{white!90!black}}
\def\da{\dclr{white!90!black}}
\def\ub{\uclr{white!90!black}}
\def\db{\dclr{white!90!black}}
\def\uc{\uclr{white!90!black}}
\def\dc{\dclr{white!90!black}}
\def\ud{\uclr{white!90!black}}
\def\dd{\dclr{white!90!black}}
\def\ue{\uclr{white!90!black}}
\def\de{\dclr{white!90!black}}
\def\ud{\uclr{black}}
\def\dd{\dclr{black}}
\begin{tikzpicture}[scale=0.325, rotate=-90]
\abcabcabc
\end{tikzpicture}
}

\noindent
\rotatebox{90}{$[03{-}102{-}201{-}30]$:}
\input{03-102-201-30fin}
\parbox[b]{1cm}{\def\uclr#1{ ++(1, 0.29) node [thin,circle,draw=black, fill=#1, thin, minimum size=6.5, inner sep=0] {}  ++(0,-0.29) }
\def\dclr#1{ ++(1,-0.29) node [thin,circle,draw=black, fill=#1, thin, minimum size=6.5, inner sep=0] {}  ++(0, 0.29) }
\def\ua{\uclr{white!90!black}}
\def\da{\dclr{white!90!black}}
\def\ub{\uclr{white!90!black}}
\def\db{\dclr{white!90!black}}
\def\uc{\uclr{white!90!black}}
\def\dc{\dclr{white!90!black}}
\def\ud{\uclr{white!90!black}}
\def\dd{\dclr{white!90!black}}
\def\ue{\uclr{white!90!black}}
\def\de{\dclr{white!90!black}}
\def\ua{\uclr{black}}
\def\da{\dclr{black}}
\def\Da{\da}
\def\Db{\db}
\def\Ua{\ua}
\def\Ub{\ub}
\def\Dc{\dc}
\def\Uc{\uc}
\def\Dd{\dd}
\def\Ud{\ud}
\begin{tikzpicture}[scale=0.325, rotate=-90]
\abcabcabc
\end{tikzpicture}
}

\noindent
\rotatebox{90}{$[03{-}111{-}111{-}30]$:}
\input{03-111-111-30fin}
\parbox[b]{1cm}{\def\uclr#1{ ++(1, 0.29) node [thin,circle,draw=black, fill=#1, thin, minimum size=6.5, inner sep=0] {}  ++(0,-0.29) }
\def\dclr#1{ ++(1,-0.29) node [thin,circle,draw=black, fill=#1, thin, minimum size=6.5, inner sep=0] {}  ++(0, 0.29) }
\def\ua{\uclr{white!90!black}}
\def\da{\dclr{white!90!black}}
\def\ub{\uclr{white!90!black}}
\def\db{\dclr{white!90!black}}
\def\uc{\uclr{white!90!black}}
\def\dc{\dclr{white!90!black}}
\def\ud{\uclr{white!90!black}}
\def\dd{\dclr{white!90!black}}
\def\ue{\uclr{white!90!black}}
\def\de{\dclr{white!90!black}}
\def\ua{\uclr{black}}
\def\da{\dclr{black}}
\def\Da{\da}
\def\Db{\db}
\def\Ua{\ua}
\def\Ub{\ub}
\def\Dc{\dc}
\def\Uc{\uc}
\def\Dd{\dd}
\def\Ud{\ud}
\begin{tikzpicture}[scale=0.325, rotate=-90]
\abcabcabc
\end{tikzpicture}
}

\noindent
\rotatebox{90}{$[12{-}102{-}111{-}21]$:}
\input{12-102-111-21fin}
\parbox[b]{1cm}{\def\uclr#1{ ++(1, 0.29) node [thin,circle,draw=black, fill=#1, thin, minimum size=6.5, inner sep=0] {}  ++(0,-0.29) }
\def\dclr#1{ ++(1,-0.29) node [thin,circle,draw=black, fill=#1, thin, minimum size=6.5, inner sep=0] {}  ++(0, 0.29) }
\def\ua{\uclr{white!90!black}}
\def\da{\dclr{white!90!black}}
\def\ub{\uclr{white!90!black}}
\def\db{\dclr{white!90!black}}
\def\uc{\uclr{white!90!black}}
\def\dc{\dclr{white!90!black}}
\def\ud{\uclr{white!90!black}}
\def\dd{\dclr{white!90!black}}
\def\ue{\uclr{white!90!black}}
\def\de{\dclr{white!90!black}}
\def\ua{\uclr{black}}
\def\da{\dclr{black}}
\def\Da{\da}
\def\Db{\db}
\def\Ua{\ua}
\def\Ub{\ub}
\def\Dc{\dc}
\def\Uc{\uc}
\def\Dd{\dd}
\def\Ud{\ud}
\begin{tikzpicture}[scale=0.325, rotate=-90]
\abcabcabc
\end{tikzpicture}
\\ \def\uclr#1{ ++(1, 0.29) node [thin,circle,draw=black, fill=#1, thin, minimum size=6.5, inner sep=0] {}  ++(0,-0.29) }
\def\dclr#1{ ++(1,-0.29) node [thin,circle,draw=black, fill=#1, thin, minimum size=6.5, inner sep=0] {}  ++(0, 0.29) }
\def\ua{\uclr{white!90!black}}
\def\da{\dclr{white!90!black}}
\def\ub{\uclr{white!90!black}}
\def\db{\dclr{white!90!black}}
\def\uc{\uclr{white!90!black}}
\def\dc{\dclr{white!90!black}}
\def\ud{\uclr{white!90!black}}
\def\dd{\dclr{white!90!black}}
\def\ue{\uclr{white!90!black}}
\def\de{\dclr{white!90!black}}
\def\ud{\uclr{black}}
\def\dd{\dclr{black}}
\begin{tikzpicture}[scale=0.325, rotate=-90]
\abcabcabc
\end{tikzpicture}
}

\noindent
\rotatebox{90}{$[03{-}102{-}102{-}201{-}30]$:}
\input{03-102-102-201-30fin}
\parbox[b]{1cm}{\def\uclr#1{ ++(1, 0.29) node [thin,circle,draw=black, fill=#1, thin, minimum size=6.5, inner sep=0] {}  ++(0,-0.29) }
\def\dclr#1{ ++(1,-0.29) node [thin,circle,draw=black, fill=#1, thin, minimum size=6.5, inner sep=0] {}  ++(0, 0.29) }
\def\ua{\uclr{white!90!black}}
\def\da{\dclr{white!90!black}}
\def\ub{\uclr{white!90!black}}
\def\db{\dclr{white!90!black}}
\def\uc{\uclr{white!90!black}}
\def\dc{\dclr{white!90!black}}
\def\ud{\uclr{white!90!black}}
\def\dd{\dclr{white!90!black}}
\def\ue{\uclr{white!90!black}}
\def\de{\dclr{white!90!black}}
\def\ua{\uclr{black}}
\def\da{\dclr{black}}
\def\Da{\da}
\def\Db{\db}
\def\Ua{\ua}
\def\Ub{\ub}
\def\Dc{\dc}
\def\Uc{\uc}
\def\Dd{\dd}
\def\Ud{\ud}
\begin{tikzpicture}[scale=0.325, rotate=-90]
\abcabcabc
\end{tikzpicture}
\\ \def\uclr#1{ ++(1, 0.29) node [thin,circle,draw=black, fill=#1, thin, minimum size=6.5, inner sep=0] {}  ++(0,-0.29) }
\def\dclr#1{ ++(1,-0.29) node [thin,circle,draw=black, fill=#1, thin, minimum size=6.5, inner sep=0] {}  ++(0, 0.29) }
\def\ua{\uclr{white!90!black}}
\def\da{\dclr{white!90!black}}
\def\ub{\uclr{white!90!black}}
\def\db{\dclr{white!90!black}}
\def\uc{\uclr{white!90!black}}
\def\dc{\dclr{white!90!black}}
\def\ud{\uclr{white!90!black}}
\def\dd{\dclr{white!90!black}}
\def\ue{\uclr{white!90!black}}
\def\de{\dclr{white!90!black}}
\def\ue{\uclr{black}}
\def\de{\dclr{black}}
\begin{tikzpicture}[scale=0.325, rotate=-90]
\abcabcabc
\end{tikzpicture}
}

\noindent
\rotatebox{90}{$[03{-}102{-}111{-}201{-}12]$:}
\input{03-102-111-201-12fin}
\parbox[b]{1cm}{\def\uclr#1{ ++(1, 0.29) node [thin,circle,draw=black, fill=#1, thin, minimum size=6.5, inner sep=0] {}  ++(0,-0.29) }
\def\dclr#1{ ++(1,-0.29) node [thin,circle,draw=black, fill=#1, thin, minimum size=6.5, inner sep=0] {}  ++(0, 0.29) }
\def\ua{\uclr{white!90!black}}
\def\da{\dclr{white!90!black}}
\def\ub{\uclr{white!90!black}}
\def\db{\dclr{white!90!black}}
\def\uc{\uclr{white!90!black}}
\def\dc{\dclr{white!90!black}}
\def\ud{\uclr{white!90!black}}
\def\dd{\dclr{white!90!black}}
\def\ue{\uclr{white!90!black}}
\def\de{\dclr{white!90!black}}
\def\ua{\uclr{black}}
\def\da{\dclr{black}}
\def\Da{\da}
\def\Db{\db}
\def\Ua{\ua}
\def\Ub{\ub}
\def\Dc{\dc}
\def\Uc{\uc}
\def\Dd{\dd}
\def\Ud{\ud}
\begin{tikzpicture}[scale=0.325, rotate=-90]
\abcabcabc
\end{tikzpicture}
\\ \def\uclr#1{ ++(1, 0.29) node [thin,circle,draw=black, fill=#1, thin, minimum size=6.5, inner sep=0] {}  ++(0,-0.29) }
\def\dclr#1{ ++(1,-0.29) node [thin,circle,draw=black, fill=#1, thin, minimum size=6.5, inner sep=0] {}  ++(0, 0.29) }
\def\ua{\uclr{white!90!black}}
\def\da{\dclr{white!90!black}}
\def\ub{\uclr{white!90!black}}
\def\db{\dclr{white!90!black}}
\def\uc{\uclr{white!90!black}}
\def\dc{\dclr{white!90!black}}
\def\ud{\uclr{white!90!black}}
\def\dd{\dclr{white!90!black}}
\def\ue{\uclr{white!90!black}}
\def\de{\dclr{white!90!black}}
\def\ue{\uclr{black}}
\def\de{\dclr{black}}
\begin{tikzpicture}[scale=0.325, rotate=-90]
\abcabcabc
\end{tikzpicture}
}


\providecommand\href[2]{#2} \providecommand\url[1]{\href{#1}{#1}}
  \def\DOI#1{{\small {DOI}:
  \href{http://dx.doi.org/#1}{#1}}}\def\DOIURL#1#2{{\small{DOI}:
  \href{http://dx.doi.org/#2}{#1}}}

\mbox{}{\let\thefootnote\relax\footnotetext{%
\begin{otherlanguage}{russian}%
Полностью регулярные коды в бесконечной гексагональной решетке

С.\,В. Августинович, Д.\,С. Кротов, А.\,Ю. Васильева

\emph{Аннотация}. Множество $C$ вершин простого графа называется полностью регулярным кодом, если для любых $i=0,1,2,...$ и $j=i-1,i,i+1$ все вершины на расстоянии $i$ от $C$ имеют одно и то же число $s_{ij}$ соседей на расстоянии $j$ от $C$. Приведена характеризация полностью регулярных кодов в графе бесконечной гексагональной решетки.

\emph{Ключевые слова}: полностью регулярный код, совершенная раскраска, регулярное разбиение, гексагональная решетка
\end{otherlanguage}
}
\end{document}